\documentclass[12pt,twoside]{article}
\pdfoutput=1

    \usepackage{graphicx}
    \usepackage{styleset}
    \usepackage{macros}
    \usepackage[colorinlistoftodos]{todonotes}
    \usepackage{enumitem}
    \let\subsubsection\subparagraph

    \title  {Protocorks and monopole Floer homology}

    \author{Roberto Ladu}

		\date{\vspace{-2ex}}

\begin{document}

\maketitle            

\begin{abstract}
We introduce and study a class of compact 4-manifolds with boundary that we call protocorks.  Any exotic pair of simply connected closed 4-manifolds is related by a protocork twist, moreover, any cork is supported by a protocork. We prove a theorem on the relative Seiberg-Witten invariants of a protocork before and after twisting and a splitting theorem 
on the Floer homology of protocork boundaries.
As a corollary we improve a theorem by Morgan and Szab\'{o} regarding the variation of Seiberg-Witten invariants with an upper bound which depends only on the topology of the data. 
Moreover, we generalize the result that only the reduced Floer homology of a cork boundary contributes to the variation of the Seiberg-Witten invariants  under a cork twist to more general cut and paste operations  where the pieces involved are $1$-connected and homeomorphic relative to the boundary.
\end{abstract}


\section{Introduction}

Corks  \cite{Akbulut91, AkbulutBook}  are a class of contractible 4-manifolds with boundary endowed with an involution of their boundaries that can be used to produce any  simply connected exotic pair  \cite{CHFS, Matveyev}.
In this paper,  inspired by \cite{CHFS, Matveyev, MorganSzabo99}, we define and  study a class of 4-manifolds with boundary that we call \emph{protocorks} which also relate any simply connected exotic pair of 4-manifolds  and which support any cork.
This informally means that the protocork determines the effect of the cork twist.
Consequently, protocorks may be used as a device to prove general theorems about \emph{all} corks, or directly to investigate exotica.

\paragraph{Main result.}
Let $X_0$ be a  closed 4-manifold decomposed as $X_0 = M\bigcup_Y N_0$ where $N_0$ is a protocork with boundary $Y$  and $M$ is a 4-manifold with $\partial M =  -Y$.
Performing a protocork twist we obtain $X_1 = M\bigcup_Y N_1$ where $N_1$ is another manifold with $\partial N_1 = Y$ which we call \emph{reflection} of $N_0$.
The difference between the Seiberg-Witten invariants of $X_0$ and $X_1$ is governed by a difference element $\Delta$, depending only on $(N_0,N_1)$, which belongs to the monopole Floer homology of $Y$, $ \HMfrom(Y)$ \cite{KM}, a module over $\ZU$.
Our main result regards  this element.

\begin{theorem}\label{thm:DeltaInHMred} Let $N_0$ be a protocork with boundary $Y$, let $N_1$ be its reflection and  let $\hat{1}\in \HMfrom(\SS^3)$ be the standard generator of $\HMfrom(\SS^3)$. Remove a ball from $N_i$, $i=0,1$ and regard the result as a cobordism $N_i \setminus \ball:\SS^3\to Y$. Choose $\mu_0$, an homology orientation for $N_0 \setminus \ball$. Then:
\begin{enumerate}
\item \label{thm:DeltaInHMred:item:Delta} there is an homology orientation $\mu_1$ on $N_1\setminus \ball$ such that, setting
\begin{equation}\label{def:relativeSWInvariantPcork}
		x_i := \HMfrom(N_i \setminus \ball, \mu_i)(\hat 1) \in \HMfrom (Y)
\end{equation}
for $i=0,1$, the difference element $\Delta := x_0-x_1$ belongs to the reduced homology $\HMred(Y)$. In particular $\Delta$ is $U$-torsion. 
\item \label{thm:DeltaInHMred:item:NMS}Let $d_\Delta \geq 0 $ be  the  smallest natural number such that $U^{d_\Delta}$ annihilates $\Delta$. Fix a metric and perturbation on $Y$ as in Subsection~\ref{subsec:CriticalPoints} and define the Morgan-Szab\'{o} number $n_{MS}$ as $2$ plus the maximum relative degree between two irreducible generators of the chain complex 
relative to the trivial \spinc structure. Then $n_{MS}~\geq~2 d_\Delta$.
\item\label{thm:DeltaInHMred:item:HomOrientation} If $N_0$ is a symmetric protocork, then $\mu_1$ realizing the first item is the  pushforward of $\mu_0$ via the diffeomorphism $\hat \tau: N_0\to N_1$ defined in Subsection~\ref{subsec:involutions}. 
\end{enumerate} 
\end{theorem}
The proof is given  in Section~\ref{sec:VariationSW}. The inequality $n_{MS}\geq 2d_\Delta$ of item \ref{thm:DeltaInHMred:item:NMS} 
leads to the following improvement of the main technical result of  Morgan-Szab\'{o} in  \cite[Theorem 1.2	]{MorganSzabo99} regarding the variation of Seiberg-Witten invariants.
 
\begin{corollary}\label{corollary:MorganSzaboImprovement} Let $d_\Delta$ be the $U$-torsion order of the class $\Delta$, i.e. the least $k\geq 0$, such that $U^{k}\Delta  =0$,
and let $X_0$ be an oriented, closed, simply connected $4$-manifold with $b_2(X_0)>1$ decomposed as $X_0 = M\bigcup_Y N_0$ 
	where 	$M$ is an oriented $4$-manifold with boundary $-Y$. Fix $\sstruc_0$, a \spinc structure on $X_0$ and  an homology orientation $\mu_0$ for $X_0$.
	 Let $X_1 = M\bigcup_Y N_1$ be the result of the protocork twist operation on $X_0$ and denote by $\sstruc_{1}$ the $\spinc$ structure induced by $\sstruc_0$ (unique up to isomorphism).
	Suppose that 
	 \begin{equation}
	  	d(\sstruc_0) = \frac 1 4 (c_1^2(\sstruc_0) - 2\chi (X_0)-3\sigma(X_0)) \geq 2 d_\Delta
	 \end{equation}
	 then 
	\begin{equation}
		SW(X_0,\sstruc_0, \mu_0) = SW(X_1, \sstruc_1, \mu_1)
	\end{equation}
	for an homology orientation $\mu_1$ on $X_1$.
\end{corollary}
  Theorem 1.2. of \cite{MorganSzabo99}, rephrased in the protocork terminology, corresponds to Corollary~\ref{corollary:MorganSzaboImprovement} with $n_{MS}$, as defined in Theorem~\ref{thm:DeltaInHMred} \ref{thm:DeltaInHMred:item:NMS}, in place of $2d_\Delta$.
  Our result improves \cite[Theorem 1.2	]{MorganSzabo99} because by item \ref{thm:DeltaInHMred:item:NMS}  our upper bound $d_\Delta$ is at least as good as Morgan-Szab\'{o}'s
 and has  the advantage of being a purely topological quantity, whilst $n_{MS}$ depends on the choice of Riemannian metric and perturbation. Moreover, it can be easily  expressed in terms of $\HMfrom(Y)$ and has  an upper bound, given by the least $k\geq 0$ such that $U^k$ annihilates $\HMred(Y)$, that can be computed without knowing $\Delta$.

\begin{proof}[Proof of Corollary~\ref{corollary:MorganSzaboImprovement}] Since the result of gluing a \spinc structure over $M$ to one over $N_0$ is unique up to isomorphism (see Proposition~\ref{propProtocork}), the gluing formula \cite[Proposition 3.6.1]{KM} applies to each \spinc structure of $X_i$ separately and we have that
\begin{spliteq}
		SW(X_i,\sstruc_i) &=  \big \langle \overset{\longrightarrow}{HM}_*( U^{d(\sstruc_0)/2} \ |\  M\setminus\ball, \sstruc_M) \circ \widehat{HM}_*(\ N_i\setminus \ball ) (\hat 1), \check 1 \big \rangle\\
									&=  \big \langle \overset{\longrightarrow}{HM}_*(M\setminus\ball, \sstruc_M)(U^{d(\sstruc_0)/2} x_i), \check 1 \big \rangle\\
	\end{spliteq}	
	where $\sstruc_M$ is the restriction of $\sstruc_0$ to $M\setminus \ball$, the \spinc structure over $N_i\setminus \ball$ is the trivial one  (see Proposition~\ref{propProtocork}), the homology orientations of $M\setminus \ball$ and $N_0\setminus \ball$ are such that their composition is $\mu_0$ and $N_1\setminus \ball$ has the homology orientation given by
	Theorem~\ref{thm:DeltaInHMred}.
	In particular we have that 
	\begin{equation} \label{eq:splittingOfSW}
		SW(X_0,\sstruc_0)- SW(X_1,\sstruc_1) =  \big \langle \overset{\longrightarrow}{HM}_*(M\setminus\ball, \sstruc_M) U^{d(\sstruc_0)/2} \Delta, \check 1 \big \rangle
	\end{equation}
	from which the thesis follows because $ U^{d(\sstruc_0)/2} \Delta = 0$ as $d(\sstruc_0)\geq 2 d_\Delta$.
	\end{proof}

It was shown in the proof of \cite[Theorem 8.1]{LinRubermanSaveliev2018} that, in the case of the Akbulut cork, 
only the reduced Floer homology of the cork boundary  contributes to the variation in Seiberg-Witten invariants due to a cork twist.
In the next corollary we generalize this result to a larger class of $4$-manifolds including all corks.
We willl assume all the manifolds to be connected, oriented and homeomorphisms to preserve the orientation.
\begin{corollary}\label{intro:corollaryCorkDiff} Let $C_0, C_1$ be compact,  simply-connected $4$-manifolds with connected boundary $Y$,   $f: \partial C_0\to \partial C_1$  a diffeomorphism and $\sstruc_0 \in \Spinc(C_0)$.
If $f$ extends to a  homeomorphism $\tilde f: C_0\to C_1$, then  setting $\sstruc_1 := \tilde f_*\sstruc_0$ and
\begin{equation}
		x_i := \HMfrom(C_i\setminus \ball, \sstruc_i)(\hat 1) \in \HMfrom(Y, {\sstruc_i}_{|Y}),
\end{equation}
it holds that  $x_1 - f_* x_0 \in \HMred(Y, {\sstruc_1}_{|Y})$ for some homology orientations on $C_0,C_1$.
\end{corollary}

\begin{proof}[Proof of Corollary~\ref{intro:corollaryCorkDiff}] Proposition~\ref{prop:CorkObstruction} ensures the existence of a symmetric protocork $N_0\subset \mathrm{int}(C_0)$  supporting $(C_0, C_1,f)$.
This means that if $T = C_0\setminus \mathrm{int}(N_0)$ and we denote by $(C_1)_{f}:\emptyset\to Y$ the cobordism  with underlaying manifold $C_1$ and $f:Y\to \partial C_1$ as map identifying the outcoming boundary, then as a cobordism  
\begin{equation}\label{eq:intro_cork_supported}
		(C_1)_{f}\simeq T\circ (N_0)_\tau,
\end{equation}
where $\tau: \partial N_0\to \partial N_0$ is the twisting map of Subsection~\ref{subsec:involutions}, moreover,  the rider to Proposition~\ref{prop:CorkObstruction}, implies that the diffeomorphism  of \eqref{eq:intro_cork_supported} sends ${\sstruc_0}_{|_T}$ to  $\sstruc_1|_{F(T)}$.
Fixing an homology orientation on $C_0$, we have that 
\begin{equation}
		x_0 = \HMfrom(C_0\setminus \ball,\sstruc_0)(\hat 1) = \HMfrom(T,{\sstruc_0}_{|_T})\circ \HMfrom(N_0 \setminus \ball)(\hat{1}),
\end{equation}
for some homology orientations  $\mu_T, \mu_{N_0}$ on $T,N_0$. By \eqref{eq:intro_cork_supported}, 
\begin{equation}
	 f^{-1}_* x_1 = \HMfrom((C_1\setminus \ball)_f, \sstruc_1 )(\hat 1) =  \HMfrom(T, {\sstruc_0}_{|_T}) \circ \tau_* \HMfrom(N_0\setminus \ball)(\hat{1}),
\end{equation}
for an homology orientation on $C_1$ obtained gluing up $\mu_T, \mu_{N_0}$ on $T,N_0$.
Now, $\HMfrom(N_0\setminus \ball)(\hat{1}) - \tau_*\HMfrom(N_0\setminus \ball)(\hat{1}) $ is the difference element $\Delta$ of Theorem~\ref{thm:DeltaInHMred} in the case of a symmetric protocork.
Thus the thesis follows from $\Delta \in \HMred(\partial N_0)$ and the fact that cobordism maps send reduced homology to reduced homology.\end{proof}
We remark that Corollary~\ref{intro:corollaryCorkDiff}  was already known when $b^+(C_0)>0$ by \cite[Proposition 3.5.2]{KM}, our addition is the case when $C_0$ is negative semi-definite.
This relates to Seiberg-Witten invariants as follows. Consider closed $4$-manifolds $X_0 = M\bigcup C_0$, $X_1 = M\bigcup_f C_1$
where $b^+(M)\geq 2$. Let $\sstruc_M\in \Spinc(M)$ and $\sstruc_i \in \Spinc(C_i)$, then under the assumptions of Corollary~\ref{intro:corollaryCorkDiff},   \cite[Proposition. 3.6.1]{KM} implies that, for some choice of homology orientations,
\begin{equation}
\sum_{\substack{\sstruc|_M = \sstruc_M \\
							\sstruc|_{C_1} = \sstruc_0}} SW(X_1, \sstruc) - \sum_{\substack{\sstruc|_M = \sstruc_M \\
							\sstruc|_{C_0} = \sstruc_1}} SW(X_0, \sstruc) = \big \langle \overset{\longrightarrow}{HM}_*(M\setminus\ball, \sstruc_M)(U^{d(\sstruc)/2} (f^{-1}_*x_1-x_0)), \check 1 \big \rangle,
\end{equation}
where $d(\sstruc)$ can be computed with any \spinc structure on the left hand side.
Corollary~\ref{intro:corollaryCorkDiff} then tells us that the difference of the sum of the Seiberg-Witten invariants on the left hand side depends only 
on an element in $\HMred(Y)$. This equation should be compared with \eqref{eq:splittingOfSW}, in that case the sum contained only one \spinc structure thanks to Proposition~\ref{propProtocork}.

\paragraph{Second result.} In order to prove  the main theorem we prove the following splitting theorem
 for the monopole Floer homology of protocork boundaries which determines it up to $U$-torsion elements.\\
It 	roughly says that the Floer homology is a direct sum of the Floer homology of a connected sum of $\SS^1\times \SS^2$  and the reduced homology $\HMred$  \cite[Definition 3.6.3]{KM}. 
The former is well known \cite[Proposition 36.1.3]{KM} and the latter has finite rank.

{\newcommand{\Sone}{\#_{1=1}^{b_1(Y)}(\SS^1\times \SS^2)}
\begin{theorem}\label{Intro:thm:Splitting} Let $Y$ be the boundary of a protocork, then as $\ZU$-modules
\begin{spliteq}
	& \HMfrom(Y) \simeq \HMfrom(\Sone) \oplus \HMred(Y)\\
	& \HMto(Y) \simeq \HMto(\Sone) \oplus \HMred(Y)\\
	& \HMbar(Y) \simeq \HMbar(\Sone).
	\end{spliteq}
In this splitting, the Floer homology of $\Sone$ is supported in the (unique up to isomorphism) torsion \spinc structure of $Y$ and the isomorphism preserves the absolute $\Q$-grading.
Moreover,	with respect to this splitting, the long exact sequence
	\begin{equation}
		\dots \to \HMto(Y)\overset{j_*}{\longrightarrow}\HMfrom(Y)\overset{p_*}{\longrightarrow}\HMbar(Y)\to\dots
	\end{equation}
	becomes the direct sum of the long exact sequence of  $\Sone$ and 
	\begin{equation}
		\dots \to \HMred(Y)\overset{id}{\longrightarrow}\HMred(Y){\longrightarrow}0\to \dots.
	\end{equation}
	
\end{theorem}
}
We prove this theorem in Subsection~\ref{subsec:FloerHomologyCobordismSplitting}, the proof relies on the functoriality of Floer homology and on the construction of some special cobordisms. 

\paragraph{Result on moduli spaces of monopoles over a protocork.} 
{\newcommand{\ntop}{N^{\mathrm{red}}}
We include a geometric result  about moduli spaces of reducible monopoles which can also be used to give an alternative, geometric, proof of our main theorem but which  is also interesting in its own right.
Recall that the relative Seiberg-Witten invariant \eqref{def:relativeSWInvariantPcork} of a protocork $N_0$  are computed by counting points in the moduli spaces $M(N_0^*, \beta)$ of perturbed monopoles over $N_0^*$ ($N_0$ with cylindrical ends attached) limiting to perturbed $3D$-monopole $\beta$ on $Y=\partial N_0$.
In Section~\ref{sec:4DModulispaces}, using the perturbations defined in Section~\ref{sec:FloerHomologyBoundary}, we show that there is only one limiting \emph{reducible} solution denoted as $[\gota_{\ntop,-1}]$, that can contribute to the relative Seiberg-Witten invariant (more \emph{irreducible} solutions are possible), then we prove the following.
\begin{theorem}\label{thm:ModuliIsPoint}
		Let $N_0$ be a protocork and $N_1$ be its reflection. Then for $i=0,1$
		there exist  a residual set of perturbations, generic enough to compute \eqref{def:relativeSWInvariantPcork}, such that  the moduli space	$M(N_i^*; [\gota_{\ntop,-1}]) $ is regular and consists of a single point, in particular is not empty.
\end{theorem}
This provides a geometric explanation for the fact that $\Delta\in \HMred(Y)$, indeed,  $\gr^{\Q} (\Delta) = -1$ and  it follows from Theorem~\ref{Intro:thm:Splitting} that $\Hfrom_{-1}(Y)/\HMred_{-1}(Y) \simeq \mathbb Z$ generated by  $x_0$ (see the proof of Theorem~\ref{thm:DeltaInHMred}\ref{thm:DeltaInHMred:item:Delta}). Hence $x_1$ has to be equal to $\pm x_0$ in the quotient, because the component of both $x_0 $ and $\pm x_1$ along the generator of the complex $[\gota_{\ntop,-1}]$ are the same by Theorem~\ref{thm:ModuliIsPoint} and cycles in $\HMred(Y)$ have zero component along it by Lemma~\ref{lemma:HMred2} below. The computations of the degree of the generators given in  Section~\ref{sec:4DModulispaces} can also be used to give an alternative proof of Theorem~\ref{thm:DeltaInHMred}\ref{thm:DeltaInHMred:item:NMS}.

The lemmas and techniques used in Section~\ref{sec:FloerHomologyBoundary} and \ref{sec:4DModulispaces} serve also to 
 lay the basis for future studies of protocorks and their boundaries from the differential geometric perspective.
 
The paper also includes, in the appendix, the proof of two technical results, regarding general manifolds, non necessarily protocorks, which are  probably known to experts but that we could not trace in the literature. The first, Proposition~\ref{prop:MorseLikePerturbationPerp}, deals
with perturbations which are pullback of a Morse function on the torus of reducibles. The second, Lemma~\ref{lemma:FromL2loctoL2},  is a characterization of reducible moduli spaces over manifolds with cylindrical ends.
}

\paragraph{Organization.} The paper is organized as follows:  in Section~\ref{sec:Protocorks} we define protocorks (Definition~\ref{def:protocork}) and the protocork twist operation (Definition~\ref{def:ProtocorkTwist}) and  we obtain some basic properties that will be used in the other sections.  The proof that any exotic pair is related by a protocork twist and
that any cork has a supporting protocork is in Subsection~\ref{subsec:ProtExotic4man}. 

The first part of Section~\ref{sec:FloerHomologyBoundary} provides the framework for the definition of  the chain complex giving the Floer homology of protocork boundaries, this is used to define the  Morgan-Szab\'o number and to study the moduli spaces of  monopoles over a protocork later in Section~\ref{sec:4DModulispaces}. The  second part Section~\ref{sec:FloerHomologyBoundary} contains the proof of Theorem~\ref{Intro:thm:Splitting}. 
In Section~\ref{sec:VariationSW} we prove Theorem~\ref{thm:DeltaInHMred}. 
In Section~\ref{sec:4DModulispaces} we prove Theorem~\ref{thm:ModuliIsPoint}.
In \autoref{appendix:ProofMorseLike} we prove Proposition~\ref{prop:MorseLikePerturbationPerp}, in  \autoref{app:ProofofLemmaL2locL2}
we prove Lemma~\ref{lemma:FromL2loctoL2}.
More information and proof sketches may  be found in the outlines at the beginning of each section.

\begin{acknowledgements} The author wishes to express  his gratitude to his supervisor Steven Sivek  for his support, encouragement and optimism and also to
  Francesco Lin, Dima Panov and Daniele Zuddas for useful discussions and thoughts on the early draft.
  The author also would like to thank Tom Mrowka, Bruno Roso, Charles Stine  and Zolt\'{a}n Szab\'{o} for helpful conversations.
  Furthermore he would like to thank Bernd Ammann, Bob Gompf and  Mike Miller for valuable correspondence.
  The author is grateful to  an anonymous reviewer for suggesting him Lemma~\ref{lem:RepByIrreducibleCycles}, which makes the proof of Theorem~\ref{thm:DeltaInHMred}\ref{thm:DeltaInHMred:item:NMS} independent from  Section~\ref{sec:4DModulispaces}.
 The author was funded by EPSRC.
\end{acknowledgements}

\section{Protocorks}\label{sec:Protocorks}

\newcommand{\valid}{protocork plumbing }
\newcommand{\vA}{\mathbf{v}^A}
\newcommand{\vB}{\mathbf{v}^B}
\newcommand{\Realiz}{\mathcal{R}}

\begin{plan} In Subsection~\ref{subsec:ReviewPlumbings} we review plumbings and their realizations. This is important to us because we will need to define diffeomorphisms of plumbings, thus to our aims is not enough  to work with the diffeomorphism type of a plumbing. In Subsection~\ref{subsec:DefProtocork} we define protocorks.
In Subsection~\ref{subsec:KirbyDiagrProtocorks} we describe Kirby diagrams of protocorks, these  will be used throughout the paper in particular  to construct cobordisms that we will exploit in the proof of our theorems. In Subsection~\ref{subsec:homological properties} we infer some basic algebraic-topological properties of protocorks and their boundary that will be used throughout the paper.
In Subsection~\ref{subsec:involutions} we define some involutions on protocork boundaries and some diffeomorphisms of manifolds associated to protocorks and we prove some properties about them.
In Subsection~\ref{subsec:ProtExotic4man} we define the protocork twist operation, we prove that any simply connected exotic pair can be related by such operation and that any cork admits a supporting protocork.
\end{plan}

\paragraph{Inspiration and motivation.} Protocorks are a natural object to consider when dealing with exotic 4-manifolds and h-cobordisms, indeed they are implicit in the proofs of \cite{CHFS, Matveyev}
and in Morgan-Szab\'{o}'s definition of complexity of an h-cobordism \cite{MorganSzabo99}. 
Protocorks with sphere-number $1$ appear also  in the unpublished thesis of Sunukjian \cite{SunukjianThesis} under the name $D(n,1)$ where the author also computes $\HMbar$ of the boundary using the vanishing of the triple cup product.
However, to the author's best knowledge this is the first paper to attempt a systematic study of  them.

Our intention is to use protocorks to prove general statements regarding corks and exotic pairs. 
This is motivated by two key features:  first of all we can enumerate protocorks, as they are in bijection with certain bipartite graphs. This is in stark contrast with the cork's situation: there is no classification of corks and it is not easy to produce corks that can potentially yield  new exotic pairs. Secondly, 
protocorks and their boundaries possess interesting symmetries (see  Subsection~\ref{subsec:involutions}) and other geometrical features that come in handy when doing gauge theory (see the discussion at the beginning of Section~\ref{sec:4DModulispaces}).

\subsection{Review of plumbings and their realizations.}\label{subsec:ReviewPlumbings}
Plumbings  were firstly introduced in the context of surgery theory by Milnor in \cite{MilnorPlumbings}. 
In this subsection we review them in the specific case of 4-manifolds and establish our notation, for examples and further discussion we refer the reader to \cite[Section 6.1]{GompfStipsicz}.

\paragraph{Plumbing graphs.} A plumbing graph is a finite, undirected multigraph  where each vertex is labelled with a pair $(g,e)\in \N_{\geq 0}\times \Z$  and each edge is decorated with a sign $+$ or $-$.
Such graphs  model an intersection of surfaces in an ambient 4-manifold: each vertex represents a Riemann surface of genus $g$  embedded with  normal bundle of Euler number $e$, and  each $+$ ($-$) edge represents a positive (negative) intersection between the surfaces.
Given a plumbing graph  $\Gamma$ with set of vertices $V$, for $i,j \in V$, we will  denote by $r_\Gamma(i,j)\in \N_{\geq 0}$  the \emph{geometric intersection number} i.e. the  number of edges between $i$ and $j$  and by $a_\Gamma(i,j)\in \Z$  the \emph{algebraic intersection number} i.e. the signed count of the edges between $i$ and $j$. 
The edges constitute a multiset therefore we will adopt the following convention: we will write  $\varepsilon \simeq (v_1,v_2, \pm)$ to denote an edge $\varepsilon$ between the vertices $v_1,v_2$ of sign $\pm$, of course we can have $\varepsilon,\varepsilon'\simeq (v_1,v_2, \pm)$ and $\varepsilon\neq \varepsilon'$.

\paragraph{Realization datum.} We will denote by $\D^n(r)\subset \R^n$, the disk of radius $r>0$ so that  $\D^n = \D^n(1)$.
To a plumbing graph $\Gamma$ there is associated  a \emph{diffeomorphism class} of 4-manifold called plumbing on $\Gamma$. If we want to associate a specific 4-manifold we need to add some data. A \emph{realization datum} for $\Gamma$ consists of the following choices:
\begin{enumerate}
\item for each vertex $v$ of $\Gamma$ with label $(g,e)$,  is given a preferred \emph{oriented} surface of genus $g$, $S_{v}$ with an \emph{oriented} $\D^2$-bundle  $N_{v}\to S_v$ of Euler class $e $ and
\item denoting by $d_v  $ the degree of $v$,  is given a preferred embedding $\varphi_{v}: \bigsqcup_{i=1}^{d_v} \D^2(2)\to S_v$ together with a preferred trivialization of $N_{v}$ restricted to the image of $\varphi_{v}$.	
\item \label{itemEdgesAssRealizationDatumDef}  Let  $v_1$ and $v_2$ be two vertices of $\Gamma$. Then 	
		for each edge of $\Gamma$, $\varepsilon \simeq (v_1,v_2, \pm)$  are chosen two unit disks 	 $D_{\varepsilon, 1}$ and $D_{\varepsilon, 2}$, from  the image of 
		\begin{equation}\label{eq:diskchoice1}
			\bigsqcup_{i=1}^{d_{v_1}} \D^2(1)\subset 	\bigsqcup_{i=1}^{d_{v_1}} \D^2(2)  \overset{\varphi_{v_1}}{\to } S_{v_1}\hookrightarrow S_{v_1}\times \{v_1\}
		\end{equation}
		and
		\begin{equation}\label{eq:diskchoice2}
				\bigsqcup_{i=1}^{d_{v_1}} \D^2(1)\subset	\bigsqcup_{i=1}^{d_{v_2}} \D^2(2)  \overset{\varphi_{v_2}}{\to} S_{v_2}\hookrightarrow S_{v_1}\times \{v_2\}
		\end{equation}	respectively.
		This choice must respect the following  global condition: each disk in the image of \eqref{eq:diskchoice1} and \eqref{eq:diskchoice2} is associated to \emph{exactly one} edge.
\item \label{itemSmoothingRealizationDatumDef}	Denote by $M_{\pm}$ the manifold with corners
		obtained by gluing  $\mathrm{int}(\D^2(2))\times \D^2\times \{0\}$ and $\mathrm{int}(\D^2(2))\times \D^2\times \{1\}$
		via the identification
		\begin{equation}\label{eqdef:Invpm}
				\operatorname{inv}_\pm:\D^2(1)\times \D^2\times \{0\}\to \D^2(1)\times \D^2 \times \{1\} = (b,f,0)\mapsto \pm (f,b,1).
		\end{equation}
		Then  for each edge $\varepsilon \simeq (v_1,v_2, \pm)$ is chosen a smoothing of the corners of $M_\pm$ \cite[Section 2]{WallBook}.
\end{enumerate}

	Calling   $M_{\pm}$a manifold with corners is a slight abuse of language because a neighbourhood of a singularity is diffeomorphic to 
	\begin{equation}
		\R^2\times \{(x_1,x_2 ) \in \R^2\ | x_1 \geq 0 \ \text{ or } x_2\geq 0\ )\},
	\end{equation} i.e. is a \emph{concave} corner,  instead of being diffeomorphic to $\R^2\times [0,+\infty)^2$, i.e. a \emph{convex} corner, as required by the usual definition corner point \cite[pg. 30]{WallBook}. Nevertheless the two are homeomorphic and the corner straightening theory \cite[Section 2]{WallBook} is carried out with little modifications.
	
\paragraph{Realization.} Let $\Gamma$ be a plumbing graph and let $\Realiz$ be a realization datum for $\Gamma$. Then we can construct an oriented  smooth 4-manifold with boundary $P(\Gamma, \Realiz)$,  \emph{the realization of the plumbing on} $(\Gamma, \Realiz)$ constructed in the following way: let
	\begin{equation}\label{defOfEgammaRealiz}
		E(\Gamma, \Realiz) = \bigsqcup_{v} N_{v}\times \{v\}
	\end{equation}
	where the $v$ varies among vertices of $\Gamma$.
	Then $P (\Gamma, \Realiz)$ is obtained from $E(\Gamma, \Realiz)$ by identifying some codimension zero submanifolds:
	for each edge $\varepsilon $ we identify $E(\Gamma, \Realiz) |_{D_{\varepsilon,1}} $ and $E(\Gamma, \Realiz) |_{D_{\varepsilon,2}} $ by making this diagram commutative
	
	\begin{equation}\label{eq:diskIdentification}
			\begin{tikzcd}	    
				&  E(\Gamma, \Realiz) |_{D_{\varepsilon,1}} \arrow[d,"\sim"] \arrow[r] &  E(\Gamma, \Realiz) |_{D_{\varepsilon,2}} \arrow[d,"\sim"]   \\
				& 	\D^2\times \D^2	\arrow[r, "\operatorname{inv}_\pm"]& \D^2\times \D^2   \\
			\end{tikzcd}
	\end{equation}
	where the vertical arrows are given by the local trivializations (where the second factor is the fiber) chosen with $\Realiz$ and $\operatorname{inv}_\pm$ is the involution defined by \eqref{eqdef:Invpm} 
	that exchanges  the base with the fiber.
	$P (\Gamma, \Realiz)$ is naturally a smooth manifold with (concave) corners along the plumbing tori (see below). 
	Exploiting $\Realiz$, we define a smooth structure on $P (\Gamma, \Realiz)$  by straightening the corners using the data of item  \ref{itemSmoothingRealizationDatumDef} in the definition of realization datum.

	\begin{remark} It can be shown that the (oriented) diffeomorphism type of $P (\Gamma, \Realiz)$ is independent of the realization datum $\Realiz$ and depends on $\Gamma$ only up to isomorphism of plumbing graphs (which  is an isomorphism of graphs that preserves the labels of the vertices and the sign of the edges). Thus we write $P (\Gamma)$ to denote the diffeomorphism type of  the plumbing on $\Gamma$.
	\end{remark}

	\paragraph{Plumbing tori.} We recall a definition that we will need in Proposition~\eqref{propProtocork} and Proposition~\eqref{prop:existenceSymmetry}.
	The torus  $\partial \D^2\times \partial \D^2\subset \D^2\times \D^2$ is preserved by $\operatorname{inv}_\pm$. Thus via the diagram \eqref{eq:diskIdentification}, each edge of $\Gamma$ gives rise to a torus lying in 	the boundary $Y:=\partial P(\Gamma, \Realiz)$. Such tori are called \emph{plumbing tori}.
	Notice that the union of plumbing tori separate  $Y$ in a collection of $\SS^1$-bundles over surfaces with boundary thus giving $Y$ the structure of a graph manifold.

	\paragraph{Plumbing structures and associated surfaces.} A plumbing structure on a 4-manifold $X$ is a diffeomorphism $f:X\to P(\Gamma, \Realiz)$ for some pair $(\Gamma, \Realiz)$.
	Given such a structure,  the surface in $ X$ associated to a vertex $v$ of $\Gamma$, is the preimage of the surface associated to $v$ in $P(\Gamma, \Realiz)$ (i.e. the image of the zero section of $N_v\times\{v\}$ in \eqref{defOfEgammaRealiz}).
	Notice that a tubular neighbourhood of an intersection of surfaces intersecting according to a plumbing graph $\Gamma$,  has a an obvious plumbing structure.

\subsection{The definition of protocork.}\label{subsec:DefProtocork}\begin{plan}
The aim of this  subsection is to define protocorks (Definition~\ref{def:protocork}), these can be thought as the incoming end of certain 5-dimensional h-cobordisms consisting merely of 5-dimensional 2-handles and 3-handles.
Such a cobordism is specified by the intersection pattern of the belt spheres of the 2-handles with the attaching spheres of the 3-handles. 
We record this information in the form of a protocork plumbing graph (Definition ~\ref{def:ppgraph}).
In order to define the protocork twist operation later,  it is convenient to keep track also of the outcoming end and of the middle level of the above-mentioned cobordism, we will call the former reflection of the protocork and the latter the plumbing associated to the graph. Hence in  Definition~\ref{def:protocork} we will define three objects: plumbing, protocork and reflection. 
\end{plan}

\begin{definition}[Protocork plumbing graph] \label{def:ppgraph} A protocork plumbing graph is a bipartite plumbing graph $\Gamma$ with $2n\geq 0$ vertices and parts labelled $A$ and $B$ satisfying the following conditions:
\begin{enumerate}
\item $A$ and $B$ have cardinality  $n$ called the \emph{sphere-number} and their elements are denoted as $\{\vA_i\}_{i=1}^n, \{\vB_i\}_{i=1}^n$,
\item the $\mathbb D^2$-bundle associated to each vertex  is the  trivial  bundle  over $\SS^2$,
\item $r_\Gamma(\vA_i, \vA_j) = r_\Gamma(\vB_i, \vB_j) = 0$ for each $i,j=1,\dots, n$,
\item  $a_\Gamma(\vA_i, \vB_j)=\delta_{i,j} $  for each $i,j=1,\dots, n$, $\delta_{i,j}$ being the Kronecker delta.
\end{enumerate}
\end{definition}

Since the only edges occur between $A$-vertices and $B$-vertices, we will write  $\varepsilon \simeq (i,j,\pm)$  to denote a positive/negative  edge between $\vA_i$ and $\vB_j$.
	
A \valid graph $\Gamma$ is called \emph{symmetric} if $r_\Gamma(\vA_i,\vB_j) = r_\Gamma(\vA_j,\vB_i)$ for each  $i,j=1,\dots, n$. In particular  every $\Gamma$ with sphere-number $1$ is symmetric.
	 Given a \valid graph $\Gamma$ we can form another \valid graph called the  \emph{reflection} of $\Gamma$ denoted  as $\overline{\Gamma}$ by swapping the labels  $A$ and $B$.  More precisely,
	 the part $A$ of $\overline\Gamma$  is the part $B$ of ${\Gamma}$ and the part $B$ of $\overline \Gamma$ is the part $A$ of ${\Gamma}$; the multiset of edges of $\overline\Gamma$ is the same multiset of edges of $\Gamma$.

	Two \valid graphs are isomorphic  if there is an isomorphism of unoriented graphs that preserves the labelling of the parts and the sign of the edges. 	
	In particular $\Gamma$ and $\overline \Gamma$ need not be isomorphic but if $\Gamma$ is symmetric they are.  We will adopt the convention of drawing the $A$-vertices on the left and the $B$-vertices on the right.
	\autoref{figure1} shows some examples of \valid graphs.

		\begin{figure}
		\includegraphics[scale=0.55]{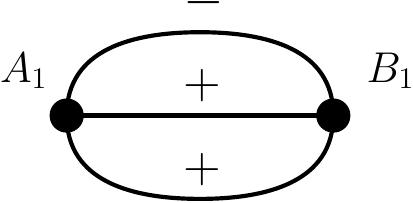}\hspace{1.5cm}
		\includegraphics[scale=0.45]{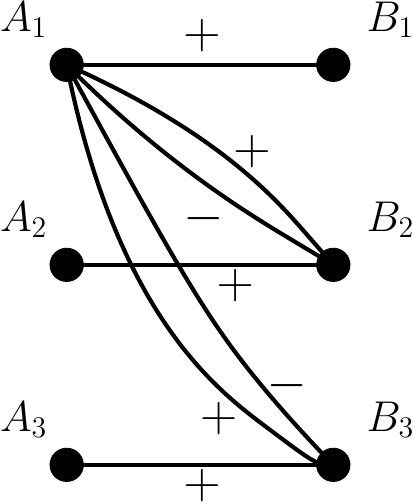}\hspace{1.5cm}
		\includegraphics[scale=0.45]{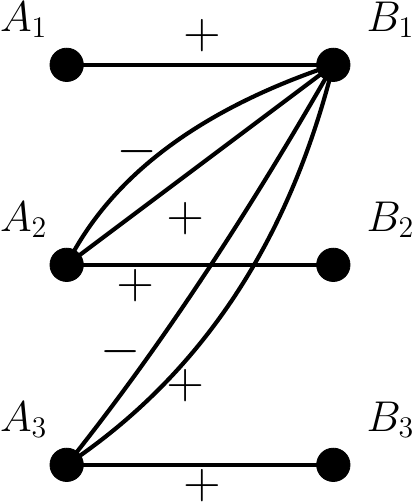}
	\caption{\label{figure1} From left to right: a symmetric \valid graph with sphere-number $1$, an asymmetric \valid graph with sphere-number $3$, the reflection of the previous example.}
		\end{figure}
		
		\begin{definition}\label{def:RealizationDatumForProtocork}A realization datum  $\Realiz$ for a protocork plumbing graph $\Gamma$ is a realization datum for the plumbing graph $\Gamma$ with an extra information: to each vertex $v$ is
		associated a preferred trivialization of the bundle $N_v\to S_v$.
		\end{definition}
		
		\paragraph{Natural framing.}
		Consider a plumbing graph  $\Gamma$ with set of vertices $V$ where each vertex is labelled with the trivial bundle over $\SS^2$. Let $\Realiz$ be a realization datum for $\Gamma$ where the model disk bundles are 
		$ \D^2\times \SS^2$.
		In this case, the plumbing realization $P:= P(\Gamma, \Realiz)$ is a quotient of  $\bigcup_{v\in V} \D^2\times \SS^2\times \{v\}$.  Let $v\in V$, then the quotient map restricts to a smooth  embedding over  $\mathrm{int}(\D^2)\times \SS^2\times \{v\}$. Notice that this extends only to a topological embedding of $\D^2\times \SS^2\times \{v\}$ because the corner straightening procedure introduces some corners in the image of this embedding.
		
		 Define  $S_v\subset P$ to be the image of $\{0\}\times \SS^2\times\{v\}\subset \mathrm{int}(\D^2)\times \SS^2\times\{v\}$  under the quotient map. 
		We call $S_v$ the \emph{sphere induced by $v$}  in $P$.  $S_v$ has a natural framing because is image of the  zero section of $\mathrm{int}(\D^2)\times \SS^2\times\{v\}$. It is thus possible to perform \emph{surgery} on the sphere induced by $v$ using its natural framing (see \cite[Definition 5.2.1]{GompfStipsicz}  pg. 154 for a definition of surgery) by embedding $\SS^2\times \D^2$ as $\D^2(\frac 1 2)\times \SS^2\times \{v\}$ via the  radial contraction $\D^2\to \D^2(\frac 1 2)$.

		In particular, 	by using the trivializations of  Definition~\ref{def:RealizationDatumForProtocork}, it follows that  if $\Realiz$ is a realization datum for a \emph{protocork} plumbing graph  $\Gamma$, there is a well defined natural framing
		for each surface $S_v$.
				
		We remark that in general, given a manifold $M^n$ and a framed sphere $f:\SS^k\times \D^{n-k}\to M$, the surgered manifold $M_f$ 		
		has a well defined smooth structure except that in a neighbourhood
		of $f(\SS^k \times \partial \D^{n-k})$ where it is defined up to the choice of  a collar of $f(\SS^k \times \partial \D^{n-k})$ in $M\setminus f(\SS^k\times \mathrm{int}(\D^{n-k}))$ and a collar of the boundary of $\D^{k+1}\times \SS^{n-k-1}$.
		In our case, we have a preferred choice because  the framed sphere is embedded as 
		
		\begin{equation}
			\D^2(\frac 1 2)\times \SS^2\times \{v\} \subset \mathrm{int}(\D^2)\times \SS^2\times \{v\}
		\end{equation}
		therefore, using polar coordinates to produce the two collars we end up with a well defined smooth structure on the surgered manifold.

	\begin{definition}[Protocork] \label{def:protocork} Let $\Gamma $ be a \valid graph of sphere-number $n$ and $\Realiz$ be  a realization datum for $\Gamma$. We  associate to $(\Gamma, \Realiz)$ the following oriented $4$-manifolds with boundary:
		\begin{itemize}
			\item The manifold $P_{1/2}(\Gamma, \Realiz)$ obtained by plumbing on $(\Gamma,\Realiz)$.  We call the spheres induced by vertices in the $A$-part \emph{$A$-spheres} and those  induced by vertices in the $B$-part \emph{$B$-spheres}.

			\item The manifold $P_0(\Gamma,\Realiz)$ obtained from $P_{1/2}(\Gamma,\Realiz)$ by surgering out the $B$-spheres  with their natural faming.
			$P_0(\Gamma,\Realiz) $ is called the \emph{protocork} associated to $(\Gamma, \Realiz)$.
			\item The manifold $P_1(\Gamma, \Realiz)$ obtained from $P_{1/2}(\Gamma, \Realiz)$ by surgering out the $A$-spheres with their natural framing. $P_1(\Gamma,\Realiz)$ is called the \emph{reflection} of $P_0(\Gamma,\Realiz)$.
		\end{itemize}
		
		\end{definition}	
		We remark that these three $4$-manifolds come with a natural identification of their boundaries. 
		
		We will use $P_{1/2}(\Gamma), P_0(\Gamma), P_1(\Gamma)$ to denote the diffeomorphism type of these manifolds, these depends just on the isomorphism class of $\Gamma$.
		With an abuse of language we will also write $N\simeq P_i(\Gamma)$ instead of $N\in P_i(\Gamma)$.

	$P_0(\Gamma)$ is said to be a \emph{symmetric protocork} if $\Gamma$ is a symmetric 	valid graph. In this case $P_1(\Gamma)$ is diffeomorphic to $P_0(\Gamma)$  as will follow from  
	Proposition~\ref{prop:existenceSymmetry}.
	Notice that $P_1(\Gamma,\Realiz)$ is also diffeomorphic to the protocork associated to $\overline{\Gamma}$ with the obvious reflected realization $\overline{\Realiz}$, this is why it is called reflection of $P_0(\Gamma, \Realiz)$.

	\subsection{Handle decomposition of protocorks.}\label{subsec:KirbyDiagrProtocorks}  
	\begin{plan}We will describe  a handle decomposition of $P_{1/2}(\Gamma), P_{0}(\Gamma)$ and $P_{1}(\Gamma)$ that will be used in the following sections. This decomposition will be given  in terms of  Kirby diagrams (\cite{KirbyCalculus}, also see  \cite[Section 5.4]{GompfStipsicz} for the dotted circle notation originally introduced by Akbulut \cite{Akbulut1977}). 	
	\end{plan}
	
	If $\Gamma$ is not connected, then $P_{t}(\Gamma)$, $t\in \{0,1/2,1\}$ is diffeomorphic to  the  boundary connected sum $  P_{t}(\Gamma_1)\natural \dots \natural  P_t(\Gamma_{|\pi_0(\Gamma)|})$ where $\Gamma_i$, $i=1,\dots, |\pi_0(\Gamma)|$ are the connected components of $\Gamma$. 
	 The Kirby diagram of a boundary connected sum is obtained by drawing the two diagrams next to each other, therefore we will assume $\Gamma$ is connected with sphere-number $n$.

	\paragraph{Kirby diagram for $P_{1/2}(\Gamma)$.} An algorithm to draw a  Kirby diagram for a plumbing is given in   \cite[Section 6.1]{GompfStipsicz}, we will apply it to the specific case of protocork plumbing graphs. 
	 Fix an embedding of $\Gamma$ in $\R^3$ where all the vertices and a neighbourhood of the endpoints of each edge lie in $\{0\}\times \R^2$.
	  Notice that protocork plumbing graphs do not have to be planar in general. 
	  Since $\Gamma$ is connected we can choose a  spanning tree $T\subset \Gamma$.
	  For each vertex $v$, we remove a small ball centered at $v$ from  $\Gamma\subset \R^3$ obtaining $\Gamma'$. 
	  Notice that this will also remove a portion of the edges meeting $v$, 	  we introduce a 0-framed circle $K_v$ in $\{0\}\times \R^2$ meeting $\Gamma'$ along these edges and fix an orientation for $K_v$ see \autoref{figure:diagramAlgo1} for an example.
	  
	  	Now for each edge not in $T$, add a dotted circle, meridional to that edge.
	Next we  replace each positive (negative) edge of the graph  with a positive (negative) clasp between the circles incident to that edge as in  \autoref{figure:clasps}. 
	Notice that  this step depends on the orientation of the knots.
	The clasp follows the path of the edge, thus  in the end, the clasps relative to edges not in $T$ will pass through a dotted circle.
	In this way we obtained a Kirby diagram with $2n$ 0-framed unknots and, denoting by $E$ the multiset of edges,  $k = |E|- |T| = |E|-(2n-1)$ dotted circles.
	See 	\autoref{figure:variousP} for an example, recall that by convention we draw the $A$-spheres on the left and the $B$-spheres on the right.
	
	\begin{figure}
	 \begin{center}
		\includegraphics[scale=0.5]{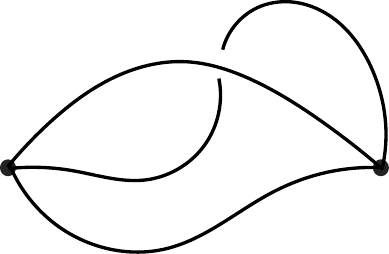}\hspace{1cm}
		\includegraphics[scale=0.5]{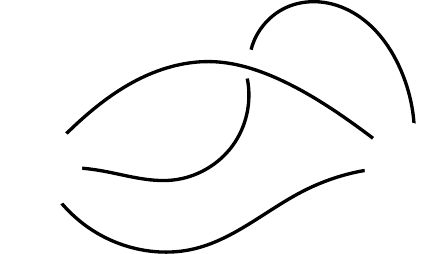}\hspace{1cm}
		\includegraphics[scale=0.5]{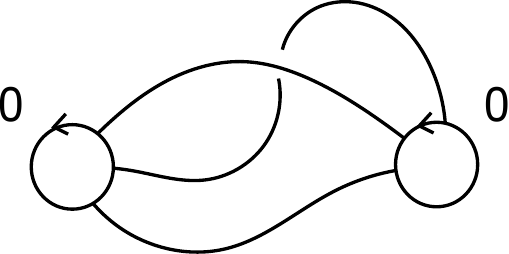}	 
	 \end{center}

	\caption{\label{figure:diagramAlgo1} From the left: embedding of $\Gamma$, the result of removal of small balls centered at the vertices  $\Gamma'$, introduction of 0-framed, oriented circles.}
		\end{figure}

\paragraph{Kirby diagram for $P_{0}(\Gamma)$ and $P_1(\Gamma)$.}	Let $K$ be a 0-framed unknot in a Kirby diagram and let $D$ be its spanning disk.
Assume that $D$ does not meet any dotted circles. Then in the interior of the 4-manifold given by the Kirby diagram, lies a  \emph{framed} sphere $S$  obtained by capping $D$ with the core of the 2-handle associated to $K$.  
In this case it can be shown \cite[Section 5.4]{GompfStipsicz} that a Kirby diagram for the surgery along $S$ is obtained by replacing the 0-framed  knot $K$ with a  dotted circle (isotopic to $K$).
Consequently, to draw a diagram for  $P_0(\Gamma)$,  we  exchange each 0-framed knot relative to $\vB_j$, $j=1,\dots, n$ with a dotted circle in the diagram for $P_{1/2}(\Gamma)$ described above. Similarly to draw the diagram for $P_1(\Gamma)$, we replace the knots relative to $\vA_i$ with dotted circles. 
See \autoref{figure:variousP} for an example.
Even though we will not make use of this fact, we mention that, as it is clear from the Kirby diagram, protocorks are diffeomorphic to the complement of $b_1(P_{0}(\Gamma))$ ribbon disks in   $\mathbb{D}^4$.

			\begin{figure}[h]
\centering
\subfloat[][\label{figure:clasps:a}]
   {\includegraphics[width=.30\columnwidth]{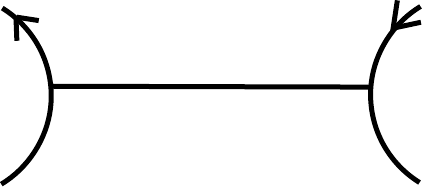}} \quad\quad\quad
\subfloat[][\label{figure:clasps:b}]
   {\includegraphics[width=.30\columnwidth]{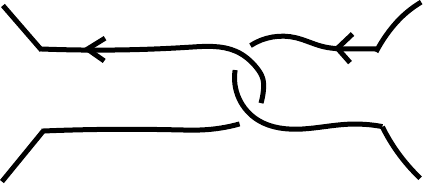}} \\
\subfloat[][\label{figure:clasps:c}]
   {\includegraphics[width=.30\columnwidth]{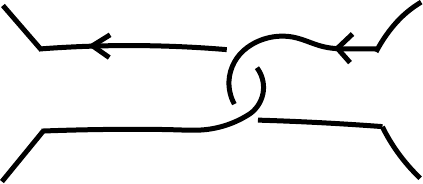}} \quad\quad \quad
\subfloat[][\label{figure:clasps:d}]
   {\includegraphics[width=.30\columnwidth]{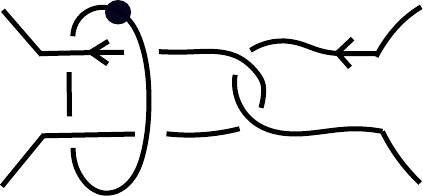}}\\
\caption{\label{figure:clasps} Figure~\ref{figure:clasps:a}) two oriented circles joined  by an edge. Figure~\ref{figure:clasps:b}) replacement of a positive edge in the spanning tree $T$ with a clasp. Figure~\ref{figure:clasps:c}) replacement of a negative edge in $T$ with a clasp. Figure~\ref{figure:clasps:d}) replacement of a positive edge \emph{not} in $T$.}
\end{figure}
\begin{figure}[h]
\centering
\subfloat[][\label{figure:variousP:a}]
   {\includegraphics[width=.15\columnwidth]{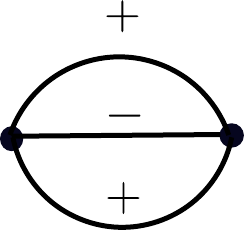}} \quad \quad \quad\quad  \quad\quad  \quad\quad  \quad\quad  \quad\quad  
\subfloat[][\label{figure:variousP:b}]
   {\includegraphics[width=.35\columnwidth]{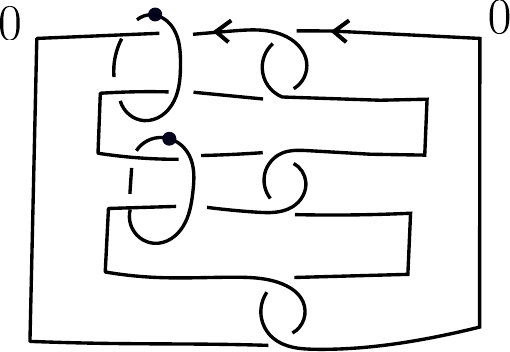}} \\
\subfloat[][\label{figure:variousP:c}]
   {\includegraphics[width=.35\columnwidth]{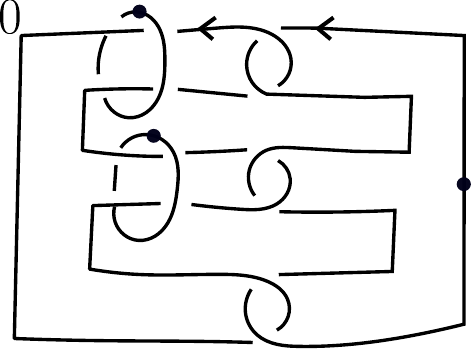}}\quad \quad \quad\quad  
\subfloat[][\label{figure:variousP:d}]
   {\includegraphics[width=.35\columnwidth]{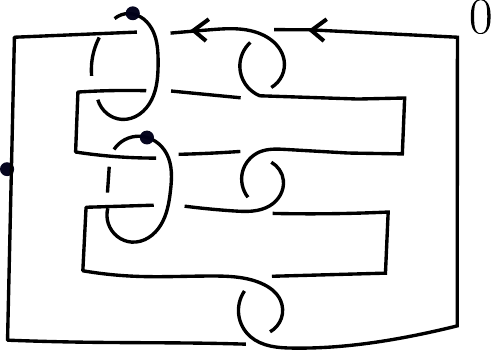}}
\caption{\label{figure:variousP}  A protocork plumbing graph (Figure~\ref{figure:variousP:a}),  its plumbing (Figure~\ref{figure:variousP:b}), its protocork (Figure~\ref{figure:variousP:c}) and the reflection of the protocork (Figure~\ref{figure:variousP:d}).}
\end{figure}

	
	

			
\subsection{Homological  properties of protocorks.}		\label{subsec:homological properties}
	\newcommand{\card}{k}
	The next proposition (Proposition~\ref{propProtocork}) summarizes the basic homological properties of protocorks and their boundary that we will need in the second part of the article. Some of these are used also in \cite{MorganSzabo99}, without a proof. 
	All the homologies below use $\Z$-coefficients when not specified otherwise.

	\begin{proposition}\label{propProtocork}Let $\Gamma$ be a connected protocork plumbing graph with sphere-number $n$ and multiset of edges $E$. 
	Let $N $ be a protocork associated to $\Gamma$ and denote by $Y := \partial N$ its oriented boundary. 
	Let  $X $ be a a closed, oriented smooth 4-manifold such that $X = M\bigcup_Y N$  where $M$ is a compact, oriented 4-manifold with $\partial M = -Y$. Then:	
		\begin{enumerate}
		\item  $H_1(N)\simeq \Z^{|E|-2n+1}$, $H_2(N) \simeq H_3(N)\simeq  (0)$, in particular $b^+(N) = 0$.
		\item $H_1(Y)$ is free and generated by cycles in the graph $\Gamma$, thus $b_1(Y) = |E|-2n+1$.
		 \label{propProtocork:isoH1}The inclusion map $i:Y=\partial N \to N$ induces  isomorphisms  $H_1(Y)\to H_1(N)$ and $H^1(N)\to H^1(Y)$.
		 \item  We can find a basis of  $H_2(Y)$ given by plumbing tori and  $H^1(Y)\otimes H^1(Y)\to H^2(Y) = \alpha \otimes \beta \to \alpha \cup \beta$ is the zero map. A fortiori, the triple cup product of $Y$ is trivial.
		\item $b^+(M) = b^+(X)$.

		\end{enumerate}
		
		\paragraph{} If in addition $H_1(X) = 0$, then
		\begin{enumerate}[resume]
		\item $H_1(M) =0$   in particular $H_1(Y)\to H_1(M)$ is trivial.
		\item $H^2(M) \to H^2(Y)$ is  surjective.
		\item $H^2(M) \simeq H^2(X) \oplus H^2(Y)$ not canonically. 
		\item Denote by $Spin^c(W)$ the isomorphism classes of \spinc structures for a manifold $W$. The restriction map 
				\begin{equation}\label{restrictionmap}
						r: Spin^c(X)\to Spin^c(M)\times Spin^c(N) = [\sstruc] \mapsto ([\sstruc|_M], [\sstruc|_N])
				\end{equation}		 
				is an injection. The only \spinc structure on $N$ is the trivial one up to isomorphism, hence $r$ establishes a correspondence between $Spin^c(X)$ and $Spin^c(M)$.
		\end{enumerate}
	\end{proposition}
	\begin{proof}
	\begin{enumerate}
	
	 \item
	 	We consider the Kirby diagram  of  \autoref{subsec:KirbyDiagrProtocorks}. Generators of  $H_1(N)$ are identified with meridians of  the dotted circles. Relations are given by the linking number of the 0-framed knots with the dotted circles. The latter is governed by $a_\Gamma(\vA_i,\vB_j) = \delta_{ij}$ which  forces all the generators that are not induced by the edges to be trivial. 
	This shows that $H_1(N) \simeq \Z^{|E|-2n+1 }$. 
	Inspecting the chain complex induced by the handle decomposition we see that the boundary operator $\partial_2: C_2(N)\to C_1(N)$ is injective, indeed it sends the generator induced by $\vA_i$ to the generator induced by $\vB_i$. Hence $H_2(N) = (0)$. Moreover 
	there are no 3-handles, therefore $H_3(N) = (0)$.

	\item  	 The Kirby diagram of $N$ induces a surgery presentation for $Y$ where all the knots are $0$-framed.
	Denote the meridian of the knots induced by $\vA_i, \vB_j, e$ for $i,j=1,\dots, n, e \in E\setminus T$ respectively by $\mu_{A_i}, \mu_{B_j}, \mu_e$.
	From the diagram we  see that  the linking matrix is given by $\ell k(\mu_{A_i}, \mu_{B_j}) = a_\Gamma(\vA_i,\vB_j) = \delta_{ij}$ and zero otherwise.
	By  \cite[Proposition 5.3.11]{GompfStipsicz}, $H_1(Y)$ is generated by the meridians of the knots 
	quotiented out by  relations of the form $0 = p_i \mu_i + \sum_j  q_j \ell k(\mu_i, \mu_j) \mu_j$ where $p_i/q_i$ is the framing coefficient of the meridian $\mu_i$.
	In our case $p_i/q_i = 0$ and the special form  of the linking matrix implies that $\mu_{A_i} = 0 = \mu_{B_j}$  while the $\mu_e$ are unrelated.
	This with the previous item shows that the inclusion map $H_1(Y)\to H_1(N)$ is an isomorphism from which follows also the assertion in cohomology.

		\item Notice that the  same argument used to compute $H_1(N)$, applies to $P_{1/2}(\Gamma)$ and allows us to identify (not canonically) the generators of $H_1(P_{1/2}(\Gamma))$ (and thus of $H_1(Y)$) with cycles in $\Gamma$ thanks to \ref{propProtocork:isoH1}.
		If we consider a basis for $H_1(Y)$ made up of these loops, then we can use as geometric duals 
				the plumbing tori  \cite{IntersectionRingOfGraphManifolds_Doig_Horn}.
				The plumbing tori do not intersect, hence the map $H_2(Y)\otimes H_2(Y)\to H_1(Y) = [S_1]\otimes [S_2]\mapsto [S_1\cap S_2]$ is trivial. 
This map is Poincar\'e dual to $H^1(Y)\otimes H^1(Y)\to H^2(Y) = \alpha\otimes \beta\mapsto \alpha \cup \beta$.

	\item   $b^+ (X)  = b^+(M) + b^+(N) + \mathrm{rank}\{\partial:H_2(X)\to H_1(Y)\}$  \cite[pg. 76]{KM}  where $\partial$ is the connecting morphism in the 
			Mayer-Vietoris long exact sequence. However $\mathrm{rank}{(\partial)} = \dim \mathrm{Ker}\{i_1+i_2: H_1(Y)\to H_1(M)\oplus H_1(N)\} = 0$ since $i_2:H_1(Y)\to H_1(N)$ is an isomorphism.
			\item Since $H_1(X) = 0$, continuing the long exact sequence we have $0\to H_1(Y)\to H_1(M)\oplus H_1(N)\to 0 $, since $H_1(Y)\simeq H_1(N)$ then $H_1(M) = 0$.
			Notice also that this implies that $H_1(Y)\overset{i_1}{\to} H_1(M)$ is trivial.
			\item $H^2(M)\to H^2(Y)$ is surjective  indeed  the triviality of $H_1(Y)\to H_1(M)$ implies by Poincar\'e duality that  $H^2(Y)\to H^3(M, \partial M) $ is trivial and that map fits into the long exact sequence of the pair $(M,\partial M)$: 
			$\dots \to H^2(M)\to H^2(Y)\to H^3(M, \partial M)\to \dots $. 
			\item $H^2(M) \simeq H^2(X) \oplus H^2(Y)$ not canonically.  Mayer-Vietoris implies that $0\to H^2(X) \to H^2(M)\oplus 0 \to H^2(Y)\to 0$ (the surjectivity comes from the above item while the injectivity from the surjectivity of $H^1(N) \overset{\sim}{\to} H^1(Y)$.
			\item Since $H^1(X) = 0$, \spinc structures over $X$ are classified by the first Chern class of the determinant bundle. The claim follows from the injectivity of the restriction map $H^2(X)\to H^2(M)\oplus H^2(N)$
			which can be seen  from the exact sequence $H^1(M)\oplus H^1(N)\overset{\sim}{\to} H^1(Y) \to H^2(X)\to H^2(M)\oplus H^2(N)$ since $H^1(N)\overset{\sim}{\to} H^1(Y)$ is an isomorphism.
			
			\end{enumerate}	
	\end{proof}

\subsection{The involutions $\tau$ and $\rho$.}\label{subsec:involutions}
	\begin{plan}We will show that if $(\Gamma, \Realiz)$ satisfies some assumptions we can construct orientation reversing involutions $\hat \rho_B,\hat \rho_A: P_{1/2}(\Gamma, \Realiz)\to P_{1/2}(\Gamma,\Realiz)$ and
	$\rho_B = \hat\rho_B|_{\partial P_{1/2}(\Gamma, \Realiz)}$, $\rho_A = \hat\rho_A|_{\partial P_{1/2}(\Gamma, \Realiz)}$.
	Informally, $\hat \rho_B$ ($\hat \rho_A$)  reflects  the $B$-spheres ($A$-spheres) while acting trivially on their normal disk and fixes the $A$-spheres ($B$-spheres) while reflecting their normal disks.
	Moreover if $\Gamma$ and $ \Realiz$ are  symmetric (Definition~\eqref{def:SymData}) we will construct an orientation preserving involution 
	$\hat\tau: P_{1/2}(\Gamma, \Realiz)\to P_{1/2}(\Gamma, \Realiz)$ restricting to some involution $\tau$ on the boundary. $\hat \tau$ also induces a  diffeomorphism $\hat \tau:P_{0}(\Gamma, \Realiz)\to P_1(\Gamma,\Realiz)$
	between the protocork and its reflection restricting to $\tau$ on the boundary.
	Informally $\hat \tau$ exchanges the $A$-spheres with the $B$-spheres.
	We remark that these diffeomorphisms  \emph{depend on the specific realization}, even though it is not reflected in the notation.
	We prove some properties about these involutions in Proposition~\ref{prop:existenceSymmetry}. The fact that $\rho_A,\rho_B$ act trivially on the first cohomology group will allow us to compute the index
	of the relevant operators later (see Proposition~\ref{prop:DimOfModuli}).
	The involution $\tau$ is relevant to our theory because plays a role analogous to the twisting map of corks in a sense that will be made precise in  in Subsection~\ref{subsec:ProtExotic4man}. For the sake of this subsection we limit ourself to 
	 Proposition~\ref{prop:tauNotExtends} which says that in analogy to what happens to corks' involutions, $\tau$ does not extend to the protocork (unless the protocork is trivial). 
	\end{plan}
	
\newcommand{\rD}{r_{\D^2}}
	\paragraph{Assumptions on $(\Gamma,\Realiz)$.}
	Denote by $S_0$  the one-point compactification of $\R^2$ equipped with the smooth structure  coming from the homeomorphism with $\SS^2$  induced by the stereographic projection $ \SS^2\setminus \{(0,0,1)\}\to \R^2$.
	
	Let $\Gamma$ be a protocork plumbing graph. We consider realization data $\Realiz$ where the model surfaces $S_v\simeq \SS^2$  are given by  $S_0$, the bundles $N_v $ are the trivial $\D^2$ bundles over $S_0$ and the embedding
	of the disks $\varphi_{v}$ sends $\bigsqcup_{i=1}^{d_v} \D^2(2)$ to the disks  of  radius $2$ centered in $(5k,0)\in \R^2$, $k=0,\dots, d_v-1$ in the obvious way using translations.
	Let $r:S_0\to S_0$ be the self-diffeomorphism induced by the reflection with respect to $\R\times \{0\}$ and denote by $\rD:\D^2\to \D^2$ the restriction of $r$ to $\D^2 = \overline B(0,1)\subset \R^2$.

	We also make an assumption about the smoothing datum, item \ref{itemSmoothingRealizationDatumDef}  in the definition of realization datum.
	For each vertex  $v$, and $i=1, \dots, d_v$, choose
	\begin{equation}
		c_{v,i}: \torus^2\times (-1,1)\times [0,1)\to S_v\times \D^2,
	\end{equation}
		  a collar of the torus $\varphi_v(\partial \D(1) \times \{i\})\times \partial \D^2$ obtained using the natural polar coordinates
	so that the $(-1,1)$ factor comes from the radial coordinate of $\varphi_v(\partial \D(1)\times\{i\})\subset \varphi_v(\D(2)\times \{i\})$ and the $[0,1)$ factor comes from the radial coordinate of $\partial \D^2\subset \D^2$. These collars induce a smooth structure on $P_{1/2}(\Gamma, \Realiz)$ as explained in the paragraph below about the smoothness of $\hat\rho_B$, indeed 
	given the collars we can introduce corners along the tori and perform gluing in a unique way \cite[Section 2.6, 2.7]{WallBook}.
	We will assume that the smoothings of the manifolds $M_\pm$ prescribed by the realization datum $\Realiz$ arise from a choice of the collars $c_{v,i}$s as above.

	\paragraph{Definition of $\hat\rho_B$ and $\rho_B$.}
	We define an  auxiliary involution $\psi$ on the collection of bundles $E(\Gamma,\Realiz)$,
	\begin{equation}\label{eq:defOfPsi}
		\psi: E(\Gamma,\Realiz) \to E(\Gamma, \Realiz)
	\end{equation}
	by defining
	\begin{equation}\label{eq:psi1action}
		\psi(b,f,\vA_i) = (b, \rD (f), \vA_i) \quad \quad \text{ for } (b,f,\vA_i)\in S_0\times \D^2\times\{\vA_i\}
	\end{equation}
	and 
	\begin{equation}\label{eq:psi2action}
		\psi(b,f,\vB_j) = (r(b),  f, \vB_j)  \quad \quad \text{ for } (b,f,\vB_j)\in S_0\times \D^2\times\{\vB_j\}.
	\end{equation} So in particular $\psi$ fixes the $A$-spheres.
	It is not difficult to check  that $\psi$ passes to the quotient $P_{1/2}(\Gamma, \Realiz) $  defining an involution $\hat \rho_B:P_{1/2}(\Gamma, \Realiz) \to P_{1/2}(\Gamma, \Realiz) $. We denote the induced involution of the boundary as $\rho_B$.
	 			
	 	\paragraph{Smoothness of $\hat{\rho}_B$.} 	 	
	 	The map $\hat{\rho}_B$ is clearly smooth outside of the plumbing tori, proving that is smooth over the tori is not immediate because of the smoothing of the corners 
	 	used to define the plumbing.  	To prove it we argue as follows.

	 	The smooth manifold $ P_{1/2}(\Gamma, \Realiz) $ can be described as the manifold obtained by the following procedure. Firstly, for every vertex $\vB_i$  in the $B$-part
		set $\tilde N_{\vB_i}:=\tilde S_0 \times \D^2$ where
		\begin{equation}
			\tilde S_0 = S_0 \setminus \varphi_{\vB_i}(\bigsqcup_{i=1}^{d_v} \mathrm{int}(\D^2(1))).
		\end{equation}
		
		Secondly,  for every vertex $\vA_i$  in the $A$-part, let $\hat N_{\vA_i}$ be the manifold obtained from $N_{\vA_i}$ by introducing a corner \cite[pg. 61]{WallBook}
		along the tori $ \varphi_{\vA_i}(\bigsqcup_{j=1}^{d_{\vA_i}}\partial \D^2(1))\times \partial \D^2$. In general this operation yields a structure of manifold with corners only up to diffeomorphism
		but since we have chosen collars $c_{\vA_i,j}$ of the tori, we obtain a precise structure.
		
		Thirdly, glue $E_B:= \bigcup_{i=1}^n \tilde N_{\vB_i}$ to $E_A:=\bigcup_{i=1}^n \hat N_{\vA_i}$ as prescribed by the edges  using again the collars  $c_{v,j}$ to obtain a  smooth structure over the glued region \cite[Section 2.7]{WallBook}. The result is a smooth manifold with a natural diffeomorphism to
		$ P_{1/2}(\Gamma, \Realiz) $.
			
		We can think of the map   $\psi$ as a map $E_B\bigsqcup E_A\to E_B\bigsqcup E_A$ preserving  $E_B$ and $E_A$.	A priori, the smoothness of the action $\psi$ over a neighbourhood of a plumbing torus in $E_A$ is not clear because we introduced corners in passing from $N_{\vA_i}$ to  $\hat N_{\vA_i}$.
		However,  using our choice of  collars $c_{v,j}$, we can see that the action of $\psi$ over such neighbourhood
		is conjugated to 
		\begin{spliteq}
			{\psi'\times id}:\torus^2\times [0,1)^2 \longrightarrow  \torus^2\times [0,1)^2
		\end{spliteq}
		where $\psi':\torus^2\to \torus^2$, $\psi(e^{i\theta_1}, e^{i\theta_2}) = (e^{i\theta_1}, e^{-i\theta_2})$; in particular is smooth and restricts to the identity on the normal cones.
		Similarly the action of $\psi$ on a neighbourhood of a plumbing torus in $E_B$  is conjugated to
		\begin{spliteq}
			{\psi''\times id}:\torus^2\times [0,1)^2 \longrightarrow  \torus^2\times [0,1)^2
		\end{spliteq}
		where $\psi(e^{i\theta_1}, e^{i\theta_2}) = (e^{-i\theta_1}, e^{\theta_2})$.
		These two maps glue  proving smoothness over a neighbourhood of the plumbing torus $\torus^2 \times (-1,1)\times [0,1)\hookrightarrow P_{1/2}(\Gamma, \Realiz)$.

		\paragraph{Definition of $\hat\rho_A$ and $\rho_A$.} To define $\hat \rho_A$ and $\rho_A$ we follow the same recipe, but exchanging the role of $A$ and $B$,  so that $\hat \rho_A$ fixes the $B$ spheres and $\hat \rho_B $ fixes the $A$ spheres.
	
	\begin{definition}[Symmetric realization datum.] \label{def:SymData}A realization  datum $\Realiz$ for a protocork plumbing graph $\Gamma$ is said to be \emph{symmetric} if  satisfies the hypothesis stated for $\hat \rho_B$ and in addition satisfies the following. For any vertex $\mathbf{v}$ and $n\in \{1,\dots, d_{\mathbf v}\}$,   let   $D^{\mathbf{v}}_n\subset S_{0}\times\{\mathbf{v}\} $ be the image of the $n$-th disk $\D^2\times\{n\}\subset \bigsqcup_{i=1}^n \D^2$ under  $\varphi_{\mathbf v}$,  then  $\Realiz$ satisfies
	\begin{enumerate}
		\item if an edge of the form $\varepsilon \simeq (i,i,\pm)$ identifies $D_{\varepsilon,1}$ with $D^{\vA_i}_n$ and $D_{\varepsilon,2}$ with $D^{\vB_i}_k$ then $k=n$ and
		\item \label{defSymDataItem2} if an edge of the form $\varepsilon \simeq (i,j,\pm)$ with $i\neq j$ identifies $D_{\varepsilon,1}$ with $D^{\vA_i}_n$ and $D_{\varepsilon,2} $ with $D^{\vB_j}_k$  then exists an edge $\varepsilon'= (j,i, \pm)$ identifying $D^{\vA_j}_k $ with $D^{\vB_i}_n$.
		
		\item For  any pair of edges $\varepsilon, \varepsilon'$ as in item \ref{defSymDataItem2}, the smoothings of $M_\pm$ (see item \ref{itemSmoothingRealizationDatumDef} in the definition of realization datum) are the same.
	\end{enumerate}	
	\end{definition}
	Notice that a protocork plumbing graph  admits this kind of realization data if and only if is symmetric.
	
	\paragraph{Definition of $\hat \tau$ and $\tau$.} Now suppose that $\Gamma$ is symmetric and let $\Realiz $ be a symmetric realization datum for $\Gamma$.
		We define an auxiliary involution on the collection of bundles $E(\Gamma,\Realiz)$,
	\begin{equation}\label{eq:defOfPsi2}
		\psi: E(\Gamma,\Realiz) \to E(\Gamma, \Realiz)
	\end{equation}
	by setting
	
	\begin{equation}
	\psi(b,f,\vA_i) = (b, f, \vB_i)   \quad \quad \text{ for } (b,f,\vA_i)\in S_0\times \D^2\times\{\vA_i\}
	\end{equation}
	and 
	\begin{equation}
		\psi(b,f,\vB_j) = (b,  f, \vA_j)  \quad \quad \text{ for } (b,f,\vB_j)\in S_0\times \D^2\times\{\vB_j\}.
	\end{equation}
	
	 So in particular $\psi$ exchanges the
	$A$-spheres with the $B$-spheres.
	It is not difficult to check, using our assumptions on $\Realiz$, that 	 $\psi$ passes to the quotient $P_{1/2}(\Gamma, \Realiz) $  defining a smooth involution $\hat \tau$. We will denote the induced involution on the boundary as $\tau$.

	\paragraph{}
	Since  $\hat \tau$  swaps the tubular neighbourhoods of the  $A$-spheres and $B$-spheres respecting their framings and the sugery operation modifies the manifold only in that neighbourhood, the involution $\hat 
\tau$  gives also a diffeomorphism $\hat \tau: P_{0}(\Gamma,\Realiz)\to P_{1}(\Gamma,\Realiz) $ that restricts to $\tau$ on the boundary.

	\begin{proposition}\label{prop:existenceSymmetry}
	Let $\Gamma$ be a protocork plumbing graph of sphere-number $n$ and $\Realiz$ a realization datum for $\Gamma$ satisfying the hypothesis stated at the beginning of this subsection. 
	Set  $Y := \partial P_{1/2}(\Gamma, \Realiz)$. Then 
	\begin{enumerate}
	\item $\hat{\rho}_B, \hat \rho_A:P_{1/2}(\Gamma, \Realiz)\to P_{1/2}(\Gamma, \Realiz) $  and $\rho_B, \rho_A:Y\to Y $ are orientation \emph{reversing},
	\item $\rho_B$ and $ \rho_A $ fix $H^1(Y; \Z)$ and commute.
	\end{enumerate}
	
	If in addition $\Gamma$ and $\Realiz$ are symmetric, then 
	\begin{enumerate}[resume]
	\item $\hat \tau:P_{1/2}(\Gamma, \Realiz)\to P_{1/2}(\Gamma, \Realiz)$, $\hat \tau: P_{0}(\Gamma, \Realiz)\to P_{1}(\Gamma, \Realiz)$ and $\tau: Y\to Y$ are orientation preserving,
	\item \label{prop:tauNotActTriviallyH1} If $\Gamma$ is symmetric and non-trivial, i.e.  exists $i\in \{1,\dots,n\}$ such that  $a_\Gamma(\vA_i,\vB_i) \not = r_\Gamma(\vA_i,\vB_i)$, then  $\tau$  acts non-trivially on $H^1(Y;\Z)$.  
	\item $\rho_B, \rho_A$ and $\tau$ generate an action of $D_8$, the dihedral group of 8-elements.
	\end{enumerate}
	\end{proposition}
	
	\begin{proof}
			\begin{enumerate}
				\item The auxiliary involution $\psi$ defined in \eqref{eq:defOfPsi} is orientation reversing, indeed on a tubular neighbourhood of the  $A$-spheres acts as a product of the identity (on the base) and a reflection (on the normal fibers). Consequently $\hat \rho_B$ is orientation reversing too. The proof for $\hat \rho_A$ and the induced map on the boundary is analogous.
				\item By Proposition~\eqref{propProtocork} it is sufficient to understand the behaviour   $\rho\in  \{\rho_A,\rho_B\}$ on the plumbing tori. 
				 $\rho$ preserve each  plumbing torus $\torus$    but changes its orientation, indeed we can see from  \eqref{eq:psi1action} and \eqref{eq:psi2action} that 
				 $(e^{i\theta_b},e^{i\theta_f})\in \torus\overset{\rho}{\mapsto} (e^{i\theta_b},e^{-i\theta_f})\in \torus$, therefore $\rho_*([\torus]) = -[\torus]\in H_2(Y)$.
			Since $\rho$ is orientation reversing, we obtain by Poincar\'e duality that 			 $\rho^*: H^1(Y)\to H^1(Y)$ is the identity.
			Moreover since $H_*(Y)$ has no torsion,   $\rho^*:H^2(Y)\to H^2(Y)$ is minus the identity and $\rho_*: H_1(Y)\to H_1(Y)$ is the identity.
			\item Follows from the fact that the auxiliary involution in \eqref{eq:defOfPsi2} is orientation preserving.
			\item Since $\Gamma$ is not trivial, there is an edge $\varepsilon\simeq (i,i,+)$ that does not belong to a spanning tree of $\Gamma$.
			Hence by Proposition~\ref{propProtocork}, the plumbing torus  $\torus$ associated to $\varepsilon$ is non-trivial in $H^2(Y)$.
			On the other hand, we see from the definition of $\tau$ that  $\tau(b,f, \vA_i) = (b,f,\vB_i) \sim_\varepsilon (f,b,\vA_i) $, hence $ \tau_*[\torus] = -[\torus]$. Since $\tau$ preserves the orientation
			the thesis follows.
 
			\item An explicit computation shows that the following relations are in place: 
				\begin{spliteq}
					& \rho_A ^2 = \rho_B^2 = \tau^2 = 1\\
					& \tau^{-1} \rho_B \tau = \rho_A \\
					& \rho_A \rho_B = \rho_B \rho_A.
				\end{spliteq}
			We can therefore define a group isomorphism from $G := \langle \rho_A, \rho_B, \tau \rangle = \langle \tau \rho_A, \tau\rangle $ to 
			\begin{equation}
			 D_8  = \langle r, s \ | \ r^4 = 1, \ s^2 = 1, \ srs = r^{-1}\rangle
			\end{equation}
			by sending $\tau \rho_A \mapsto r \in D_8$ and $\tau\mapsto s\in D_8$.
			\end{enumerate}					
	\end{proof}
	Regarding item \ref{prop:tauNotActTriviallyH1} in the above proposition, we also point out that  it is not difficult to construct examples with sphere-number larger than one  where $\tau$ 
	does not even preserve the plumbing tori. 
	
	We conclude this subsection by showing that, similarly to what happens to the cork's involutions,  $\tau$ does not extend to a diffeomorphism of $P_0(\Gamma,\Realiz)$.
	\begin{lemma}\label{lemma:gammanotslice}
		Let $\Gamma$ be a protocork plumbing graph with sphere-number $n$. 
		Let $i\in \{1,\dots, n\}$ be such that $r_\Gamma(\vA_i, \vB_i)>1$
		and let $\gamma^B_i \subset \partial P_0(\Gamma)$ be a circle meridional to the sphere $B_i$ in $P_{1/2}(\Gamma)$.
		Then $\gamma^B_i$ is not slice in  $P_0(\Gamma)$, i.e.   does not bound a disk in the protocork $P_0(\Gamma)$.
	\end{lemma}
	\begin{proof} We will add some 2-handles to $P_0(\Gamma)$ and show that the image of $\gamma_i^B$ in this new manifold is not slice. 
		Since adding 2-handles can only improve sliceness this will imply the thesis.
		With reference to the Kirby diagram constructed in Subsection~\ref{subsec:KirbyDiagrProtocorks}, 	let $C_1,\dots, C_N$, $N\in \N$ be the collection of dotted circles. 
		Since $r_\Gamma(\vA_i, \vB_i)>1$, we can  suppose without loss of generality that
		$C_1, C_2$ are dotted circles introduced by a pair of $+,-$ edges between $\vA_i$ and $\vB_i$. 
		We proceed by cancelling $C_3,\dots, C_N$ by adding  0-framed 2-handles meridional to each of them.
		Now we have reduced to the case where $\Gamma$ is the graph showed in \autoref{figure:variousP}  $i=1$, $d_i = 3$.
		In this case we can add a pair of 2-handles to $P_0(\Gamma)$ so that $(P_0(\Gamma), \gamma^B_1)$ is diffeomorphic to the pair $(W, \gamma)$ shown in Figure~\ref{fig:PA:e} where $\gamma$ is the black dashed curve.
		This is  \cite[Exercise 9.3.5]{GompfStipsicz}, since it is particularly relevant to us, we give  a proof in \autoref{fig:PA} for completeness.
		Now $W$ is the Akbulut cork  \cite{Akbulut91} and $\gamma$ is a meridional circle to the its dotted circle. $\gamma$ is not slice in $W$ (\cite{Akbulut91}, also  \cite[Theorem 9.3]{AkbulutBook}).
	\end{proof}
	\begin{proposition}\label{prop:tauNotExtends}  Let $\Gamma$ be a symmetric protocork plumbing graph of sphere-number $n$ and $\Realiz$ be a symmetric realization datum for $\Gamma$. Suppose that for some $i\in\{1,\dots, n\}$ $r_\Gamma(\vA_i, \vB_i)>1$ and let $Y:= \partial P_0(\Gamma, \Realiz)$. Then the involution $\tau: Y\to Y$ does not extend to a diffeomorphism of $P_0(\Gamma, \Realiz)$.
	\end{proposition}
	\begin{proof}
		Let $\gamma^A_i \subset Y$ and $\gamma^B_i\subset Y$	be circles meridional to the spheres $A_i$ and $B_i$ in $P_{1/2}(\Gamma, \Realiz)$ respectively.
		Then $\tau$ exchanges $\gamma^A_i$ and $\gamma^B_i$. $\gamma^A_i$ is slice in $P_0(\Gamma, \Realiz)$ while $\gamma^B_i$ is not by Lemma~\ref{lemma:gammanotslice}.
		Therefore $\tau$ cannot extend to a diffeomorphism of the protocork.
	\end{proof}

	\begin{figure}[p]
\centering
\subfloat[][\emph{Protocork} $P_0(\Gamma)$.\label{fig:PA:a}]
   {\includegraphics[width=.45\columnwidth]{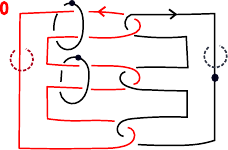}} \quad
\subfloat[][\emph{Adding 2-handles (blue)}.\label{fig:PA:b}]
   {\includegraphics[width=.45\columnwidth]{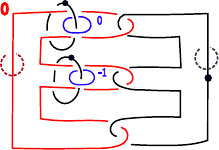}} \\
\subfloat[][\label{fig:PA:c}]
   {\includegraphics[width=.45\columnwidth]{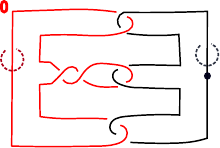}} \quad
\subfloat[][\label{fig:PA:d}]
   {\includegraphics[width=.45\columnwidth]{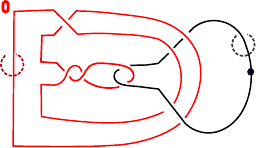}}\\
   \subfloat[][\emph{Akbulut cork}.\label{fig:PA:e}]
   {\includegraphics[width=.45\columnwidth]{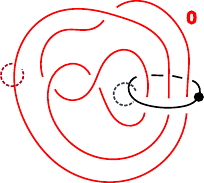}} \quad
\caption{The picture shows how we obtain the Akbulut cork \eqref{fig:PA:e} by adding two 2-handles to the protocork \eqref{fig:PA:a}. The dashed circles represent the meridian to the $A$-sphere (red) and to the $B$-sphere (black). In \eqref{fig:PA:b} we add the 2-handles (blue). To obtain  \eqref{fig:PA:c} we cancel the $0$-framed blue with the dotted circle, and similarly after sliding the red to the blue we cancel the $-1$-framed blue with the remaining dotted circle.
To obtain \eqref{fig:PA:d} we apply the Lemma shown in \autoref{figure:Lemma}. \eqref{fig:PA:d} is clearly isotopic to \eqref{fig:PA:e}.
\label{fig:PA}}
\end{figure}

		\begin{figure}
		\begin{center}
		\includegraphics[scale=1.2]{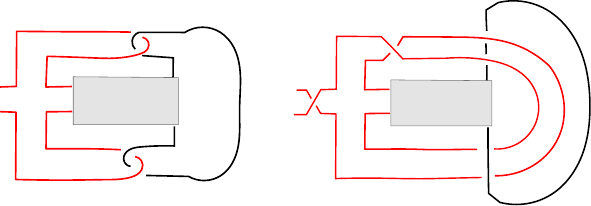}
		\end{center}
	\caption{\label{figure:Lemma} Lemma used to simplify the diagram in \autoref{fig:PA}. The diagram on the right is obtained by isotopying the red (performing an half rotation of the exterior fork).}
		\end{figure}
		
	
\subsection{Protocorks and exotic 4-manifolds.}\label{subsec:ProtExotic4man}
\begin{plan}
In this subsection we define, in analogy to the case of corks, the operation of \emph{protocork twist}. Firstly we define it for general protocorks, in this case the operation consists in cutting out the protocork and gluing  back its reflection. In case the protocork is symmetric we obtain the same result (up to diffeomorphism) gluing back the protocork via the involution of the boundary $\tau$.
Then  we show that any exotic pair is related by a protocork twist (Proposition~\eqref{prop:ExoticCompleteness}), this follows directly from the definition of protocork without much surprise as
the definition of protocork is forged exactly to this end.
Next we show that corks, and many other simply-connected twisting triples,   admit supporting protocorks (Proposition~\eqref{prop:CorkObstruction}).
\end{plan}

\newcommand{\emb}{\mathbf{e}}
\paragraph{Corks and exotic pairs.} A pair $(X_0, X_1)$ of oriented, closed, smooth, 4-manifolds is said to be an exotic pair if $X_0$ is (orientation preservingly) homeomorphic  but not (orientation preservingly) diffeomorphic to $X_1$.
It is a result of  of Curtis, Freedman, Hsiang and Stong \cite{CHFS} and Matveyev \cite{Matveyev} that if the pair is simply connected then $X_1$ can be obtained from $X_0$ by an operation called \emph{cork twist}  which we now review.
\begin{definition}\label{def:cork} An (abstract) \emph{ cork} is a pair $(W,\tau)$ where $W$ is a compact, \emph{contractible}, oriented, smooth 4-manifold  with $\partial W\neq \emptyset$ and $\tau:\partial W \to \partial W$ 
is an orientation preserving smooth involution such that  $\tau$  does \emph{not} extend to an orientation preserving diffeomorphism of $W$.
\end{definition}
We remark that by \cite{Freedman} the involution of a cork always extend to an homeomorphism of the full manifold thanks to the contractibility assumption.

The study of corks was pioneered  by Akbulut who found the first cork in \cite{Akbulut91}. This cork, which we call \emph{Akbulut cork}, is shown in Figure~\ref{fig:PA:e} where the dashed curves are exchanged by the involution $\tau$.  
The action of the involution $\tau$ is best seen by isotopying the diagram so to obtain the symmetric link \cite[Figure 10.4]{AkbulutBook}, then $\tau$ is induced  by a rotation which preserves the link. 
 We point out that this is not the only nomenclature in the literature as several authors call Akbulut corks what we call corks and Akbulut-Mazur cork  the Akbulut cork.

Now let $X$ be a closed oriented smooth  4-manifold $X$,  $(W,\tau)$ a  cork and let   $\emb: W\to X $ be  an orientation preserving embedding. 
							Then we can form another 4-manifold 
							\begin{equation}\label{eqDef:corktwist}
								X(W,\tau, \emb) := (X\setminus \mathrm{int} (\emb(W)))\bigcup_{\emb \circ \tau} W
							\end{equation}
							by cutting out $\emb(W)$ and gluing it back  via the map $\emb\circ \tau:\partial W\to \partial (X\setminus \emb	(W))$.
							We say that $X(W,\tau, \emb) $ is obtained from $X$ by a \emph{cork twist}.
							If $X$ and $X(W,\tau, \emb) $ are \emph{not} diffeomorphic we say that $(W,\tau,\emb)$ is \emph{effective} \cite{AKMREquivariantCorks}.

\newcommand{\Diffeop}{\mathrm{Diffeo}_+}
   \paragraph{Protocork twist.}  Given $M_0, M_1$ oriented manifolds we will denote by $\Diffeop(M_0,M_1)$ the set of orientation preserving diffeomorphism $M_0\to M_1$.
   \begin{definition}
   We call  \emph{twisting triple}  a triple $(M_0, M_1, \alpha)$ where $M_0, M_1$ are oriented,  compact 4-manifolds with boundary and $\alpha \in \Diffeop(\partial M_0, \partial M_1)$.
   Two twisting triples $(M_0, M_1, \alpha)$, $(M_0', M_1', \alpha')$ are said to be  \emph{isomorphic} if exist $\phi_i \in \Diffeop(M_i, M_i')$ for $i=0,1$, such that $\alpha =  \phi_1^{-1}|_{\partial M_1'}\circ\alpha'\circ\phi_0|_{\partial M_0}$.
   \end{definition}
   
   \begin{definition}[Abstract protocork]\label{def:AbstractProtocork} An \emph{abstract protocork} is a twisting triple $(P_0, P_1, \alpha)$ isomorphic to $(P_0(\Gamma, \Realiz), P_1(\Gamma, \Realiz), id)$ for some $\Gamma, \Realiz$ protocork plumbing graph and realization.
	\end{definition}

	Given an abstract protocork $(P_0,P_1, \alpha)$ and an embedding of $P_0$ in a closed $4$-manifold, the protocork twist operation consists in  cutting out $P_0$ and gluing in $P_1$ using the identification $\alpha$. This is made precise by the following definition.
	
	\begin{definition}[Protocork twist] \label{def:ProtocorkTwist}   Let $(P_0, P_1, \alpha)$ be an abstract protocork, let 
	 $X$ be a closed, oriented 4-manifold and let $\emb:~P_0\to X$ be a smooth, orientation preserving embedding. 
	The manifold
		\begin{equation}
				X(P_0, P_1, \alpha,\emb) := \left ( X\setminus \mathrm{int}(\emb(P_0)\right) \bigcup_{\emb\circ\alpha} P_1
		\end{equation}
		is said to be obtained from $X$ by twisting  $(P_0, P_1, \alpha)$ via $\emb$.
	\end{definition}

	In the symmetric case we can give more economical definitions resembling those of cork and cork twist. 
	We define a \emph{twisting pair} to be a pair $(M_0, \iota)$ where $M_0$ is a compact,  oriented 4-manifold with boundary and $\iota\in \Diffeop(\partial M_0)$.
   Two twisting pairs $(M_0, \iota)$ and $(M_0', \iota')$ are \emph{isomorphic} if exists $\phi_0 \in \Diffeop(M_0, M_0')$ such that $\iota = \phi_0^{-1}|_{\partial M_0'}\circ \iota'\circ\phi_0|_{\partial M_0}$.
	
	\begin{definition}[Abstract symmetric protocork]\label{def:AbstractSymProt}
	An \emph{abstract symmetric protocork} is a twisting pair $(P_0, \iota)$ isomorphic to $(P_0(\Gamma, \Realiz), \tau)$ for some
	 symmetric  protocork plumbing graph $\Gamma$ with symmetric realization datum $\Realiz$ and involution $\tau$ induced by $(\Gamma, \Realiz)$ as in Subsection~\ref{subsec:involutions}.
	\end{definition}

	Given an abstract symmetric protocork $(P_0, \iota)$ and an orientation preserving embedding 	$\emb: P_0\to  X$ in a smooth, oriented 4-manifold $X$, we define the manifold obtained  from $X$ by  twisting 
$(P_0, \iota)$ via $\emb$ to be 
	\begin{equation}\label{eqDef:protocorkTwist}
			X(P_0,\iota, \emb)  := X(P_0, P_0, \iota, \emb) = \left ( X\setminus \mathrm{int}(\emb(P_0)\right) \bigcup_{\emb\circ\iota} P_0.
	\end{equation}
	The analogy between \eqref{eqDef:protocorkTwist} and  \eqref{eqDef:corktwist} motivates the name protocork twist.
	
	\begin{remark} Since protocork twists are defined by gluing two manifolds along their boundaries, the smooth structure on the resulting manifold 
	is well defined only up to the choice of the collars of the boundaries and  different choices lead to diffeomorphic manifolds. 
	Therefore, if this choice is not made explicit we have to interpret the result just as a diffeomorphism type.
	This is the same level of ambiguity present in the definition of cork twist.
	\end{remark}

	\begin{remark}\label{rem:EffectOfRealization} The manifolds (up to diffeomorphism)  that can be produced from $X$ by twisting $(P_0(\Gamma, \Realiz), P_1(\Gamma, \Realiz), id)$ do not depend on the specific realization chosen, what matters is that we use the same realization for the protocork and its reflection.
	Indeed if $\Realiz'$ is a different realization,
	there exist diffeomorphisms $\varphi_i: P_i(\Gamma, \Realiz)\to P_i(\Gamma, \Realiz')$, $i=0,1$  that coincide over the boundary, i.e. $\varphi_0|_{\partial P_0(\Gamma, \Realiz)} = \varphi_1|_{\partial P_1(\Gamma, \Realiz)}$; these can be constructed by isotopying the surgery data. 
	
	Similarly,	 the manifolds that can be produced from $X$ by twisting an abstract symmetric protocork $(P_0, \iota)$, where $P_0\simeq P_0(\Gamma)$ depend only on the isomorphism class of $\Gamma$. 
	This was not immediate from the definition because the involution $\tau$ depends on the realization datum.
	\end{remark}

	\begin{proposition}\label{prop:ExoticCompleteness} For any exotic pair $(X_0, X_1)$ of simply connected, oriented, closed 4-manifolds  there exists an
	abstract symmetric protocork $(P_0, \iota)$ with connected protocork plumbing graph and an embedding $\emb: P_0\to X_0$ such that $X_1$ is diffeomorphic to $X_0(P_0, \iota, \emb)$.
	\end{proposition}
	The proof is along the lines of the  first part of the  argument of \cite{CHFS, Matveyev}, the difference is that we do not need to complete the subcobordism to obtain contractible ends.

	\begin{proof}
	
	The first part of the proof of the h-cobordism  theorem (\cite{SmaleHCob}, see also \cite{MilnorLecturesHCob}) ensures the existence of an h-cobordism $W^5:X_0\to X_1$ with the following properties: 
			\begin{itemize}
			\item $W^5$ is obtained by gluing simultaneously  $n\geq 0$   2-handles $\{h^2_j\}_{j=1}^n$ to $X_0\times \{1\}\subset X_0\times [0,1]$ obtaining  $W'$
			and then by gluing simultaneously $n$   3-handles $\{h^3_i\}_{j=1}^n$ to $\partial_+ W'$. In particular there are no handles of index $\lambda \neq 2,3$.
			\item Let the 4-manifold $W_{1/2}\subset W$ be  the middle 
			level of the cobordism obtained after attaching all the $2$-handles. 
			 Let $B_j\subset W_{1/2}$ be the belt sphere of $h^2_j$ and let $A_i\subset W_{1/2}$ be the attaching sphere of $h^3_i$.  Then  their algebraic intersection number in $W_{1/2}$
			 is equal to the Kronecker delta:	$B_j\cdot A_i = \delta_{i j}$ and the normal bundle of $B_j$, $A_i$ is trivial for any $i,j$.
			\end{itemize}

			It follows that the intersection graph of the spheres $\{B_j, A_i\}_{i,j}$ induces a protocork plumbing graph $\Gamma$.
			 If $\Gamma$ is not connected, then we can add  extra pairs $+,-$ of geometrically cancelling intersections by pushing one of the spheres into another.
			 In a similar manner if $\Gamma$ is not symmetric we can symmetrize it by adding some \emph{geometrically} cancelling intersections between the spheres $\{B_j, A_i\}_{i,j}$.
			 
			 Now let $P_{1/2}\subset W_{1/2}$ be a tubular neighbourhood of the intersection of the spheres. $P_{1/2}$ 
			 is clearly diffeomorphic to $P_{1/2}(\Gamma, \Realiz)$ for some realization datum $\Realiz$ by a diffeomorphism that sends $A_i\subset P_{1/2}$ to the sphere associated to the vertex $\vA_i$ in $P_{1/2}(\Gamma, \Realiz)$ and $B_j\subset P_{1/2}$ to  the sphere associated to the vertex $\vB_j$ in $P_{1/2}(\Gamma, \Realiz)$.

			Notice that $P_{1/2}$ is the middle level of a subcobordism $A^5\subset W^5$ obtained by flowing $P_{1/2}$ with the gradient of the Morse function of $W$ forward and backward in time (and adding the limiting critical points). Since the belt spheres of the $h^2_j$s and the attaching spheres of the $h^3_i$s lie in $P_{1/2}$, $A^5$ will contain all critical points of the Morse function and therefore $W\setminus \mathrm{int}(A)$ will be a trivial cobordism , i.e. diffeomorphic to $M\times [0,1]$ where $M$ is a compact $4$-manifold with boundary.
			
			It follows from Morse theory that :
			\begin{itemize}
			\item the incoming boundary of  $A$, $A_0\subset X_0$, is diffeomorphic to the manifold $P_0$
			 obtained from $P_{1/2}$ by surgering  the belt spheres $B_j$ of the 2-handles.
			\item Similarly, the outcoming boundary of $A$, $A_1\subset X_1$, is diffeomorphic to $P_1$, the manifold obtained from $P_{1/2}$ by surgerying the attaching spheres $A_i$ of the 3-handles.	
			\item  $X_0 \simeq M \bigcup P_0$ and $X_1 \simeq M\bigcup P_1$, where we can assume that $P_0$ and $P_1$ have the same boundary and the gluing maps are the same because
			surgery interests only the interior of $P_{1/2}$.
			\end{itemize}
			
			Clearly there are  diffeomorphisms $f_t: P_t \to P_t(\Gamma, \Realiz)$ for $t\in \{0,1/2, 1\}$ such that $f_t|_{\partial P_t}^{\partial P_t(\Gamma,\Realiz)}$ is independent of $t$, these can be constructed with an appropriate choice of $\Realiz$ or by isotopying the surgery data.
			Therefore we obtain    diffeomorphisms   $ X_0 \simeq M \bigcup P_0(\Gamma, \Realiz)$ and $ X_1 \simeq M\bigcup P_1(\Gamma, \Realiz)$.
			Since $\Gamma$ is symmetric,   we obtain the conclusion using the diffeomorphism $\hat \tau$ of Subsection~\ref{subsec:involutions}.
	\end{proof}	
	
	In Figure~\ref{fig:PA} we have seen how the Akbulut cork has a protocork embedded in its interior which determines
	the effect of the Akbulut cork twists. This suggests the following definition.
	\begin{definition}[Supporting protocork]
		A twisting triple $(C_0, C_1, f)$ is supported by an abstract protocork  $(P_0, P_1,\alpha)$
		if there exists embedding 
		$\emb: P_0\hookrightarrow \mathrm{int}(C_0)$ and diffeomorphism 
		\begin{equation}\label{eq:FInDefOfSupportingProtocork}
					F: (C_0\setminus \mathrm{int}(\emb(P_0))\bigcup_{\emb \circ \alpha} P_1\to  C_1
		\end{equation}
		 restricting to $f$ on the boundary.
	\end{definition}
	In particular, a cork $(C,f)$ is supported by a symmetric protocork $(P_0, \tau)$ if there exist an embedding $\emb: P_0\hookrightarrow \mathrm{int}(C)$ and $	F: C_1\to (C\setminus \mathrm{int}(\emb(P_0))\bigcup_{\emb \circ \tau } P_0$ diffeomorphism restricting to $f$ on the boundary.
	Notice that in this case, if  $C$ embeds in a 4-manifold  $X_0$, then the result of the cork twist, $X_1$ 	is obtained by twisting the supporting protocork $P_0$.
	The example of Figure~\ref{fig:PA} shows that the protocork of Figure~\ref{fig:PA:a} supports the Akbulut cork, indeed in this 
	case the involution of the protocork extends over the complement restricting to the involution of the cork on the boundary. 	

	It turns out that  any cork has a supporting symmetric protocork as corollary of the following.

	\begin{proposition}\label{prop:CorkObstruction}		Let $(C_0, C_1, f)$ be a twisting triple such that $\pi_1 (C_0) = 1$, $\partial C_0$ is connected and $f$ extends to an orientation preserving \emph{homeo}morphism $\tilde f: C_0\to C_1$. Then there exists a  symmetric protocork $(P_0, \tau)$ supporting $(C_0, C_1, f)$. Moreover we can assume that the map $F$ in \eqref{eq:FInDefOfSupportingProtocork} satisfies
	$(F^{-1})^*(x) = ({\tilde f}^{-1})^*(x)$ for every $x \in Im\left(H^2(C_0) \to H^2(C_0\setminus \mathrm{int}(\emb(P_0))\right)$.
	In particular, any cork $(C,f)$ is supported by a symmetric protocork.	
														 
	\end{proposition}

\begin{proof}
Since $f$ extends to an homeomorphism and  $C_0$ is simply-connected with connected boundary,  \cite{boyer_1986} implies that  there exists an $h$-cobordism $W:C_0\to C_1$ relative to $f$, 
in particular there is an identification: 
					\begin{equation}
						 \partial W\simeq  \bar C_0 \bigcup_{id} (\partial \bar C_0\times I) \bigcup_{f} C_1.
					\end{equation}

				We consider a Morse function $\cF:W\to \R$ such that $\cF(\partial_- W) = 0$, $\cF(\partial_+W) = 1$ and $\cF|_{\partial\bar C_0\times I} $ is the projection onto the second factor.
			Now we can repeat the  same argument of  the proof of Proposition~\eqref{prop:ExoticCompleteness}    obtaining  that there is a subcobordism $A\subset W$, $A:A_0\to A_1$, such that
			$W\setminus A$ is  a trivial cobordism and   $A_0 \simeq P_0(\Gamma, \Realiz)$, $A_1 \simeq P_1(\Gamma, \Realiz)$ with $\Gamma$ connected and symmetric.
			Moreover, setting  $M_0:= C_0 \setminus \mathrm{int}(A_0)$ and $M_1 := C_1 \setminus \mathrm{int}(A_1) $, we have that 
			\begin{align*}
				\overline{\partial_- W} = M_0\bigcup A_0 \simeq M_0\bigcup P_0(\Gamma, \Realiz)  
			\end{align*}
			and
			\begin{equation*}
							   \partial_+W = M_1\bigcup A_1 \simeq M_0\bigcup P_1(\Gamma,\Realiz)			
			\end{equation*}
			where the diffeomorphism $ M_0\bigcup P_1(\Gamma,\Realiz)  \overset{ \simeq}{\to}\partial_+W$ restricts to $f$ on $\partial C_0$ because of our choice of  $\cF|_{\partial \bar C_0\times I}$.
			Now, since $\Gamma$ is symmetric  we can construct the  diffeomorphism $M_0\bigcup_\tau P_0(\Gamma,\Realiz)\to C_1 $ of the thesis using the diffeomorphism $\hat \tau$ of Subsection~\ref{subsec:involutions}.
	To justify the last claim, we use that the $h$-cobordism $W$ constructed in  \cite[Prop. 4.2]{boyer_1986}  induces  $f_*:H_2(C_0)\to H_2(C_1)$ in homology and hence $f^*:H^2(C_1)\to H^2(C_0)$ in cohomology since $\pi_1(C_i) = 1$ for $i=0,1$.
\end{proof}

\begin{remark} In the specific case of a cork or more generally a twisting pair $(C,f)$ with  $C$ contractible,  we do not need to invoke \cite{boyer_1986}, as we can explicitely construct an $h$-cobordism relative to $f$. It is sufficient to find a \emph{contractible} smooth 5-manifold  $W^5$, with boundary $\partial W \simeq \bar C\bigcup_f  C$. Since $\bar C\bigcup_{f^{-1}} C \simeq_{\mathrm{TOP}} \Sphere^4$ there is, by Wall's theorem \cite{Wall}	 an  h-cobordism $\tilde W :\bar C\bigcup_f C  \to \Sphere^4 $, this is constructed just using handles of index $2$ and $3$ so it is  also simply connected.  We construct $W$  by capping $\Sphere^4\simeq  \partial_+ \tilde W$ with 	$\D^5$. 
\end{remark}

 \section{On the monopole Floer homology of protocork boundaries.}\label{sec:FloerHomologyBoundary}
\begin{plan}
The section is divided in four parts. The first subsection, Subsection~\ref{subsec:backgroundFloer} recaps the relevant background and notation from Kronheimer and Mrowka's book \cite{KM}. The reader accustomed with \cite{KM} may well skip it.

The second subsection, Subsection~\ref{subsec:MorseLikePerturbation}, describes in a general setting what we call Morselike perturbations, 
these are are perturbations of the Chern-Simons-Dirac functional  obtained by pulling back a Morse-Smale function defined on the torus of flat connections. 
We will use them to gain control over the reducible critical points. This idea is not new, in fact ad hoc applications of Morselike perturbations appear in several places in the literature as a device to pass from a Morse-Bott critical submanifold to a Morse one, e.g. \cite{BraamDonaldson}.  In particular, we borrowed the idea from Section 35.1 and Chapter 36  of \cite{KM}.  We need however some stronger results about them with respect to those used in the book; we develop them in full generality in Proposition~\ref{prop:MorseLikePerturbationPerp} and report the technical proof in \autoref{appendix:ProofMorseLike}.

The third and fourth subsection are specific to the case of $Y^3$ being a protocork boundary, and constitute the core of the section.
In Subsection~\ref{subsec:CriticalPoints}  we describe the relevant geometric setting over a protocork boundary $Y$ useful to define its Floer homology, we point out some key properties of the chain complex, we  define the Morgan-Szab\'o number and show  how it relates to the original definition.

In the last subsection Subsection~\ref{subsec:FloerHomologyCobordismSplitting},  we will prove Theorem~\ref{Intro:thm:Splitting} which provides a splitting of the Floer homology of $Y$. The proof will rely only on the functorial properties of Floer homology so it will be independent of the previous subsections. 
\end{plan}

\subsection{Background on generators of the chain complex.}\label{subsec:backgroundFloer}
\begin{plan}
In this section we will use the monopole Floer machinery developed by Kronheimer and Mrowka in their book \cite{KM}.
The aim of this subsection is to recap definitions and notation from \cite{KM} that we will need in the rest of the section.  We make no claim of originality.
\end{plan}

We recall that the abstract recipe for Floer homology is to construct a chain complex where the generators are the restpoints of a vector field (usually gradient of a functional)  and the
differential counts flowlines between them.
In the case of monopole Floer homology, the vector field we are interested in is the formal gradient of the Chern-Simons-Dirac  functional $\calL$ defined over the quotient configuration space ($\B(Y,\sstruc)$ in the following).
The quotient configuration space however is not a manifold in general due to the presence of reducible configurations. 
This issue is solved by Kronheimer and Mrowka by working on the blown-up quotient configuration space $\B_k^\sigma(Y,\sstruc)$ which is an Hilbert manifold with boundary. The gradient of $\calL$  admits a lift 
to the blown-up, $(\mathrm{grad}(\calL))^\sigma$,  which is used to define  the chain complex.
We briefly review how these objects are defined.

\begin{recap}[Classical 3D configuration spaces]
Let $Y$ be a closed, oriented Riemannian 3-manifold, and let $\sstruc$ be a \spinc structure on $Y$, with associated  spinor bundle $S\to Y$ and Clifford multiplication $\rho: TY\to \End(S)$.
We suppose that $\rho$ is compatible with the orientation of $Y$, i.e. if $e_1, e_2, e_3 $ is an oriented orthonormal frame then $\rho(e_1)\rho(e_2)\rho(e_3) = 1$. 
Given  a  \spinc connection $B$, we denote by $B^t$ the induced connection on the determinant bundle  $\Lambda^2S$ and by    $F_{B^t}\in \Lambda^2(Y, i\R)$ its curvature.   $D_B: C^\infty(Y;S)\to C^\infty(Y; S)$ will denote the \spinc \  Dirac operator induced by $B$.
Let $k>2$ be a natural number, $k$  will be the \emph{regularity parameter} of our Sobolev spaces of sections. In particular  $k>2$ will guarantee the embedding into the space of continuous functions.
By $\A_k(Y, \sstruc)$ we denote  the space of $L^2_k$ Sobolev \spinc connections on $S$. The \emph{configuration space} is defined as
\begin{equation}
	\Conf_k(Y, \sstruc) = \A_k(Y,\sstruc) \times L^2_k(Y; S)
\end{equation}
where $L^2_k(Y; S)$ denotes the Sobolev space of $L^2_k$-sections of the spinor bundle $S\to Y$. 
As a general rule, ommission of the subscript $k$ means that \emph{smooth} sections (or connections) are considered.
We  will denote by $\mathcal{T}_j = L^2_j(Y; iT^*Y\oplus S)$ the tangent space of $\Conf_k(Y, \sstruc) = \A_k(Y,\sstruc) \times L^2_k(Y; S)$ consisting of $L^2_j$-sections $j\leq k$ (see  \cite[Section 9]{KM}).
The \emph{gauge group} is defined as
\begin{equation}
	\G_{k+1}(Y)= \{ u \in L^2_{k+1}(Y;\C) \ | \ \ \norm {u(y)} = 1\  \forall y \in Y\}
\end{equation} with the subspace topology. Notice that the evaluation at a point makes sense because $k>2$ guarantees that $u\in L^2_{k+1}(Y)\subset C^0(Y)$ by Sobolev embedding theorem.
 $\G_{k+1}(Y)$ acts on $\Conf_k(Y,\sstruc)$  by
 \begin{equation}\label{eq:actionOfG}
 	u\cdot(B, \psi) = (B-\frac {du}{u}\otimes 1_S, u\ \psi)
 \end{equation}
 for $u\in \G_{k+1}(Y)$ and $(B,\psi) \in \Conf_k(Y,\sstruc)$, here $1_S\in \End(S)$ is the identity map, so that $\frac {du}{u}\otimes 1_S\in L^2_{k+1}(Y; \Lambda^1 Y\otimes \End(S))$.
Configurations that have trivial (non-trivial) stabilizer under the gauge group action are said to be irreducible (reducible). The reducible configurations are precisely those with identically zero spinor component.
The \emph{quotient configuration space} $\B_k(Y,\sstruc) := \Conf_k(Y,\sstruc) /\G_{k+1}(Y,\sstruc)$ is the quotient of the configuration space by the action of the gauge group. 
We will denote by $\torus\subset \B_k(Y,\sstruc)$ the torus of flat connections, i.e. 
\begin{equation}
	\torus := \{[(B,0)] \in \B_k(Y,\sstruc) \ | \ F_{B} = 0\},
\end{equation}
clearly $\torus\neq \emptyset$ only when $\sstruc$ is torsion and in this case  $\torus\simeq H^1(Y; \R)/ H^1(Y; \Z)$ is a torus of dimension $b_1(Y)$.
\end{recap}
\begin{recap}[Blow-up of configuration spaces] The configuration spaces $\Conf_k(Y,\sstruc)$  have the inconvenience that the action of the gauge group is not free and therefore $\B_k(Y,\sstruc)$ is not a manifold in general.  This issue is solved by Kronheimer and Mrowka by introducing blown-up configurations spaces (see Section 6 and 9 of \cite{KM}) defined as follows.
\begin{equation}
	\Conf_k^\sigma(Y,\sstruc) := \A_k(Y,\sstruc) \times \R_{\geq 0} \times \SS(L^2_k(Y;S)),
\end{equation}
where $\SS(L^2_k(Y;S))$ is the unit sphere with respect to the $L^2$-norm (notice it is not the $L^2_k$-norm).  We call elements with zero $\R_{\geq 0}$-component \emph{reducibles} and their complement \emph{irreducibles}.
The tangent bundle of  $\Conf_k^\sigma(Y,\sstruc) $ is denoted by $\cT^\sigma_k$, see \cite[Section 9]{KM} for an explicit definition.
$\G_{k+1}(Y)$ acts \emph{freely} on $\Conf_k^\sigma(Y,\sstruc)$ by 
\begin{equation}\label{eq:actionOfGsigma}
	u\cdot(B,r, \psi)  = (B-\frac {du}{u}\otimes 1_S, r, u\psi)
\end{equation}
for $u\in \G_{k+1}(Y), (B,r,\psi)\in \Conf_k^\sigma(Y,\sstruc)$.
The quotient space is denoted by $\B_k^\sigma(Y, \sstruc)$ and  is an Hilbert manifold with boundary \cite[Corollary 9.3.8]{KM}. The boundary consists of equivalence classes of reducible configurations. 
There is a projection map $\pi: \Conf_k^\sigma(Y,\sstruc)\to \Conf_k(Y,\sstruc)$ given by $(B,r, \psi)\mapsto (B,r\psi)$. $\pi$ is $\G_{k+1}(Y)$-equivariant, hence defines a map  between the quotient configuration spaces $\pi: \B_k^\sigma(Y,\sstruc)\to \B_k(Y,\sstruc)$ which is a diffeomorphism over the irreducible locus.
\end{recap}
\newcommand{\gradLift}{(\operatorname{grad}\calL)^\sigma}
\newcommand{\gradL}{\operatorname{grad}\calL}
\begin{recap}[Chern-Simons-Dirac functional and the lifted gradient]
Fix a reference connection $B_0\in \A_k(Y, \sstruc)$ then the  Chern-Simons-Dirac functional  $\calL: \Conf_k(Y,\sstruc)\to \R$ is defined as
\begin{equation}
	\calL (B,\psi) = -\frac 1 8 \int_Y (B^t - B_0^t)\wedge (F_{B^t}-F_{B_0^t})+ \frac 1 2 \int_Y \langle D_B \Psi, \Psi\rangle d\operatorname{vol}.
\end{equation}
It can be checked that $\calL$ does not depend on the choice of the connection $B_0$.
$\calL$ is not $\G_{k+1}(Y)$-invariant in general, however we will be concerned only with torsion \spinc structures and in this case $\calL$ will be  gauge invariant and thus will descend to a functional
$\calL: \B_k(Y,\sstruc)\to \R$. The formal gradient of $\calL$, is a map $\gradL: \Conf_k(Y,\sstruc)\to \mathcal{T}_{k-1}$ explicitely defined in  \cite[Eq. (4.3)]{KM}.
In principle, this is the vector field that we would like to use to define monopole Floer homology, however as said above, due to technical issues, the actual definition uses another vector field, $\gradLift: \Conf_k^\sigma(Y,\sstruc)\to \mathcal{T}_{k-1}^\sigma$, which is a lift of $\gradL$ to the blow-up.
$\gradLift$ is defined as follows  \cite[pg. 117]{KM}:
\begin{spliteq}
	\gradLift(B,r,\psi) = \begin{pmatrix}
								& \frac 1 2 *F_{B^t} + r^2\rho^{-1}(\psi\psi^*)_0\\
								& \Lambda(B,r,\psi) r\\
								& D_B\psi - \Lambda(B,r,\psi)\psi\\
								\end{pmatrix}\in \cT_{k-1}^\sigma
\end{spliteq}
where $\Lambda(B,r,\psi) = \langle \psi, D_B\psi\rangle_{L^2(Y)}$ and  $(\psi\psi^*)_0 $ denotes the traceless part of the hermitian endomorphism $\psi\psi^*$.
Contrarily to $\gradL$, the lifted gradient is not, in general,  the gradient of a functional defined on the blow-up.
\end{recap}
\begin{recap}[Taxonomy of critical points]
Although $\gradLift$ is not the formal gradient of a functional over the blow-up,
its restpoints are  called critical points since they generate the chain complexes, in analogy with Morse homology. According to \cite[Proposition 6.2.3]{KM}, if $(B,r,\psi)\in \Conf^\sigma_k(Y,\sstruc)$ is a critical point of $\gradLift$ then either
\begin{itemize}
\item $(B,r,\psi)$ is irreducible (i.e. $r\neq 0$), and $(B,r\psi) = \pi(B,r,\psi)$ is a critical point of $\gradL$ or
\item $(B,r,\psi)$ is reducible (i.e. $r= 0$) and $(B,0)  = \pi(B,r,\psi)$ is a critical point of $\gradL$ and $\psi$ is an eigenvector of $D_B$.
\end{itemize}
Thus over the irreducibles the critical points of $\gradLift$ and $\gradL$ coincide under $\pi$ while each reducible critical point of $\gradL$ can introduce infinitely many (one for each normalized eigenvector) critical points in the blowup.
Reducible critical points are those whose equivalence class lie in the boundary of $\B_k^\sigma(Y,\sstruc)$, we can thus divide them in two classes:
\begin{enumerate}
\item critical points relative to a \emph{positive} eigenvalue, called  \emph{boundary-stable} and
\item critical points relative to a \emph{negative} eigenvalue, called  \emph{boundary-unstable}.
\end{enumerate}
The name comes from the fact that   flowlines of $-\gradLift$ which are  not contained entirely in $\partial \B_k^\sigma(Y,\sstruc)$  either start from boundary-unstable  to arrive at an irreducible critical point or start from an irreducible to arrive at  a boundary-stable critical point.
\end{recap}

\newcommand{\gradLiftp}{(\operatorname{grad}\Lpert)^\sigma}
\newcommand{\gradLp}{\operatorname{grad}\Lpert}

\begin{recap}[Perturbations]
The Chern-Simons-Dirac functional may have degenerate critical points    \cite[Definition 12.1.1]{KM}) and furthermore the trajectories of $-\gradLift$ may  be singular moduli spaces.
This problem is soved in \cite{KM} by perturbing $\gradL$ with a generic  \emph{tame perturbation}. 
A tame perturbation   \cite[Definition 10.5.1]{KM}) is a continuous map 
\begin{equation}
	\pert:\Conf(Y,\sstruc)\to \mathcal{T}_0
\end{equation}  satisfying some conditions ensuring nice properties of the perturbed moduli spaces (e.g. compactness), we will denote by $\pert^0$ and $\pert^1$  the connection and spinor component of $\pert$ respectively.
 We can use $\pert$	 to perturb $\gradLift$ obtaining $\gradLiftp(B,r,\psi)$ as follows:
\begin{equation}\label{eq:perturbedLiftedGradient}
\gradLiftp(B,r,\psi) := \begin{pmatrix}
									& \frac 1 2 * F_{B^t} + r^2\rho^{-1}(\psi\psi^*)_0 + \pert^0(B,r\psi)\\
									&  \Lambda_\pert(B,r,\psi)r\\
									& D_B \psi + \tilde\pert^1(B,r,\psi) - \Lambda_\pert(B,r,\psi) \psi\\
								\end{pmatrix}	\in \mathcal{T}^\sigma_{k-1}
\end{equation}
where denoting by $\Differential_x$ the  Fr\'echet derivative at $x$,
\begin{equation}
	\tilde \pert^1(B,r,\psi)= 	\begin{cases}  \frac 1 r \pert^1(B,r\psi) & \text{ if } r\neq 0 \\
																\Differential_{(B,0)}\pert^1(0,\psi) &  \text{ if } r = 0, \\
										\end{cases}
\end{equation}
and $ \Lambda_\pert(B,r,\psi) = Re\left\langle \psi, D_B \psi + \tilde \pert^1(B, r,\psi) \right \rangle_{L^2}$. The classification of critical points explained in the previous paragraph makes sense even in the perturbed setting. In fact, $\gradLiftp$ is the lift of $\gradLp = \gradL +\pert$ and, similarly to $\gradLift$, the critical points of $\gradLiftp$ are either the lift of an irreducible critical point of $\gradLp$ or 
 triples $(B,0,\psi)$ with $\psi$ an eigenvector of 
 	\begin{equation}\label{def:perturbedDirac}
 		D_{B,\pert} :=  D_B + \Differential_{(B,0)}\pert^1(0,\cdot).
 	\end{equation} Thus it still makes sense to speak of boundary-stable and unstable critical points.
 The authors of \cite{KM} construct a \emph{large Banach space of perturbations} $\Pert$ in  \cite[ Section 11.6]{KM},  consisting of tame perturbations, which is used in several constructions as an input for the Sard-Smale theorem to obtain a non-degenerate functional and regular moduli spaces (Theorem~12.1.2 and  Theorem~5.1.1 in \cite{KM}).
 Despite the notation, a tame perturbation does not have to be the formal gradient of a function, however this is true for $\pert \in\Pert$, thus we have also a perturbed functional $\Lpert = \cL + f_\pert$ where $f_\pert: \Conf(Y,\sstruc)\to \R$ is a primitive of $\pert$.
 \end{recap}

\begin{recap}[Chain complexes]\label{recap:ChainComplexes} Section 22 of \cite{KM}  defines the chain complexes  $\check{C}, \hat{C}, \bar{C}$ giving rise to Floer homology groups. We briefly review how these are generated by the critical points of $\gradLiftp$. First of all, we choose an \emph{admissible} perturbation $\pert \in \Pert$    \cite[Definition 22.1.1]{KM}), in particular the critical points of $\gradLiftp$ are non-degenerate and the moduli spaces of trajectories are regular. For any $\Lambda =\{x,y\}$ 2-element set, $\Z\Lambda$ will denote the coefficient group 
\begin{equation}
	\Z\Lambda = \langle x,y \ | \ x= -y\rangle.
\end{equation}
Notice that a choice of a generator $x$ or $y$ establishes an isomorphism $\Z\Lambda \simeq \Z$, indeed we could have worked with $\Z$-coefficients  instead of using $\Z\Lambda$ but then we would have to choose an orientation for each generator of the complex.
Now let $\crit^o, \crit^s, \crit^u$ be the set of \emph{gauge-equivalence classes} of irreducible, boundary-stable and boundary unstable critical points of $\gradLiftp$ respectively and set
\begin{align}
	& C^o = \bigoplus_{[\gota]\in \crit^o} \Z\Lambda ([\gota]) 
	&  C^s = \bigoplus_{[\gota]\in \crit^s} \Z\Lambda ([\gota]) & 
	& C^u = \bigoplus_{[\gota]\in \crit^u} \Z\Lambda ([\gota]). 
\end{align}
where $\Lambda ([\gota])$ is a 2-element set of orientations of a moduli space associated to $[\gota]$  \cite[Section 20.3]{KM}.
In Section 22 of \cite{KM} are defined homomorphisms  $\partial^x_y:C^x\to C^y$ and $\bar\partial^x_y:C^x\to C^y$ for $x,y\in \{s,u,o\}$ between these complexes obtained  by counting $1$-dimensional moduli spaces  of trajectories (only reducible ones in the latter case) of $\gradLiftp$.
Now define
\begin{align}
	& \check C = C^o\oplus C^s
	&  \hat C = C^o\oplus C^u & 
	& \bar C = C^s\oplus C^u. 
\end{align}
The Floer homology groups $\HMto(Y,\sstruc), \HMfrom(Y,\sstruc), \HMbar(Y,\sstruc)$ are the homologies of the chain complexes $(\check C, \check\partial), (\hat C, \hat\partial), (\bar C, \bar\partial)$, where $\check \partial, \hat \partial, \bar \partial $ are constructed from   the abovementioned homomorphisms $\partial^x_y$, and $\bar \partial^x_y$. 
For a precise definition  we refer the reader to \cite[Section 22]{KM}. 
 Although as defined the Floer homology groups depend  on the perturbation $\pert$ and on the metric  there is a canonical isomorphism between the Floer homologies arising from a different choice of data (Riemannian metric and perturbation)  \cite[Corollary 23.1.6]{KM}. Thus the Floer homology groups are topological invariants of $Y$.
\end{recap}

\subsection{Morselike  perturbations.}\label{subsec:MorseLikePerturbation}
\newcommand{\Lf}{\cL + f}
\begin{plan}A \emph{Morselike perturbation} is a perturbation of $\cL$ that is  the pullback of a Morse function defined on the torus of flat connections $\torus$;  we give a formal definition below.
This  has the advantage of giving us a complete understanding of the reducible critical points and of the flow between them \emph{in the blow-down}.
In general such a perturbation is not admissible: critical points may be degenerate (even the reducible ones  in the blow-down if the Dirac operator is not invertible) and the moduli spaces of trajectories may be singular. We can use the theorems of Chapter 12 and Theorem 15.1.1 in \cite{KM} to obtain an admissible perturbation $\Lpert = \cL + f + f'$ where $\operatorname{grad} f' \in \Pert$, however, this would 
defy the purpose of the Morselike perturbation $f$ because the perturbation $f'$ may alter reducible critical points and the flow between them.
This issue is only apparent, indeed a slight modification  of the proofs of \cite{KM} shows that in the case of a Morselike perturbation,  we can sharpen the result of \cite{KM} and assume that the connection component of $\operatorname{grad} f' $   vanishes on the reducible locus at the cost of possibly perturbing slightly the Morse function. We state this as Proposition~\ref{prop:MorseLikePerturbationPerp} here and give a proof in  Appendix~\ref{appendix:ProofMorseLike}.
\end{plan}

\paragraph{Definition of Morselike perturbations.}
We continue with the notation of Subsection~\ref{subsec:backgroundFloer}, thus $(Y,g)$ will denote  a closed oriented Riemannian 3-manifold with torsion \spinc structure $\sstruc$.
Fix a reference  \emph{flat} \spinc connection $B_0\in \A_k(Y,\sstruc)$ so that $\A_k(Y,\sstruc) = B_0 + L^2_k(iT^*Y) $.
Define $P: \Conf_k(Y,\sstruc)\to \Conf_k(Y,\sstruc)$ by
 \begin{equation}\label{def:projP}
													P(B_0 + b\otimes 1_S, \psi) = (B_0+ (\PP_{\ker \Delta}b)\otimes 1_S, 0),
\end{equation}
where  $\PP_{\ker \Delta}:L_k^2(Y, i\R)\to L^2_k(Y,i\R)$  is  the $L^2$-projector on the space of harmonic 1-forms i.e. 
 the kernel of the Hodge Laplacian $\Delta:L^2_k(iT^*Y)\to L^2_{k-2}(iT^*Y)$. The existence of $\PP_{\ker \Delta}$ is ensured by Hodge theory, notice that 
the image of $P$ is the set of flat connections.  The map $P$ passes to the quotient defining a  retraction
\begin{equation} \label{def:RetractionToTorus}
 p_\torus: \B_k(Y,\sstruc) \to \torus. 
\end{equation}

 \begin{definition} A functional $f:\Conf(Y,\sstruc)\to \R$ of the form 
 		\begin{equation}
				f (x) :=  f_\torus([P(x)])\quad \text{ for all } 	x\in\Conf(Y,\sstruc)
		\end{equation}
		 where $ f_\torus:\torus\to \R$ is a Morse-Smale function, is called \emph{Morselike perturbations}.	
 
 \end{definition}

Notice that the \emph{reducible} critical points of $\cL + f$ are the critical points of $f$ and  $\torus$ 	is invariant under the gradient flow of $\cL +f $, the reducible trajectories (in the blow-down) are precisely the gradient trajectories of the Morse function $f$.  The gradient of $f$ is a tame perturbation.

\newcommand{\PertZero}{\Pert^{\perp}}
\begin{definition}[$\PertZero$] \label{def:PertZero} We denote by $\PertZero < \Pert$ the Banach subspace of perturbations that vanish on the reducible locus. In formulae, $\pert\in \PertZero$ if $\pert(B,0) = 0$ for all $B\in \A(Y,\sstruc)$.
\end{definition}

\begin{remark} We could have defined $\PertZero$ by requiring that only the connection component vanishes on the reducible locus, i.e. 
$\pert^0(B,0) = 0$ for all $B\in \A(Y,\sstruc)$.
Indeed equivariancy of tame perturbations implies  that $\pert^1(u\cdot(B,\psi)) = u\pert^1(B,\psi)$ for any $u\in\G(Y), (B,\psi)\in \Conf(Y)$,
thus the spinor component always vanishes over the reducible locus; its derivative instead may be non-zero and perturb the equations in the blown-up model.
In light of this, a perturbation $\pert \in \Pert$ with primitive $f_\pert:\Conf(Y,\sstruc)\to \R$ belongs to $\PertZero$ if and only if $f_\pert$ is constant on the reducible locus.
\end{remark}
The  subspace $\Pert^\perp$ of perturbations vanishing on the reducible locus is clearly a proper closed subspace of $\PertZero$. 
These perturbations are important to us, because they do not alter the reducible trajectories in the (classical) configuration space.

\begin{proposition} \label{prop:MorseLikePerturbationPerp} Let $f$ be a Morselike perturbation induced by $ f_\torus:\torus \to \R$.  Then there is a residual subset $U_0\subset \PertZero$ such that for any $\mathfrak{u}_0 \in U_0$,  after possibly enlarging $\Pert$, 
there are a closed subspace  $Z<  \Pert$ depending on $\mathfrak{u}_0$ and 
 a neighbourhood of zero  $ U \subset Z$   such that  $\mathfrak{u}\in  U$ implies that 
 \begin{enumerate}
 \item  the critical points of $\left(\grad(\cL + f)+\mathfrak{u}_0\right)^\sigma $ are non-degenerate,
 \item  $\mathfrak{u}$ vanishes in a neighbourhood of the critical points of $\grad(\cL + f)+\mathfrak{u}_0 $,
 \item $\left(\grad(\cL+f) +\mathfrak{ u}_0+ \mathfrak{ u}\right)^\sigma$ has the same critical points of $\left(\grad(\cL + f)+\mathfrak{u}_0\right)^\sigma$,
 \item   $\mathfrak u = \grad( h_\torus([P(\cdot)]) + \mathfrak{q}^\perp $ where  $\mathfrak{q}^\perp \in \PertZero$  and  $ f_\torus+ h_\torus$ is  Morse-Smale.
 \item  In addition,   the set of  $\mathfrak{u} \in U  $ for which  $\grad f +\mathfrak{ u}_0+ \mathfrak{u}$ is admissible is residual in $U$.
 \end{enumerate}
\end{proposition}
The proof is given in \autoref{appendix:ProofMorseLike}. Notice that since $u_0 \in U_0\subset \PertZero$, the \emph{reducible} critical points of $\grad(\cL+f) +\mathfrak{ u}_0+ \mathfrak{u}$ are the critical points of $\grad(\cL+f)$ in the blow-down.

\subsection{Critical points for $Y$ a protocork boundary.}\label{subsec:CriticalPoints}
\begin{plan}In this subsection we will describe our setting, in particular the data (metric and perturbation) that we will use and the critical points. We will finally draw some conclusions on the boundary maps and relative degree (Proposition~\ref{prop:bargr}) and define an homomorphism that will come in handy 
in the proof of Theorem~\ref{thm:DeltaInHMred}.
The main ingredients will be a metric constructed by  Morgan and Szab\'{o} and Morselike perturbations as described in Subsection~\ref{subsec:MorseLikePerturbation}.
The last part of the section deals with Morgan-Szab\'{o}'s number, introduced in Theorem~\ref{thm:DeltaInHMred} item \ref{thm:DeltaInHMred:item:NMS}, we show that the original definition of \cite{MorganSzabo99} coincides
with the one that we use here.
\end{plan}

\paragraph{Setting for Floer homology.}
The following result due to Morgan and Szab\'{o}  will be fundamental.
		\begin{lemma}[Lemma 2.1 in \cite{MorganSzabo99}]  \label{lemma:DiracInjective} Let $Y$ be the boundary of a protocork and let $\sstruc$ be the trivial \spinc structure. 
				Then there  exists a Riemannian metric on $Y$ such that for \emph{any} flat  \spinc connection $B$  the kernel of the Dirac operator
					$D_B: \Gamma(S)\to \Gamma(S)$ is trivial.
		\end{lemma}
		Notice that the the conclusion of the above lemma is true in general for manifolds admitting a metric with positive scalar curvature. The proof  of \autoref{lemma:DiracInjective} does not show  that $Y$
		has a metric with positive scalar curvature (which is false unless the protocork is trivial) but is  based instead on a neck stretching argument.
		
		The metric of  Lemma~\ref{lemma:DiracInjective} can be described as follows. Recall that $Y$ is obtained by gluing two trivial $\SS^1$-bundles 	over  $\Sigma_A, \Sigma_B$,  surfaces diffeomorphic to $\SS^2$ minus some open disks. 
		The bundles $\Sigma_A\times \SS^1, \Sigma_B\times \SS^1$ are  endowed with the product metric where
		the metric over $\Sigma_A$ and $\Sigma_B$ is such that the boundary circles have unit length and restricts to a product metric in a collar 	of the boundary. In this way,  there are necks isometric to $[-T,T]\times \torus^2$ embedded in $Y$ for some $T>0$.
		The authors of \cite{MorganSzabo99} show that  if $T$ is large enough then the thesis of Lemma~\ref{lemma:DiracInjective} follows.
		
		The form of the metric allows us to assume that  the involutions $\rho_A,\rho_B$ and $\tau$ (when defined) are  isometries, indeed it is not difficult, using the realization datum described in Subsection~\ref{subsec:involutions}, to construct
		metrics that satisfy the above hypothesis and for which the involutions are isometries.

\newcommand{\nirr}{N_{\text{irr}}}
\newcommand{\nred}{N_{\text{red}}}
\newcommand{\Mred}{M^{\text{red}}}
\newcommand{\Red}[1]{\mathcal{R}_{#1}}

For the rest of this section,  $(Y,g)$ will be the  boundary of a  protocork  endowed with a metric $g$ given by \ref{lemma:DiracInjective}.  Let $\sstruc$ be a trivial \spinc \ structure on $Y$, this is unique up to isomorphism since $H^*(Y;\Z)$ has no torsion. 
We will use a perturbation $\pert \in \Pert$ of the form
\begin{equation}\label{defAdmissiblePerturbationqOverY}
		\pert  := \grad f + \pert'
\end{equation}
where $f$ is Morselike perturbation $f$ and $\pert'\in \PertZero$. We will also assume without loss of generality that 
\begin{assumption}\label{assumptionPerturbationEnum}

\begin{enumerate}
	\item $\pert$ is admissible,
	\item $f_{\torus}$ has a unique critical point of Morse index $b_1(Y) = \dim \torus$ on the torus of flat connections $\torus$, and
	\item \label{assumption2OnPert} $\pert'$ is so small in $\Pert$-norm that for any  $[(B,0)]\in \torus$, the perturbed Dirac operator $D_{B,\pert}$ (c.f. \eqref{def:perturbedDirac}) is invertible. 
	\end{enumerate} 
\end{assumption}

The existence of such $\pert$ is ensured by Proposition~\ref{prop:MorseLikePerturbationPerp}  observing that, thanks to  Lemma~\ref{lemma:DiracInjective}, the unperturbed Dirac operator $D_B$ is invertible for all $[(B,0)]\in \torus$ and invertibility is an open condition. We denote the primitive of $\pert'$ by $f'$ and set 
\begin{equation}
	\Lpert = \cL + f + f'.
\end{equation}

				\paragraph{Critical points.}
 				The critical point set  of $\Lpert $ in $\B(Y,\sstruc)$ is  finite and consisting of  $\nirr$  irreducibles
 				and  $\nred$ reducibles  $[\gota_1], \dots, [\gota_{\nred}]$, where  $\gota_i \in \torus$ is a flat connection and a critical point of $f_\torus$.  We will assume   $[\gota_{\nred}]$ to be  the unique critical point of $f_\torus$ of  maximal Morse index.
 				
 				On the blow-up  $\B^\sigma(Y,\sstruc)$, the zero locus  of $(\mathrm{grad}\Lpert)^\sigma$ consists of 
 				\begin{itemize}
 				\item 	$[\gotb_1],\dots, [\gotb_{\nirr}] \in \B^\sigma_k(Y,\sstruc)$, the lift of the irreducibles 
 				\item 	and a tower of reducibles for each $[\gota_i]$, $i=1,\dots \nred$.
			\end{itemize} 				  				
 				
 				Recall that the elements of the tower correspond to eigenvalues of the Dirac operator $D_{B_i,\pert}$ (which has simple spectrum)  where $B_i$ is the  connection associated to $\gota_i$. 
 				These eigenvalues are real and if we order them in increasing order: 

 				\begin{equation}\label{eq:ordering_eigenvalues}
 					\dots < \lambda_{-2}< \lambda_{-1} < 0 < \lambda_{0} < \lambda_1 < \dots
 				\end{equation}

 				We can denote  the  elements of the tower over $[\gota_i]$ by $[\gota_{i,j}]$
 				where $j \in \Z$ indicates that  					$[\gota_{i,j}]$ corresponds to the eigenvalue $\lambda_j$.
 				So in particular $[\gota_{i,0}]$ corresponds to the \emph{smallest positive} eigenvalue, $[\gota_{i,-1}]$  to the larger negative one etc. Therefore $[\gota_{n,i}]$ will be boundary-\emph{unstable}
 				 for $i<0$ and boundary-\emph{stable} for $i\geq 0$.
 				 
 				 \begin{remark}\label{remark:EigenvaluesSpectralFlow} Thanks to \ref{assumption2OnPert} of Assumption~\ref{assumptionPerturbationEnum}, the spectral flow of any path of perturbed Dirac operators $D_{B_t,\pert}$, $[(B_t,0)]\in \torus$ is zero. Therefore we will  enumerate the eigenvalues so that the numbering is consistent with the spectral flow, i.e. the $0$-th eigenvalue  of $D_{B_0,\pert}$ and the $0$-th eigenvalue of $D_{B_1,	\pert}$
				are the endpoints of a continuous curve $\{(\lambda(t),D_{B_t,	\pert} )\ |  \ \lambda(t) \in \sigma (D_{B_t,	\pert})\}$ where $B_t$ is flat for each $t$.
 				 \end{remark}				
								
			\newcommand{\pointa}{[\gota_{n,i}]} 				
			\newcommand{\pointb}{[\gota_{m,j}]} 				
			\newcommand{\indf}{\mathrm{index}_f }
		
		\paragraph{On the moduli spaces between reducible critical points.}
		Consider two reducible critical points in the blow-up $[\gota_{n,i}], [\gota_{m,j}]$, then we have
		$\Mred_z(\pointa, \pointb)$, the moduli space of \emph{reducible} trajectories on $\R\times Y$ limiting to
		$\pointa$ for $t\to -\infty$ and to $\pointb$ for $t\to +\infty$ of homotopy class $z$  \cite[Definition 13.1.1]{KM}.
		Adopting the convention of  \cite[Chapter 16.6]{KM}, we denote by 
		\begin{equation}
			\bar\gr (\pointa, \pointb) = \dim \Mred_z(\pointa, \pointb)
		\end{equation}
		the virtual dimension of the moduli space. Notice that since the \spinc structure is torsion $\bar\gr (\pointa, \pointb) $ does not depend from the  path $z$, so we have omitted it from the notation. 
		It can be shown that $\bar \gr$ is additive, meaning that $\bar\gr (\pointa, \pointb) + \bar\gr (\pointb, [\gota_{r,k}]) = \bar\gr (\pointa, [\gota_{r,k}])$.
		
		\paragraph{} For a reducible critical point $\pointa$, denote by 
		\begin{equation}\label{defIndfMorse}
			\indf \pointa \in \N
		\end{equation} the index of the Morse function $f_\torus$ at $[\gota_i] \in \torus$.
		\autoref{remark:EigenvaluesSpectralFlow}, and the Morselike perturbation $f$ imply the following properties analogous to the case of  $\SS^1\times \SS^2$  (cf. \cite[Section 36.1]{KM} and  the related  \cite[Lemma 33.3.3]{KM}). 
		\begin{proposition}\label{prop:bargr}
			\begin{enumerate}
				\item $\bar \gr(\pointa, \pointb) = \indf \pointa - \indf \pointb + 2(i-j)$.
				\item $\bar \partial^s_u = 0$, 
			\end{enumerate}
		\end{proposition}

		Notice also that despite this,  the proof given for $\SS^1\times \SS^2$ that  $\bar \partial^u_s = 0$ does not go through because the metric of  $Y$ does not have non-negative scalar curvature.

\newcommand{\proj}{\mathrm{proj}}
{ \newcommand{\ntop}{N^{\mathrm{red}}}
\paragraph{The homomorphism $\proj$.}Recall from \autoref{recap:ChainComplexes} that the group of chains defining  $\HMfrom(Y)$ and $\HMbar(Y)$ are respectively $\hat C = C^o\oplus C^u$  and $\bar C = C^s\oplus C^u$.  
Suppose for simplicity  to have chosen  an orientation for all the generators of all the chain complexes involved, so that the group of chains will have coefficients in $\Z$ instead of $\Z\Lambda([\mathfrak{c}])$.
Then we define  $\Z$-homomorphisms
\begin{align}
	\mathrm{proj}: \HMfrom(Y) \to \Z &
		& \mathrm{proj}: \HMbar(Y)\to \Z
\end{align}
by projecting along the component $[\gota_{\ntop, -1}]\in C^u$. The next lemma guarantees that $\proj$ is well defined in homology.
\begin{lemma}\label{lemma:HMred1}
Consider a generator $[c]\in \hat C = C^o\oplus C^u$, then $\langle \hat \partial [c], [\gota_{\ntop,-1}]\rangle = 0$ i.e. $\hat \partial [c] $ has no component along $[\gota_{\ntop,-1}]$.
The same holds in $\bar C$, i.e. if $[c]\in \bar C$ is a generator then $\langle \bar \partial [c], [\gota_{\ntop,-1}]\rangle = 0$.
Consequently reading  the component along  $[\gota_{\ntop,-1}]$ gives  well defined $\Z$-homomorphisms $\proj: \HMfrom(Y)\to \Z, \proj:\HMbar(Y)\to \Z$ 
\end{lemma}
\begin{proof}[Proof of Lemma~\ref{lemma:HMred1}]Recall that  the boundary map of $\hat C$ is given by
	\begin{equation*}\hat \partial = 
		\begin{bmatrix}	&	\partial^o_o & \partial^o_u\\
							& -\bar \partial^s_u \partial^o_s & - \bar\partial^u_u - \bar \partial^s_u\partial^u_s\\
		\end{bmatrix}: C^o\oplus C^u\to C^o\oplus C^u.
	\end{equation*}
	By Propostion~\ref{prop:bargr}, $\bar\partial^s_u= 0$, therefore  the only trajectories to $[\gota_{\ntop,-1}]$ come from elements of $C^u$.
	On the other hand if $[c] = [\gota_{k,-i}]\in C^u$ then an application of Proposition~\ref{prop:bargr} shows that 
	$\bar \gr([\gota_{k,-i}], [\gota_{\ntop,-1}]) =\mathrm{ index}_f([\gota_k]) -b_1 + 2(1-i)$. On the other hand, $\mathrm{ index}_f([\gota_k]) -b_1  = 0$ because otherwise we would not have a  flow of the connection component. Consequently 
	$\bar \gr([\gota_{k,-i}], [\gota_{\ntop,-1}]) = 2(1-i)\neq 1$ for $i\neq 1$. Therefore there are no $1$-dimensional moduli spaces of trajectories from $[\gota_{k,-i}]$ to $ [\gota_{\ntop,-1}]$.
\end{proof}
Another  useful corollary of the above lemma is the following sufficient condition for belonging to the reduced homology.
\begin{lemma}\label{lemma:HMred2}
	If $\sigma = (\sigma^o, k[\gota_{\ntop, -1}])  \in  \hat C = C^o\oplus C^u$, $k\in \Z$ is a cycle and $[\sigma]  \in HM^{red}(Y)$, then $k = 0$, therefore $\proj(HM^{red}(Y)) = 0$.
\end{lemma}
	\begin{proof}[Proof of Lemma~\ref{lemma:HMred2}] By definition, 
		$HM^{red}(Y) = \ker p_*:\HMfrom(Y)\to \HMbar(Y)$,  where $p$ is the  anti-chain map
		\begin{equation*} p = 
		\begin{bmatrix}	&	\partial^o_s &  \partial^u_s\\
							& 0  & 1\\
		\end{bmatrix}: C^o\oplus C^u\to C^s\oplus C^u.
	\end{equation*}
	Therefore $p(\sigma) = (\partial^o_s \sigma^o +\partial^u_s[\gota_{\ntop, -1}],  k[\gota_{\ntop, -1}])$ is equal to $\bar \partial \eta $ for some $\bar \eta \in \bar C$.
	Now \autoref{lemma:HMred1} implies that $k= 0$.
	\end{proof}
}

\paragraph{The Morgan-Szab\'{o} number.} In view of the notation just introduced, the Morgan-Szab\'{o} number  defined in 
 Theorem~\ref{thm:DeltaInHMred} \ref{thm:DeltaInHMred:item:NMS} is equal to 
	\begin{equation}\label{eqDef:NMS_num}
		n_{MS} = 2 + \max_{i,j\in \{1,\dots, \nirr\}} \gr([\gotb_i], [\gotb_j]).
	\end{equation}
	Notice that this number is dependent on the metric $g$ and the perturbation $\pert$.

	Morgan and  Szab\'{o}'s paper \cite{MorganSzabo99} came out before that monopole Floer homology was developed by Kronheimer and Mrowka \cite{KM}, 
	therefore the original definition of $n_{MS}$, which is given by the definition of $\mathcal k$ in \cite[Proposition 2.3]{MorganSzabo99}, is slightly different.
	Now we want to explain why the two definitions are equivalent.
	
	In \cite{MorganSzabo99}, the Chern-Simons-Dirac functional is perturbed using an exact $2$-form $\mu = d\eta \in \Omega^2(Y, i\R)$, i.e.
	\begin{equation}
			\Lpert_{MS}(B, \psi) := \cL(B, \psi) +\frac 1 4 \int_Y (B-B_0)^t\wedge d\eta
	\end{equation} 
	 The resulting 
	 3-monopole equations are:
	\begin{spliteq}\label{MSnumber_aux_1}
		& \frac 1 2 *(F_{B^t} - d\eta) +\rho^{-1}(\psi \psi^*)_0 = 0\\
		&  D_{B}\psi = 0 \\
	\end{spliteq}
	for $(B,\psi) \in \Conf(Y,\sstruc)$. 
	The $1$-form $\eta$ is chosen in such a way that the irreducible solutions are non-degenerate and the moduli spaces of trajectories are regular.
	We will also suppose without any loss of generality that $\eta$ has  $L^2_k$-norm so small that  the perturbed Dirac operator $\psi \mapsto D_{B}\psi + \frac 1 2\rho(\eta)\psi$ is still injective
	for any flat $B$.
	Then the authors of \cite{MorganSzabo99} define $\mathcal{k}$  as $2$ plus the maximum relative degree (i.e. the dimension of the space of trajectories of $-\grad \Lpert_{MS}$) between two irreducible solutions. 
	Notice that $\Lpert_{MS} - \cL$ is not a perturbation of the type considered in the beginning of Subsection~\ref{subsec:CriticalPoints}, as can be easily seen  from the fact that the reducible solutions
	are precisely a translation of the torus of flat connections.

	\begin{proposition}
	There exists a perturbation $\pert$ of the form \eqref{defAdmissiblePerturbationqOverY} satisfying Assumption~\ref{assumptionPerturbationEnum} 	 such that the number $n_{MS}$ as defined in \eqref{eqDef:NMS_num} is equal to $\mathcal k$.
	\end{proposition}
	\begin{proof}
	Define the diffeomorphism $P_\eta : \Conf_k(Y,\sstruc)\to \Conf_k(Y,\sstruc)$ as 
	\begin{equation}
		 P_\eta(B, \psi) = ({B-\frac 1 2\eta\otimes 1_S}, \psi).
	\end{equation}
	Then $P_\eta$ establishes  a diffeomorphism from  the moduli space of solutions of \eqref{MSnumber_aux_1} to the moduli space of solutions of 
	\begin{spliteq}\label{MSnumber_aux_2}
		& \frac 1 2 *F_{B^t}  +\rho^{-1}(\psi \psi^*)_0 = 0\\
		&  D_{B}\psi + \frac 1 2\rho(\eta)\psi = 0, \\
	\end{spliteq}
 $(B,\psi) \in \Conf(Y,\sstruc)$, i.e. the critical points of 
 \begin{equation}
 	\Lpert_1(B,\psi) := \cL(B, \psi)  +\frac 1 4 Re\langle \rho(\eta) \psi, \psi\rangle_{L^2(Y)}.
 \end{equation}	
 	In addition $P_\eta$ preserves the relative degree between two solutions.	
 	Indeed the relevant Fredholm operators differ by post-composition with a translation and hence have the same index.

	Now we consider a cylinder function of the form $f \sum_{[\alpha]} \beta_{[\alpha]}$ where $f$ is a chosen Morselike perturbation, $[\alpha]$ ranges over \emph{irreducible} critical points  and 
	$\beta_{[\alpha]}$ is a cylinder function that vanishes in a neighbourhood of $[\alpha]$ and is equal to $1$ outside of a slightly larger neighbourhood disjoint from the reducible locus. 
	These functions are constructed as in the paragraph preceding the definition of $\Pert_f$ in the proof of Proposition~\ref{prop:MorseLikePerturbationPerp}.
	
	We claim that exists $\epsilon_0>0$ such that  $\forall \epsilon \in (0, \epsilon_0)$ the perturbed functional  
	\begin{equation}
		\Lpert_{2, \epsilon} =  \Lpert_1 + \epsilon f \sum_{[\alpha]} \beta_{[\alpha]}
	\end{equation}	
	 has as reducible critical points the critical points those $f$	and 	as irreducible critical points those of $\Lpert_1$. By contradiction  suppose $\epsilon_n\to 0$ and let $(B_n, \psi_n) \in \Conf(Y)$
	 be \emph{irreducible } critical points of $\Lpert_{2,\epsilon_n}$ that are not critical points of  $\Lpert_1$. Then, by the properness property of tame perturbations  \cite[Proposition 11.6.4]{KM}, 
	up to applying a gauge transformation or passing to a subsequence we can assume that $(B_n , \psi_n)\to (B_\infty, \psi_\infty)$ critical point of  $\Lpert_1$.
	The configuration $(B_\infty, \psi_\infty)$ cannot be irreducible because, thanks to the bump function $\beta_{[\alpha]}$, each irreducible critical point of $\Lpert_1$ has a neighbourhood 
	where it  is the only critical point of $\Lpert_{2,\epsilon_n}$. Thus $(B_\infty, \psi_\infty)$ is reducible,  $\psi_\infty= 0$ and $B_\infty$ is flat.
	Consider the normalized sequence $\Phi_n = \psi_n/\norm{\psi_n}_{L^2(Y)}$, these unit spinors satisfy the elliptic equation $ D_{B_n}\Phi_n + \frac 1 2\rho(\eta)\Phi_n = 0 $.
	Thanks to elliptic regularity we can extract a converging subsequence $\Phi_n \to \Phi_\infty$ to a unit spinor in the kernel of  $ D_{B_\infty}+ \frac 1 2\rho(\eta)$, but the latter operator
	is injective by assumption. This concludes the proof of the claim.

	We choose  $ \epsilon \in (0, \epsilon_0)$ and put $ \Lpert_{2} := \Lpert_{2, \epsilon}$. 
	The relative degree between the irreducible critical points of $\Lpert_{2}$  is the same of  $\Lpert_1 $ because  the relative Fredholm operator varies by a compact perturbation. 

	It is possible that the reducible critical points of $\Lpert_2$ are degenerate  \emph{in the blow-up} . We can remedy this by adding a small perturbation 
 	$\pert_3\in \PertZero$ vanishing in a neighbourhood $\cO$ of the \emph{irreducibles}  critical points (cf. the definition of $\PertZero_{\cO}$ in \autoref{appendix:ProofMorseLike} with the caveat that therein $\cO$ is also a neighbourhood of reducible critical points).
 	If the perturbation is small enough  no irreducible critical points are introduced in the blow-down (the reducible critical points are not affected by a perturbation in $\PertZero$). 
 	This follows from  \cite[Proposition 11.6.4]{KM} and an application of the implicit function theorem exploting   that if ${\grad \Lpert_2(B,0) = 0}$, $\pert_3 \in \PertZero$ then $\grad \Lpert_2(B,0) + \pert_3 (B,0) =0$ and that 
 	reducible critical points are non-degenerate  in the blow-down, i.e. the relevant operator is an isomorphism when acting over the Coulomb slice at the critical point.
 	
 	 We denote the perturbed functional as $\Lpert_3 = \Lpert_2 + f_{\pert_3}$, where $f_{\pert_3}$ is a primitive of $\pert_3$.
	
	After possibly enlarging the large Banach space of perturbations $\Pert$, we see that the tame perturbation associated to $\Lpert_3$ takes the form: 
	$\grad (\epsilon f) + \mathfrak{u}_0$, where $\mathfrak{u}_0\in \PertZero$  and $\epsilon f$ is Morselike. 
	This is almost of the kind of perturbations required, except for the fact that the moduli spaces of trajectories might not be regular. 
	This issue can be solved by adding another small perturbation  $\mathfrak{u} =\pert_{\epsilon f}+\pert''$  as in Lemma~\ref{app:Lemma2_0MorseLike} in the proof of Proposition~\ref{prop:MorseLikePerturbationPerp}.	

	This shows that $\mathcal k$ as defined  \cite[Proposition 2.3]{MorganSzabo99} is equal to $n_{MS}$ as defined in \eqref{eqDef:NMS_num} using the perturbation $\pert := \grad (\epsilon f) + \mathfrak{u}_0 + \mathfrak{u}$.
 	\end{proof} 

\subsection{A splitting theorem for  Floer homology.}\label{subsec:FloerHomologyCobordismSplitting}			
\begin{plan}The aim of this subsection is to prove Theorem~\ref{Intro:thm:Splitting}. We will construct two cobordism $W$ and $Q$
such that $Q\circ W$ is trivial, this will provide us with a first  splitting, then the vanishing of the triple cap product (Proposition~\ref{propProtocork}) will allow us to characterize the two factors. The proof will not rely on the other parts of this section.
\end{plan}

	\newcommand{\sumb}{b_1}
	\paragraph{Notation.}
	Throughout this subsection, $\Gamma$ is a given protocork plumbing graph with sphere-number $n$, $\Realiz$ a given realization datum for $\Gamma$ and we denote by $N_0, N_1 $ and $Y$  respectively the realization $P_0(\Gamma,\Realiz)$, its reflection $P_1(\Gamma, \Realiz)$ and their common boundary $\partial P_0(\Gamma,\Realiz)$.
	We will also adopt the shorthand notation
		 \begin{equation}
		 	b_1 := b_1(Y),
		 \end{equation} and $n M := M\#\dots\#M$ to denote the $n$-fold connected sum of a $3$-manifold $M$ with itself, thus $\sumb \SS^1\times \SS^2 = \#_{i=1}^{b_1(Y)}(\SS^1\times \SS^2)$.
	Also for a cobordism $W:M_0\to M_1$ we will denote by $\partial_- W = M_0$ the incoming boundary (with opposite orientation) and by $\partial_+ W = M_1$ the outcoming boundary.
	\paragraph{The cobordism $W$.} We construct a cobordism $W: \sumb \SS^1\times \SS^2\to Y$. Start with a Kirby diagram $\mathcal D$ for $N_0$ of the type described in Subsection~\ref{subsec:KirbyDiagrProtocorks}.
	The diagram $\mathcal D$ will consists of a link with components of three types:
	\begin{itemize}
	\item  the $0$-framed knots $\alpha_1,\dots, \alpha_n$, relative to the vertices $\vA_i$, $i=1,\dots, n$, 
	\item the dotted circles $\beta_1,\dots, \beta_n$  relative to the vertices $\vB_i$, $i=1,\dots, n$, and 
	\item the dotted cicles $\mu_1,\dots, \mu_{b_1}$ associated to edges in the complement of a spanning tree of $\Gamma$.
	\end{itemize}
	 Considering the collection of the $\mu_i$s alone induces a surgery presentation for $\sumb \SS^1\times \SS^2$, indeed when looking at the boundary of a handlebody induced by a Kirby-diagram the dotted circles contribute as $0$-framed knots.
	 Then we construct $W$ in two steps: firstly we  add $n$ 1-handles to $(\sumb\SS^1\times \SS^2)\times \{1\}\subset (\sumb\SS^1\times \SS^2)\times [0,1]$ obtaining $W'$. 
	 The outcoming boundary $\partial_+ W'$ has a surgery presentation given by considering only the $\mu_i$s and the $\beta_i$s in $\mathcal{D}$ and regarding them as $0$-framed.
	 The second step  consist in gluing $n$  2-handles to $\partial_+ W'$, with reference to the surgery presentation just described, the attaching curves of the $2$-handles will be given 
	 by the $\alpha_i$s and will be $0$-framed. 
	 We notice also that  $W$ is an homology cobordism because the $1$-handles are algebraically cancelled by the $2$-handles.

	 \paragraph{The cobordism $Q$.} We will construct a cobordism $Q:Y\to \sumb\SS^1\times \SS^2$ such that the composition $Q\circ W: \sumb\SS^1\times \SS^2\to \sumb\SS^1\times \SS^2$ is a trivial cobordism, i.e. $Q\circ W \simeq (\sumb\SS^1\times \SS^2)\times I$. Consider the surgery presentation for $Y$ induced by the diagram $\mathcal{D}$.
	 We construct $Q$ in two steps. Firstly for each $i=1,\dots, n$,  we add a $0$-framed $2$-handle meridional to $\beta_i$ to $(\sumb\SS^1\times \SS^2)\times \{1\}\subset (\sumb\SS^1\times \SS^2)\times [0,1]$ obtaining a cobordism $Q'$. 
	 A surgery presentation for $\partial_+ Q'$ is given by the surgery presentation for $Y$ induced by the diagram $\mathcal{D}$, together with the attaching circles (meridional to the $\beta_i$s) of the $0$-framed 2-handles we just add.
	 Denote by $\gamma_i$, $i=1,\dots, n$ a loop in $\partial_+ Q'$ meridional to $\alpha_i$.
	 As a second step, for each $i=1,\dots, n$ we glue  a 3-handle to $\partial_+ Q'$ in such a way that,  its  attaching sphere meets $\gamma_i$ exactly in one point.
	Such a sphere exists in $\partial_+ Q'$, because one can slide $\alpha_i$ over the newly added $0$-framed $2$-handles meridional to the $\beta_j$s	  so that $\alpha_i$ becomes a $0$-framed unknot unlinked from the rest of the diagram having $\gamma_i$ as one of its meridians.

	\begin{lemma}\label{lemma:Qleftinverse} $Q$ is a left inverse of $W$, i.e. $Q\circ W: \sumb\SS^1\times \SS^2\to \sumb\SS^1\times \SS^2$ is a trivial cobordism.
	\end{lemma}
	\begin{proof} The cobordism $Q\circ W$ is constructed from $\sumb\SS^1\times \SS^2\times [0,1]$  by attaching handles to $\sumb\SS^1\times \SS^2\times \{1\}$ in four steps by repeating the process
	described to construct $W$ and then $Q$.
	Now it is sufficient to note that by construction the addition of the $2$-handles of $Q$ will cancel the $1$-handles of $W$ and the $3$-handles will cancel the $2$-handles of $W$.
	We are thus left with a cobordism with no handles hence trivial. 
	\end{proof}
	
	Recall that, as stated in \cite[Theorem 3.4.3]{KM}, Floer homology defines a covariant functor from the cobordism category to the category of $\ZU$-modules. In particular, we have $\ZU$-homomorphisms
	\begin{spliteq}
	& \HMfrom(W) : \HMfrom(\sumb\SS^1\times \SS^2) \to \HMfrom(Y)\\
	& \HMfrom(Q) : \HMfrom(Y) \to \HMfrom(\sumb\SS^1\times \SS^2)
	\end{spliteq}
	such that, thanks to Lemma~\ref{lemma:Qleftinverse},
		\begin{equation}\label{eq:CompositionIsId}
				\HMfrom(Q)\circ\HMfrom(W) = \HMfrom(Q\circ W) = \HMfrom((\sumb\SS^1\times \SS^2)\times I) = \hat{\mathrm{id}},
		\end{equation}
	the identity $\HMfrom(\sumb\SS^1\times \SS^2)\to \HMfrom(\sumb\SS^1\times \SS^2)$.
	With analogous maps $\HMto(W)$, $ \HMto(Q)$  and $\HMbar(W)$, $ \HMbar(Q)$ composing to $\check {\mathrm{id}}$ and $\bar{\mathrm{id}}$ respectively.

	\newcommand{\Sone}{\sumb\SS^1\times \SS^2}
	Since $\Sone$ has a metric with positive scalar curvature by  \cite[Proposition 36.1.3]{KM} the Floer homology  groups  of $\Sone$ are isomorphic as $\ZU$-modules to 
	\begin{spliteq}	\label{eq:FloerHomologyofConnSumofS1S2}
	& \HMfrom(\Sone) \simeq \Lambda^*(\Z^{b_1})\otimes_{\Z} \ZU\\
	& \HMto(\Sone) \simeq \Lambda^*(\Z^{b_1})\otimes_{\Z} \Z[U^{-1}, U]/\Z[U]\\
	& \HMbar(\Sone) \simeq \Lambda^*(\Z^{b_1})\otimes_{\Z} \Z[U^{-1}, U]].
	\end{spliteq}
	We will denote  by 
	\begin{align}
	& \hat{G} < \HMfrom(Y)	& \check{G} < \HMto(Y)  & &  \bar{G} < \HMbar(Y) &
	\end{align}
		The image of $\HMfrom(W), \HMto(W)$ and $\HMbar(W)$ respectively.	
		
		Before proceding to the next proposition, we recall that for any $3$-manifold $M$,  there is a long exact sequence \cite[pg. 52]{KM}
		\begin{equation}\label{ExactSequenceinFloerHomology}
		\cdots \to  \HMto(M) \overset{j_*}{\to} \HMfrom(M) \overset{p_*}{\to}\HMbar(M) \to \dots,
		\end{equation}
		and the reduced homology  $\HMred(M) < \HMfrom(M) $ is  defined as the kernel of the map $p_*: \HMfrom(M)\to \HMbar(M)$   \cite[Definition~3.6.3 ]{KM}. For any $3$-manifold, the reduced homology has finite rank.
	
		The next proposition  clearly implies Theorem~\ref{Intro:thm:Splitting}.
	\begin{proposition}\label{prop:SplittingOfHomology}
	Let $\sstruc$ be the the unique (up to isomorphism) torsion   \spinc structure of $Y$.
	The  cobordism map $\HMfrom(W)$ has image contained in $\HMfrom(Y,\sstruc)$, preserves the absolute $\Q$-grading and has 
	 a left inverse. The same applies to $\HMto(W)$ and $\HMbar(W)$, 
	 hence $\hat{G}, \check G, \bar G$ are isomorphic to the $\ZU$-modules on the right of \eqref{eq:FloerHomologyofConnSumofS1S2}.
	In addition,  we have that 
	\begin{spliteq}
	& \HMfrom(Y) \simeq \hat G \oplus \HMred(Y)\\
	& \HMto(Y) \simeq \check G \oplus \HMred(Y)\\
	& \HMbar(Y) \simeq \bar G
	\end{spliteq}
	
	With respect to this splitting, the map $j_*:\check G \oplus \HMred(Y)\to  \hat G \oplus \HMred(Y)$ is given by $j_*(g,q) = (0,q)$ and the map $p_*:\hat G \oplus \HMred(Y)\to \bar G$ 
	is given by $p_*(g,q) = p^{\Sone}_*(g)$ where $p_*^{\Sone}$ is   conjugated to $p_*: \HMfrom(\Sone)\to \HMbar(\Sone)$ via $\HMfrom(W)$ and $\HMbar(W)$.
	\end{proposition}
	\begin{proof} The Floer homology of $\Sone$ is supported in its torsion \spinc structure, $\sstruc_{\Sone}$ \cite[Proposition 36.1.3]{KM} and since $W$ is an homology cobordism there is a unique  \spinc structure over $W$ extending $\sstruc_{\Sone}$ (up to isomorphism).  By considering a trivial \spinc structure we see that  it must   restrict to $\sstruc_0$ over $Y$. To see that $\gr^{\Q}$ is preserved it is enough to compute the degree of the map  \cite[Equation (28.3)]{KM}
	\begin{equation}
	\frac 1 4 (c_1^2-\sigma (W)) - \frac 1 2 (\chi(W) + \sigma(W) + b_1(\Sone)- b_1(Y)) = 0.
	\end{equation}
	We proceed to prove the splitting part of the proposition.
	$Q$ provides the left inverse for the cobordism maps induced by $W$ (see Lemma~\ref{lemma:Qleftinverse} and \eqref{eq:CompositionIsId}).
	Consequently there are submodules $\hat Q, \check Q, \bar Q$ given by the kernel of $\HMfrom(Q), \HMto(Q), \HMbar(Q)$ such that
	\begin{align}
	&  \HMfrom(Y)=\hat{G}  \oplus \hat Q	& \HMto(Y) = \check{G} \oplus \check Q & &   \HMbar(Y) = \bar{G}\oplus \bar Q.&
	\end{align}
	Firstly we will show that  $\bar{Q}= \{0\}$.
	The diagram  
	\begin{equation}
		\HMbar(\Sone)\overset{\HMbar(W)}{\longrightarrow} \HMbar(Y)  \overset{\HMbar(Q)}{\longrightarrow} \HMbar(\Sone)
	\end{equation}	 
	is, up to isomorphism, 
	\begin{equation}\label{GGQG}
		\bar G\hookrightarrow \bar G \oplus \bar Q \to \bar G
	\end{equation}	 
	where the first map is the inclusion into the first summand and the second map the projection to the first summand.
	Now we use that, as showed in Proposition~\ref{propProtocork}, the triple cup product of $Y$ vanishes, therefore   \cite[pg. 687-688]{KM} implies that
	$\HMbar(Y)\simeq \Lambda^*(\Z^{b_1})\otimes_{\Z} \Z[U^{-1}, U]]$.
	
	The maps of  \eqref{GGQG} can be promoted to $\Z[U^{-1}, U]]$-module homomorphisms. 
	Tensoring by $\R$ so that $\R[U^{-1}, U]]$ is a PID, and $\HMbar(Y)\otimes \R$ is  a free $\R[U^{-1},U]]$-module of rank $2^{b_1}$, we see that $\bar Q \otimes \R = 0$, because  the maps of \eqref{GGQG} tensored with $\R$ have to be isomorphisms for dimensional reasons.  
		This also forces the maps of \eqref{GGQG} to be isomorphisms over $\Z[[U]]$ because if $\bar Q\neq 0$ also $\bar Q \otimes \R\neq 0$ as $\HMbar(Y)$ is torsion free. 
		
	This shows that $\bar{Q} = \{0\}$, now we will show that  $\check Q \simeq \hat Q = \HMred(Y) $.
	The commutative diagram 
	\begin{equation}
	\begin{tikzcd}
		&\HMto(\Sone)\arrow[r, "\HMto(W)"]\arrow[d, "j_*"]& \HMto(Y)\arrow[r, "\HMto(Q)"]\arrow[d, "j_*"]& \HMto(\Sone)	\arrow[d, "j_*"]\\ 
		&\HMfrom(\Sone)\arrow[r, "\HMfrom(W)"]\arrow[d,"p_*"]& \HMfrom(Y)\arrow[r, "\HMfrom(Q)"]\arrow[d, "p_*"]& \HMfrom(\Sone)\arrow[d, "p_*"]\\ 
		&\HMbar(\Sone)\arrow[r, "\HMbar(W)"]& \HMbar(Y)\arrow[r, "\HMbar(Q)"]& \HMbar(\Sone)\\ 
	\end{tikzcd}
	\end{equation}
	is isomorphic to 
	\begin{equation}
	\begin{tikzcd}
		&\check G\arrow[d, "0"]\arrow[r, hookrightarrow]& \check G \oplus\check Q\arrow[r, "pr_1"] \arrow[d]& \check G	\arrow[d,"0"]\\ 
		&\hat G \arrow[d,hookrightarrow]\arrow[r,hookrightarrow] & \hat G \oplus \hat Q \arrow[r, "pr_1"]\arrow[d]& \hat G\arrow[d, hookrightarrow]\\
			&\bar G \arrow[r,hookrightarrow,"\simeq"] & \bar G \arrow[r, "\simeq"]& \bar G\\
	\end{tikzcd}	
	\end{equation}
	where the first and last column are the exact sequence  \eqref{ExactSequenceinFloerHomology} for $\Sone$ ($j_*$ is zero by  \cite[Proposition 36.1.3]{KM}).   The middle column is the exact sequence \eqref{ExactSequenceinFloerHomology} for $Y$, 
	and since the diagram is commutative and horizontally the first maps are inclusions, we see that when restricted to the $G$-factors
	the middle column coincides with  exact sequence for $\Sone$.
	 The claim $\hat Q = \HMred(Y)$ is equivalent to $\hat Q = \ker (\hat G \oplus \hat Q\to \bar G)$ which is now clear from the last two rows of the diagram. To see that $j_*:\check{Q}\overset{\simeq}{\to} \hat{Q} $ is an isomorphism we consider again the diagram above for surjectivity and for injectivity the exact sequence $\bar{G}\to \check{G}\oplus \check{Q}\to \hat G\oplus \hat Q$ restricting over the $G$s to the exact sequence of $\Sone$. 
	\end{proof}

\section{On the variation of the Seiberg-Witten invariants.}\label{sec:VariationSW}
\begin{plan}		
The aim of this section is to give a proof of the main theorem, Theorem~\ref{thm:DeltaInHMred}. The proof of items \ref{thm:DeltaInHMred:item:Delta} and \ref{thm:DeltaInHMred:item:HomOrientation}
rely on the study of some cobordisms maps and, in the case of \ref{thm:DeltaInHMred:item:Delta}, 
on Theorem~\ref{Intro:thm:Splitting}, used  in the form of Proposition~\ref{prop:SplittingOfHomology}. 
We point out that \ref{thm:DeltaInHMred:item:HomOrientation} is a non-trivial statement due to the fact that even though we have a well defined notion of composition of homology orientations, in practice,
is hard  to compute the composite orientation from the definition. The idea of our proof is to reduce to the case of an isomorphism between cobordisms, where the variation of the homology orientation 
is easily computed from the action on homology.
\end{plan}

\paragraph{Setting.}
In this section we adopt the same notation of Subsection~\ref{subsec:FloerHomologyCobordismSplitting}, in particular $N_0 = P_0(\Gamma,\Realiz)$  is a protocork, $N_1 = P_1(\Gamma, \Realiz)$ its reflection, $Y=\partial N_0 = \partial N_1$ is their common boundary and $b_1 = b_1(Y)$.
For each $i= 0,1$, we write  $N_i\setminus \ball$ to denote the manifold $N_i$ minus a ball removed from its interior, and we regard it as a cobordism $N_i\setminus \ball: \SS^3\to Y$.
We recall that an homology orientation \cite[Definition 3.4.1]{KM} for a cobordism $C:\partial_- C \to \partial_+ C$ is an orientation of 
\begin{equation}\label{def:homologyOrientationLine}
	\Lambda^{\max} H^1(C; \R) \otimes \Lambda^{\max} I^+(C)\otimes \Lambda^{\max} H^1(\partial_+ C; \R)
\end{equation}
where $I^+(C)$ is a chosen maximal non-negative subspace for the intersection form on $I^2(C) :=  \ker\left\{H^2(C;\R) \to H^2(\partial C; \R)\right\}$.
We recall that the homology orientation is part of the data necessary to define the cobordism maps induced by a cobordism. 
Choose   $\mu_0$,  an homology orientation for $N_0\setminus \ball$.

	\newcommand{\crita}{[\gota_{\ntop,-1}]}
	\begin{proof}[Proof of Theorem~\ref{thm:DeltaInHMred} \ref{thm:DeltaInHMred:item:Delta}] 
	 By Proposition~\ref{prop:SplittingOfHomology}, we have a splitting $\HMfrom(Y) \simeq \hat{G}\oplus HM^{red}(Y)$ 
	  where $\hat{G}= Im(\HMfrom(W))$, and $W: \Sone \to Y$  has been defined in Subsection~\ref{subsec:FloerHomologyCobordismSplitting}.
	  From the construction of $W$, it is clear that $N_0\setminus \ball = W\circ F$ where $F:\SS^3\to \Sone$ is the cobordism obtained by removing a ball from $\natural_{i=1}^{b_1}\SS^1\times \D^3$
	  Choosing  homology orientations on $W$ and $F$ so that their composition agrees with $\mu_0$, we see that
	  \begin{equation}
	   x_0 = 	\HMfrom(N_0\setminus\ball) (\hat 1) = \HMfrom(W)\circ \HMfrom(F)(\hat 1),
	  \end{equation}
	  therefore  $x_0\in \hat{G}$.

	     We recall that since we are dealing with torsion \spinc structures, the Floer homologies of $\SS^3, \Sone$ and $Y$ are endowed with an absolute $\Q$-grading \cite[Section 28.3]{KM} such that multiplication by $U$ is an endomorphism of degree $-2$. 
		The generator     $\hat 1 \in \HMfrom(\SS^3)\simeq \ZU$  has degree $\gr^{\Q} (\hat 1)  = -1$.
		And, as $\Q$-graded $\Z[[U]]$-modules,  
		\begin{spliteq}
		& {\HMfrom(\Sone) \simeq \Lambda^* (\Z^{b_1})\otimes_{\Z} \Z[[U]]},\\
		&  \gr^{\Q}\left((e_{i_1}\wedge \cdots \wedge e_{i_k})\otimes 1\right) = (k-1) - b_1,\\
		\end{spliteq}
		 where $\{e_i\}_{i=1}^{b_1}$ denotes the standard basis of $\Z^{b_1}$. This  follows from the discussion of \cite[Chapter 36]{KM}, alternatively we can use computations analogous to those of  Proposition~\ref{prop:DimOfModuli}.
		
	Now, $\HMfrom(F)$ is a homomorphism of degree \cite[Eq. (28.3)]{KM}
		\begin{equation}
		\frac 1 4 \big(\underbrace{c_1^2}_{0}-2\underbrace{\chi(F)}_{-b_1} -3\underbrace{\sigma(F)}_{0}\big) -\frac 1 2 \big(\underbrace{b_1(\Sone)}_{b_1} - b_1(\SS^3)\big) = 0.
		\end{equation}
		Thus,  $\HMfrom(F)(\hat 1)$ belongs to the $\Z$-submodule generated by $(e_1\wedge \cdots \wedge e_{b_1})\otimes 1$.
		Moreover, since also $\HMfrom(W)$ has degree zero (see Subsection~\ref{subsec:FloerHomologyCobordismSplitting}),  the elements of $\hat G$ of degree $\gr^{\Q}(x_0)=-1$ constitute a free $\Z$-module of rank $1$.
	
		We claim that this $\Z$-module is generated by $x_0$. To see this, we consider the cobordism $C:Y\to \SS^3$ that, with respect to the surgery presentation  of $Y$ described in Subsection~\ref{subsec:KirbyDiagrProtocorks}, is obtained by adding $b_1(Y)$ 0-framed  2-handles to $Y\times\{1\}\subset Y\times [0,1]$ meridional to the dotted circles generating $H_1(Y)$. Then, from the Kirby diagram described in Subsection~\ref{subsec:KirbyDiagrProtocorks}
		it is clear that in the composition $C\circ (N_0\setminus \ball)$ is a trivial cobordism because  the  $2$-handles coming from $C$ will cancel $b_1(Y)$ $1$-handles of $N_0\setminus \ball$ and, at this point, the remaining  $1$-handles can be cancelled		
		in pairs with the  $2$-handles. 
		Consequently, for some homology orientation of $C$, $\mu_C$,
		\begin{equation}
		\HMfrom(C, \mu_C)(x_0) =  \HMfrom(C\circ (N_0\setminus \ball))(\hat 1) = \hat 1 \in \HMfrom(\SS^3)
		\end{equation}
		 which proves the claim because $\hat 1$ is a generator of the $\Z$-submodule of elements of its degree.
		 
	By looking at the Kirby diagram for $N_1$ given in Subsection~\ref{subsec:KirbyDiagrProtocorks}, we see that also the composition $C\circ (N_1\setminus \ball)$ is a trivial cobordism hence 
	  $\HMfrom(C,\mu_C)(x_1) = \pm\hat 1$. Let $\mu_1$ be the homology orientation for $N_1\setminus \ball$ such that $\HMfrom(C,\mu_C)(x_1) = \hat 1$.
	Now, with respect to the splitting $\HMfrom(Y) \simeq \hat{G}\oplus HM^{red}(Y)$, $x_1$ splits as $x_1 =: (g, q) \in \hat{G}\oplus HM^{red}(Y)$. 
	 Since $\HMfrom(C,\mu_0)$ maps $x_1$ and $ x_0$ to the same element,   $\gr^{\Q}(x_1) = \gr^{\Q}(x_0)$; consequently $g = k x_0$ for some $k\in \Z$.
	  Moreover $\HMred(\SS^3) = (0)$ implies  that $\HMfrom(C,\mu_C)(q) = 0$, hence  we obtain that, with the homology orientation $\mu_C$,
	  \begin{equation*}
	  	\HMfrom(C)(x_1) = \HMfrom(C)(g,q) =  \HMfrom(C)(g,0) =  \HMfrom(C)(k x_0) =  k \HMfrom(x_0) = k \hat 1,
	  \end{equation*}
	  and consequently $k = 1$. Hence $\Delta = x_0-x_1 = q \in \HMred(Y)$ and item \ref{thm:DeltaInHMred:item:Delta} of Theorem~\ref{thm:DeltaInHMred} is proved.
	 \end{proof}
	 
	 \begin{proof}[Proof of Theorem~\ref{thm:DeltaInHMred} \ref{thm:DeltaInHMred:item:NMS}] In this proof  the geometric setting over $Y$, and the  $N_i$s is that used in Subsection~\ref{subsec:PointlikeModuli} and Subsection~\ref{subsec:FormalDimension}. 	 
	 In particular the generators of the chain complex $\hat C$ (see \autoref{recap:ChainComplexes}) are those described in Subsection~\ref{subsec:CriticalPoints} and we can express the Morgan-Szab\'o  number
	 as 
	 \begin{equation}
	 	n_{MS} = 2+ \max_{[\gotb_i], [\gotb_j]}\gr([\gotb_i], [\gotb_j])
	 \end{equation}
	 where $ [\gotb_i], [\gotb_j]\in \B^\sigma_k(Y,\sstruc)$ vary on irreducible critical points of the perturbed Chern-Simons-Dirac functional.
	 Beware that the parenthesis of $[\gotb_i] $ are due to gauge equivalence class and not to the Floer homology class (which may be even undefined since we do not know whether  $[\gotb_i]$ is a cycle).
	 We will use the following lemma. 
	\begin{lemma} \label{lem:RepByIrreducibleCycles} If $\overline{\partial}^s_u = 0$ then any element of $\HMred(Y)$ can be represented by an irreducible cycle.
	\end{lemma}	 
	\begin{proof}[Proof of Lemma~\ref{lem:RepByIrreducibleCycles}] Let $(x,y)\in C^o\oplus C^s$ be a cycle such that $[(x,y)] \in \HMred(Y) =\ker p_*$. Then there exists $(z,t) \in C^s\oplus C^u$ such that 
	\begin{equation}
		p(x,y) = (\partial^o_s x + \partial ^u_s y, y)	 =\overline \partial (z,t) = (\overline \partial^s_s z + \overline\partial ^u_s t,  \underbrace{\overline\partial^u_u z}_{ = 0}+ \overline\partial ^u_u t).
	\end{equation}	 
	Thus $y =\overline\partial ^u_u t$, hence  $[(x,y)] = [(x,y) + \hat \partial (0,t)] = [(x,y) + (\partial ^u_o t, - \overline \partial^u_u t) ] =[(x + \partial^u_o t, 0)]$,
	where we used again that $\overline{\partial}^s_u = 0$  in the expression for the differential $\hat \partial$.
	\end{proof}
	 \newcommand{\ntop}{N^{\mathrm{red}}}
	 We want to show that $U^{n_{MS}/2}\Delta = 0$.  By Theorem~\ref{thm:DeltaInHMred}, $\Delta \in \HMred(Y)$ and by  Proposition~\ref{prop:bargr} $\overline{\partial}^s_u = 0$, thus Lemma~\ref{lem:RepByIrreducibleCycles} tells us that	 \begin{equation}
	 	\Delta = \left[\sum_{i=1}^{N_{\mathrm{irr}}}c_i[\gotb_i] \right] \in \HMfrom(Y)
	 \end{equation}
	 is represented by an  \emph{irreducible} cycle for some coefficients   $c_i\in \Z$, $i=1,\dots,N_{\mathrm{irr}} $. Most importantly, we can assume that \emph{not all of the $c_i$s are zero}
	 for otherwise $\Delta =0$ and the statement would be trivial.  
	 
	 Now,  the definition of $n_{MS}$ implies that 
	 \begin{equation}
	 	U^{n_{MS}/2}\left[\sum_{i=1}^{N_{\mathrm{irr}}}c_i[\gotb_i] \right] = [r]\in \HMfrom(Y)
 	\end{equation} for some  $r\in C^u$  cycle of \emph{reducibles}.
	 Indeed, let $i_0$ be an index such that $c_{i_0}\neq 0$, and suppose that $U^{n_{MS}/2}\left[\sum_{i=1}^{N_{\mathrm{irr}}}c_i[\gotb_i] \right]$ has
	 a representative with a non zero component along an  irreducibles $[b_j]$ , then $\gr([\gotb_{i_0}], [\gotb_j]) = n_{MS}$ which is impossible. 
	 
	 Since $\Delta\in \HMred(Y)$, also $U^{n_{MS}/2}\Delta \in \HMred(Y)$, therefore $p_*[r] = 0 \in \HMbar(Y)$.
	 This tells us that there exists $(\eta^s, \eta^u) \in \bar{C}$ such that $\bar \partial (\eta^s, \eta^u) = p(r) = (\partial^u_s r, r)$. 
	 From which we obtain, thanks to  $\overline\partial^s_u = 0$ (see Proposition~\ref{prop:bargr}), that $\overline \partial^u_u\eta^u = r$.
	 This implies that $[r] = 0$ in $\HMfrom(Y)$ because
	 \begin{equation}
	 \hat \partial \eta^u = \partial^u_o \eta^u -\overline\partial^u_u \eta^u = -r,
	 \end{equation}
	 where $ \partial^u_o \eta^u = 0$ for the definition of $n_{MS}$ as above.
	 Consequently $[r]= 0$.
	\end{proof}
	
	\begin{proof}[Proof of Theorem~\ref{thm:DeltaInHMred} \ref{thm:DeltaInHMred:item:HomOrientation}] For the rest of this section we suppose that $N_0$ is symmetric. The map $\tau$ has been defined in Subsection~\ref{subsec:involutions}.
	The proof of the statement \ref{thm:DeltaInHMred:item:HomOrientation}  is based on the following lemma:
		\begin{lemma} \label{lemma:constructionofCobT}There exists a cobordism $T: Y\to \SS^3$, such that 
		\begin{enumerate}
						\item $\tau:Y\to Y$ extends to an orientation preserving self-diffeomorphism of $T$, $f:T\to T$ which restricts to the identity on $\SS^3$,
						\item  $T\circ (N_0\setminus \ball) :\SS^3\to \SS^3$ is the trivial cobordism,
						\item $H^1(T) =0$, $\ker(H^2(T,\R)\to H^2(\partial T,\R)) = 0$.
		\end{enumerate}
		\end{lemma}
	We postpone the proof of this lemma to the next paragraph and we continue the proof. 
	Since $x_0  = \HMfrom(N_0\setminus \ball,\mu_0) (\hat 1) $, the second item of Lemma~\ref{lemma:constructionofCobT} implies that $\HMfrom(T,\mu_T)(x_0) = \hat 1$ for some homology orientation $\mu_T$.
	Denote by $(N_0\setminus \ball)_\tau: \SS^3\to Y$ the cobordism  obtained from $N_0\setminus \ball$ by replacing the identification map of the outcoming boundary with $\tau$.
	Clearly, from the definition of $\tau$ and $\hat\tau$ in Subsection~\ref{subsec:involutions} it follows that
	\begin{equation}
		x_1 = \HMfrom(N_1, \hat \tau_*\mu_0)(\hat 1)  = \HMfrom ((N_0\setminus \ball)_\tau,\mu_0)(\hat 1).
	\end{equation}
	On the other hand, the first item of Lemma~\ref{lemma:constructionofCobT} tells us that 
	\begin{equation}
	(T, \mu_T)\circ ((N_0\setminus \ball)_\tau, \mu_0) \simeq  (T, f_*\mu_T) \circ (N_0 \setminus \ball,\mu_0),
	\end{equation}
	which implies that $\HMfrom(T, \mu_T) (x_1) = \pm \hat 1 $; to conclude the proof we have to show that this is equal to $+\hat 1$.
	
	The third item of Lemma~\ref{lemma:constructionofCobT} implies that for $T$ the vector space \eqref{def:homologyOrientationLine} 
	 is zero dimensional hence $f$  preserves the homology orientation, i.e. $f_*(\mu_T) = \mu_T$.
	Consequently 
	 	\begin{equation}
	 		\hat 1  = \HMfrom(T, \mu_T) (x_i) 
	 	\end{equation} for any $i\in \{0,1\}$ which concludes the proof.
	\end{proof}

We are left to prove Lemma~\ref{lemma:constructionofCobT}.
\begin{proof}[Proof of Lemma~\ref{lemma:constructionofCobT}]	
Let $\Gamma$ be a symmetric protocork plumbing graph.
We will construct  a Kirby diagram for $P_0(\Gamma)$ and  closed curves  $\gamma_1,\dots, \gamma_{b_1(Y)}$ in its boundary such that $\gamma_i$ intersects the belt sphere of the $i$-th 1-handle of $P_0(\Gamma)$ exactly once and $\tau$ (for some realization $\Realiz$) preserves the support of $\{\gamma_i\}_i$.
At this point $T$ will be constructed by gluing  $2$-handles  to $Y\times\{1\}\subset Y\times I$ along the curves $\{\gamma_i\}_{i=1}^{b_1}$,  this will allow us to extend $\tau$ to $T$ and the other property will cause the cancellation of the 1-handles of $N_0\setminus \ball$ when composed with $T$.

\paragraph{Embedding $\Gamma$ and the curves $\gamma_i$.}
Let $n$ be the sphere number of  $\Gamma$, since $\Gamma$ is symmetric, it is obtained by adding some edges  to the graph $\Gamma'$ embedded in $\R^3$ showed in Figure~\autoref{fig:Cob:1} for the case $n=3$.
We choose as spanning tree for $\Gamma$, the gray edges  in  Figure~\autoref{fig:Cob:1}, i.e. all the horizontal edges and the lower diagonal edges.
All the other edges  $e_1,\dots, e_k$ are in excess and will contribute to generate $b_1(Y)$.  To each of them we associate an 8-shaped  loop $\gamma_i$ on the embedding of $\Gamma$ as showed in  Figure~\autoref{fig:Cob:2}, if the edge is $e_k$, the path $\gamma_k$ will be given by $e_k$, the two horizontal edges meeting $e_k$ and the diagonal edge below $e_k$.

Now  pick  an edge of  $\Gamma$ not in $\Gamma'$, and call it $e$. Suppose $e$ joins $\vA_i$ to $\vB_j$ and $i< j$. Then we embed it as in \autoref{fig:Cob:3}, in such a way that is \emph{above} all the other edges.
If instead $i>j$ then we embed it as in  Figure~\autoref{fig:Cob:4}, so that it is \emph{below} all the other edges. If instead $i=j$ then we embed it in the plane of the vertices as in  Figure~\autoref{fig:Cob:5}.
Then we associate to $e$  loop $\gamma_e$, given by a zig-zag  (diagonal edge, horizontal edge) as showed in  Figure~\autoref{fig:Cob:3},  Figure~\autoref{fig:Cob:4}, Figure~\autoref{fig:Cob:5} depending on $i,j$.
Since $\Gamma$ is symmetric, there is an edge $e'$ associated to $e$, such with $e' \simeq (\vA_j, \vB_i, \pm)$. We apply the same procedure to embed $e'$ obtaining a new loop $\gamma_{e'}$.
Iterating the previous steps  we embed all of $\Gamma$.

We use this data to construct a Kirby diagram for $P_0(\Gamma)$ as in Subsection~\ref{subsec:KirbyDiagrProtocorks}.  The loops $\{\gamma_i\}_i$ induce up to isotopy some loops in the diagram. 
Notice that we will have also an associated symmetric realization datum given by the pairs $e, e'$ and thus a diffeomorphism $\tau$
that exchanges $\gamma_{e}$ and $\gamma_{e'}$ (we parametrize the curves so that $\tau$ preserves the orientation of $\gamma_e$ and $\gamma_{e'}$).

\begin{figure}[p]
	\centering
		\subfloat[][\label{fig:Cob:1}]
		   {\includegraphics[width=.20\columnwidth]{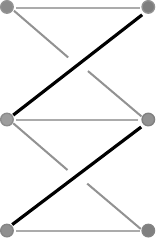}} \quad
	\subfloat[][\label{fig:Cob:2}]
		   {\includegraphics[width=.20\columnwidth]{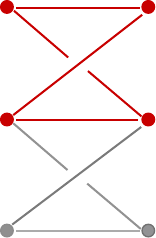}} \\
	\subfloat[][\label{fig:Cob:3}]
		   {\includegraphics[width=.25\columnwidth]{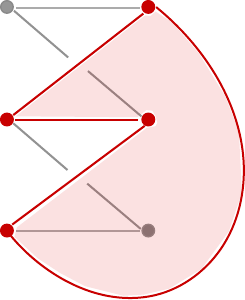}} \quad
	\subfloat[][\label{fig:Cob:4}]
		   {\includegraphics[width=.25\columnwidth]{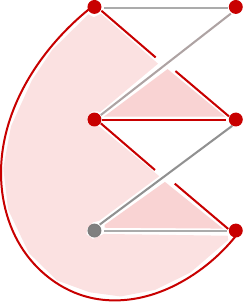}}\\
	   \subfloat[][\label{fig:Cob:5}]
	   {\includegraphics[width=.20\columnwidth]{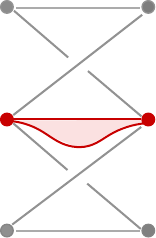}} \quad
	\caption{\eqref{fig:Cob:1} The graph $\Gamma'$ for $n=3$. The spanning tree is in gray while the black edges corresponds to the excess arcs $e_1,e_2$. \eqref{fig:Cob:2} In red, the $8$-shaped loop $\gamma_1$ corresponding to $e_1$. \eqref{fig:Cob:3} Case $i<j$. Notice the arc is above the rest of the graph. In red, the loop $\gamma_e$ with the disk it bounds in light red. \eqref{fig:Cob:4} Case $i>j$. Notice the arc is below the rest of the graph. In red, the loop $\gamma_e$ with the disk it bounds in light red. \eqref{fig:Cob:3} Case $i=j$.  In red, the loop $\gamma_e$ with the disk it bounds in light red.\label{fig:Cob}}
\end{figure}

\paragraph{Properties of $\gamma_i$}
The set of curves $\gamma_i$ has some properties.  First of all, with reference to the Kirby diagram,  each curve will pass exactly once trough the dotted circle of  its associated edge and  through of a certain number of dotted-circles associated
to the diagonal edges of $\Gamma'$ in the case that  $\gamma$ passes through the upper   diagonal edges, i.e. if $e$ joins $\vA_i$ to $\vB_j$ and $i< j$.

Secondly, if we add to $P_0(\Gamma)$ $b_1(Y)$ 0-framed 2-handles along the curves $\gamma_i$ then these will cancel out the 1-handles of $P_0(\Gamma)$.
Indeed, the $\gamma_i$ relative to the 8-figure shaped paths will appear as loops meridional to the dotted circles of the clasps associated to the diagonal edges.
Hence we can cancel these dotted circles. 
After this cancellation, the remaining loops $\gamma_i$ can be isotoped in the diagram to curves that are meridional to the remaining dotted circles.
This is because they do not interact between each other, as can be seen from the fact that the $\gamma_i$s can be pushed to the boundary of a tubular neighbourhood $\nu \Gamma\subset \R^3$ of $\Gamma$ in in such a way that their image bound disks in the complement of $\nu\Gamma$.  Figure~\autoref{fig:Cob:3},  Figure~\autoref{fig:Cob:4}, Figure~\autoref{fig:Cob:5} provide some intuition for these disks.
As a result, we can cancel also the remaining dotted circles.
Notice that the result of the 2-handle addition is  $\mathbb{B}^4$ since we will have cancelled all the clasps from the diagram except one for each pair  $\vA_i, \vB_i$.

 Let $T$ be the cobordism from $Y$ to $\SS^3$ obtained by gluing the 2-handles as above. 
Since the set of attaching curves $\gamma_i$ is preserved by $\tau$, we can extend $\tau$ to a diffeomorphism  $f: T\to T$, which exchanges the 2-handles relative to each pair of $\gamma_{e}$ and $\gamma_{e'}$ exchanged by $\tau$ and if $\tau$ sends  $\gamma_i\to \pm \gamma_i$ then we extend it to a diffeomorphism of the handle $\D^2\times \D^2$ accordingly.
The requirement that $f$ restricts to the identity on $\SS^3$ is simple to fulfull because any orientation preserving diffeomorphism of $\SS^3$ is isotopic to the identity, therefore we can just modify $f$ in a collar neighbourhood of the boundary to satisfy this requirement.

\paragraph{Homological properties.} To prove the third item of Lemma~\ref{lemma:constructionofCobT}, we notice that $H^1(T) = (0)$ because $T$ is obtained by gluing $2$-handles along the curves generating  $H^1(Y)$ hence these generators are killed. Moreover $H^2(T,Y) \simeq \Z^{b_1}$ and $H^3(T,Y) = (0)$, because $H^*(T,Y)$ is generated by the handles used in the construction and these account only to  $b_1$ 2-handles. 
Now the long exact sequence of $(T,Y)$ gives 
\begin{equation}
	0 \to \underset{\simeq \Z^{b_1}}{H^1(Y)} \to   \underset{\simeq \Z^{b_1}}{H^2(T,Y)} \to H^2(T) \to  \underset{\simeq \Z^{b_1}}{H^2(Y)} \to 0
\end{equation}
from which  is clear that 
\begin{equation}
	\ker \left(H^2(T,\R)\to H^2(\partial T,\R)\simeq H^2(Y,\R)\oplus H^2(\SS^3, \R)\right)= (0).
\end{equation}
This concludes the proof  of Lemma~\ref{lemma:constructionofCobT}.
\end{proof}

\section{Moduli spaces of reducible monopoles over a protocork.}\label{sec:4DModulispaces}		
\begin{plan}	
In this section we study the  moduli spaces of perturbed reducible monopoles over a protocork $P_0(\Gamma)$. 
In Subsection~\ref{subsec:background4D} we recap the relevant background and notation from the book \cite{KM} by  Kronheimer and Mrowka that we will need.
In Subsection~\ref{subsec:CriticalPoints} we compute the formal dimension of the  reducible moduli spaces of monopoles asymptotic to 
the critical points described in Subsection~\ref{subsec:CriticalPoints}. .
In the last subsection we show that the moduli space of reducible monopoles limiting to $[\gota_{N^{\mathrm{red}},-1}]$ consist of a single point when appropriate perturbations are considered proving Theorem~\ref{thm:ModuliIsPoint}. This is a fact of its own interest, because  this moduli space constributes to the Seiberg-Witten invariants. 
Our propositions will descend from some nice properties of protocorks, in particular: 
\begin{enumerate}
	\item $H^1(P_0(\Gamma) \to H^1(\partial P_0(\Gamma))$ is surjective,
	\item the diffeomorphism $\rho_A$ (see Subsection~\ref{subsec:involutions}) is orientation reversing and preserve $H^1(\partial P_0(\Gamma))$,
	\item in  the \emph{classical}  moduli space of \emph{unperturbed} 3-monopoles over $\partial P_0(\Gamma)$, the torus of reducibles is a Morse-Bott singularity and 
	\item the Dirac operator on $\partial P_0(\Gamma)$ associated to flat connections  is invertible.
\end{enumerate}
The main difficulty is to deal with the and  perturbations and the  blown-up moduli spaces where all the spectrum of the Dirac operator matters.
\end{plan}

\subsection{Background on 4D moduli spaces.}\label{subsec:background4D}
\begin{plan}
In this subsection we review the relevant background   from  \cite{KM} and establish our notation. The definition of
the perturbed moduli space asymptotic to a $3$-monopole, together with the functional spaces  used  therein will be particularly important to us. We make no claim of originality.
\end{plan}

\begin{recap}[Classical and blown-up configuration spaces on a compact 4-manifold with boundary]
	\renewcommand{\sstruc}{\mathfrak{s}_X}
	Let $X$ be a \emph{compact}, smooth, oriented 4-manifold with boundary $\partial X = Y$ endowed with a Riemannian metric and \spinc structure $\sstruc$ consisting of spinor bundle $S_X\to X$ and Clifford 			multiplication $\rho_X: TX\to \Hom(S_X, S_X)$ extended as usual to $\Lambda^\bullet X$. The bundle $S_X $ splits as $S^+\oplus S^-$, which are respectively the   $-1$ and $+1$ eigenbundles of $		\rho_X(\operatorname{vol}_X)$, where $\operatorname{vol}_X$ 
	is the Riemannian volume form.
	 We will denote the Dirac operator induced by $A\in \cA(X,\sstruc)$ acting on positive spinors as $D^+_A:\Gamma(S^+)\to \Gamma(S^-)$. As in  the 3D case, given a \spinc connection $A$, we will denote by $A^t$ the induced connection on the determinant bundle. 
		Let $k\geq 0$, then similarly to the 3-dimensional case, we can define 
	\begin{spliteq}\label{defConf4D}
				& \Conf_k(X, \sstruc) = \cA_k(X,\sstruc)\times L^2_k(X; S^+)\\
				& \G_{k+1}(X) = \{u\in L^2_{k+1}(X;\C) \ | \  |u(x)|  = 1 \}\\
				& \B_{k}(X,\sstruc) = \Conf_{k}(X,\sstruc)/\G_{k+1}(X),\\
	\end{spliteq}
	which are respectively the configuration space, gauge group and quotient configuration space. The action of $u\in \G_{k+1}(X)$ is the same as in \eqref{eq:actionOfG}.   In the following we will always assume that $k	> 2$ and if $k$ is omitted, smooth objects are considered.
	We also have the blown-up spaces: 
	\begin{spliteq}
				& \Conf_k^\sigma(X, \sstruc) = \cA_k(X,\sstruc)\times \R_{\geq 0} \times\{ \varphi \in L^2_k(X; S^+) \ | \ \norm{\varphi}_{L^2(X)} = 1\}\\
				& \B^\sigma_{k}(X,\sstruc) = \Conf^\sigma_{k}(X,\sstruc)/\G_{k+1}(X),\\
	\end{spliteq}
	where $u\in \G_{k+1}(X)$ acts as in \eqref{eq:actionOfGsigma}.
	The completion of the tangent spaces of $\Conf_k(X, \sstruc)$ and $\Conf_k^\sigma(X, \sstruc)$ in the $L^2_j$-norm, $j\leq k$ are denoted by $\cT_{j}$ and $\cT_j^\sigma$ respectively.
	As in the 3-dimensional case there is a blow-down map $\pi: \Conf^\sigma_k(X,\sstruc)\to \Conf_k(X,\sstruc)$ given by $(A,s, \varphi) \mapsto (A,s\varphi)$  that is a diffeomorphism over the irreducibles. 
	We also introduce the vector bundle $\cV_j\to \Conf_k(X,\sstruc)$, $j\leq k$ defined by 
	\begin{equation}\label{defVbundle}
		\cV_j = L^2_j(X; i\su(S^+)\oplus S^-)
	\end{equation}
	where $\su(S^+)$ is the bundle of traceless skew-hermitian endomorphisms of $S^+$.
	The blow-up  $\cV_j^\sigma \to \Conf^\sigma_k(X,\sstruc)$  is defined as the pull-back along $\pi$ of $\cV_k$.
	This will be the codomain for the Seiberg-Witten map. 
\end{recap}

\begin{recap}[Special notation in the case of the cylinder] \label{recap:SpincStructuresOnCylinders} Let $Z = I\times Y$, where $I\subset \R$ is an interval,  and let $\sstruc$ be a \spinc structure on $Y$ with spinor bundle $S$ and Clifford multiplication $\rho$ as in Subsection~\ref{subsec:backgroundFloer}.
We can endow $Z$ with a \spinc structure $\sstruc_Z$   where, denoting by $\pi:Z\to Y$ the projection onto the second factor, the spinor bundle is $S_Z := \pi^*S\oplus \pi^*S $,
	and Clifford multiplication $\rho_Z:TZ \to \End(S_Z)$ is given by 
	\begin{equation*}
		\rho_Z(\partial/\partial t) := \begin{bmatrix}
														0 & -1\\
														1 & 0 \\
													\end{bmatrix} 										
										\quad \text{ and } \quad  														
													 \rho_Z(v) := 	\begin{bmatrix}
														0 &  - \rho(v)^*\\
														\rho(v) & 0 \\
													\end{bmatrix}
	\end{equation*}
	for $v \in TY\hookrightarrow TZ.$ 	In particular the bundles of positive and negative spinors  are $S^+ = S^- = \pi^*S$.
	 A   configuration $(A,\phi) \in \cA_k(Z, \sstruc_Z)\times L^2_k(S^+)$  induces a time-dependent configuration 
		$(\check{ A}(t), \check \phi(t)) \in \Conf_k(Y)$ defined implicitly by putting 
		\begin{equation}
			(A, \phi)|_{\{t\}\times Y} =  (\check{A}(t)  + c(t)dt \otimes 1_S, \check{ \phi}(t))
		\end{equation}
		for some $c:Z\to\R$, so that $\check{A}(t)$ is in temporal-gauge.
	Moreover,  we have an isomorphism of vector bundles
	\begin{spliteq}
			& \pi^*\Lambda^1(Y) \overset{\sim}{\longrightarrow}  i\su(S^+)\\
			& \pi^*\alpha\mapsto \frac{1}{2} \rho_Z\left( \pi^*(*\alpha) +  \pi^*(\alpha)\wedge dt) \right).\\
	\end{spliteq}
	In particular  $\mathcal{V}_0(Z)\simeq L^2(Z; \pi^*\Lambda^1(Y)\oplus \pi^*S)$.
	Now, consider a perturbation $\pert: \Conf(Y,\sstruc) \to \cT_0(Y)$. 
	This determines a map 
	\begin{equation}\label{defOfhatpert}
		\hat{\pert}: \Conf(Z,\sstruc_Z) \to \mathcal{V}_0(Z)\simeq L^2(Z; \pi^*\Lambda^1(Y)\oplus \pi^*S),
	\end{equation}
	defined by 
	\begin{equation}\
		\hat \pert (A,\varphi)|_{\{t\}\times Y}  = \pert(\check A(t), \check{\varphi}(t)) \in \cT_0(Y)= L^2_0(iT^*Y\oplus S).
	\end{equation}
	More information may be found in \cite[Chapter 4]{KM}.
	\end{recap}
	{	\newcommand{\pertp}{{\mathfrak{p}}}
		\newcommand{\linearAction}{\mathbf{d}^\sigma}
		\newcommand{\dsdag}{\mathbf{d}^{\sigma,\dagger}}
		\renewcommand{\Re}{\mathrm{Re}}
	\begin{recap}[Perturbations] \label{recap:4DPerturbations}We will briefly describe the structure of the perturbations used over 4-manifolds, more details may be found in  \cite[Section 24]{KM}. Suppose that $(X,\sstruc_X)$ and  $(Y,\sstruc)$, $Y=\partial X$  are as above.
	In particular, $X$ is compact, and we will suppose that there is a collar of $\partial X$  isometric to  $(-C, 0]\times Y$ for some $C>0$,
	where $\sstruc_X$ is constructed with the \spinc structure of $Y$, $\sstruc$ as in the previous paragraph. 
	 The perturbation will be supported in this collar. 
	Fix $\pert $,  an admissible perturbation for $(Y, \sstruc)$, and let $\beta\in C^\infty([-C,+\infty))$ be a cut-off function, such that $\beta(t) = 1$ for $t\geq 0$, $\beta(t) = 0$ for $t\leq -C/2$.
	Let $\beta_0\in C^\infty((-C,+\infty))$ be a bump-function with $\supp (\beta_0) \subset (-C,-C/2)$.
	For any perturbation  $\pert_0 \in \Pert(Y,\sstruc)$, we define a perturbation over $(-C,0]\times Y$ by setting
	\begin{equation}\label{defPert4D}
			\hat {\mathfrak{p}} =  \beta_0\hat\pert_0 +  \beta \hat\pert : \Conf((-C,0]\times Y, \sstruc_X|_{(-C,0]\times Y}) \to \cV_0 ((-C,0]\times Y)
	\end{equation}
	where we are using \eqref{defOfhatpert}.	
	This induces a perturbation 
	\begin{equation}
		\hat{\mathfrak{p}}: \Conf_k(X,\sstruc_X)\to \cV_k(X)
	\end{equation}
	that depends only on the behaviour of the configuration on the collar. 
	Notice that  $\hat {\mathfrak{p}}$ will be slice-wise equal to  $\pert$ in a neighbourhood of $\{0\}\times Y$.
	$\hat{\pertp}$ induces a perturbation on the blow-up as follows.
		Writing $\mathcal{V}_k = L^2_k(X, i\mathfrak{su}(S^+)\oplus S^-)$, we have the decomposition $\hat \pertp = (\hat \pertp^0, \hat \pertp^1)$.
	We define $\hat\pertp^\sigma$ on the blow-up to be the section 
	\begin{spliteq}	
		\hat{\pertp}^\sigma & : \Conf^\sigma_k(X, \sstruc_X)\to \mathcal{V}_k^\sigma\\
		& \hat{\pertp}^{0,\sigma}(A,s,\phi) = \pertp^0(A,s\phi)\\
		& \hat{\pertp}^{1,\sigma}(A,s,\phi) =\begin{cases} \frac 1 s \hat\pertp^1(A,s\phi) \quad \text { if } \ s\neq 0\\
				 \Differential_{(A,0)} \hat\pertp^1 (\phi)\quad \text { if } \ s = 0\\ 
				 \end{cases}
	\end{spliteq} 
	
	\end{recap}
	
	\begin{recap}[Blown-up moduli spaces over compact $X$]  The  Seiberg-Witten map (in the blown-up setting) is defined as
	\begin{spliteq}
		 \cF^\sigma   & : \Conf^\sigma_k(X, \sstruc_X)\to \mathcal{V}_k^\sigma\\
		&(A,s,\varphi) \mapsto (\frac{1 }{2}\rho_X(F^+_{A^t} - s^2(\varphi \varphi^*)_0, D^+_A \varphi)\\
	\end{spliteq} which is a lift of the classical Seiberg-Witten map under the blow-down map $\pi:\Conf^\sigma(X,\sstruc_X)\to \Conf(X,\sstruc_X)$.
	Given a perturbation $\hat \pertp^\sigma$ as in the above paragraph, we can define the perturbed Seiberg-Witten map 
	\begin{spliteq}\label{defFpertsigma}
		 \cF_\pertp^\sigma  = \cF^\sigma + \hat{\pertp}^\sigma : \Conf^\sigma_k(X, \sstruc_X)\to \mathcal{V}_k^\sigma.\\
	\end{spliteq}
	The zero locus of $\cF^\sigma_\pertp$ is invariant under $\G_{k+1}(X)$ and  we have the moduli spaces of solutions
	\begin{equation}
		M(X,\sstruc_X) =\{x \in \Conf_k(X,\sstruc_X) \ | \  \cF^\sigma_\pertp(x) = 0\} / \G_{k+1}(X) \subset \B_k(X,\sstruc_X)
	\end{equation}
	 which we will refer to as the moduli space of perturbed monopoles. 
	\end{recap}

	\newcommand{\Lloc}{L^2_{k,loc}}	
	\begin{recap}[Moduli spaces for manifolds with cylindrical ends]
	We will denote by  $X^*$, the manifold $X$ with cylindrical ends attached to the boundary using the diffeomorphism $\partial X\simeq Y$,
	\begin{equation}
			X^* = X \bigcup ([0,+\infty) \times Y).
	\end{equation}
	On $X^*$ we extend the Riemannian metric of $X$ so that  the infinite cylinder  is isometric to  $(-C,+\infty)\times Y$.
	The relevant spaces in this context are the  $\Lloc$-configuration spaces:
	\begin{equation}
	 \mathcal{C}_{k,loc}(X^*,\sstruc_X)  = \cA_{k,loc}(X^*,	\sstruc_X)\times \Lloc(X^*; S^+).
	\end{equation}
	Note however that $\Lloc(X^*)$ is not a Banach space as it is not normable. Also notice that the $L^2_k$-norm of an element of $\Lloc(X^*)$ can be infinite in general.
	The $\Lloc$-blown-up configuration spaces are defined as follows. Let 	$\Sphere := (\Lloc(X^*,S^+)	 \ \setminus\ \{ 0\})	\big / \R_>$, then put 
	\begin{spliteq*}
		&\Conf^\sigma_{k,loc}(X^*, \sstruc):= \big\{(A, \phi,\R_>\phi )| \ (A, \phi)\in \Conf_{k,loc} (X^*, \sstruc), \phi \neq 0 \big\} \subset \Conf_{k,loc} (X^*, \sstruc)\times \Sphere.\\
	\end{spliteq*}
	The relevant group of tranformations is the gauge group $\G_{k+1,loc}(X^*)$ defined as in \eqref{defConf4D} but using $\Lloc$ Sobolev spaces.
	We can define as usual $\B_{k,loc}(X^*,\sstruc_X)$ and $\B^\sigma_{k,loc}(X^*, \sstruc_X)$ by quotienting out the above defined configuration spaces by the action of $\G_{k,loc}(X^*)$.
	Over $\Conf_{k,loc}(X^*,\sstruc_X)$ we have a bundle $\mathcal{V}_{j,loc}(X^*)$ defined as in \eqref{defVbundle} but using $L^2_{k,loc}$-coefficients.
	In the blow-up setting, the bundle $\mathcal{V}^\sigma_{j,loc}(X^*)$,  is defined as follows. 
	First of all, we define the tautological line bundle
	\begin{equation}
		\mathcal{O}(-1) := \big \{ (A,\phi, \R_>\phi, z\phi) \ | \  (A,\phi, \R_>\phi) \in \Conf_{k,loc}(X^*), \ z \in \C \big\} \to \Conf_{k,loc}(X^*) \\
	\end{equation}
	then we put
	\begin{spliteq*}
	 & \mathcal{V}_{j,loc}^\sigma(X^*) := \mathcal{O}(-1)^*\otimes \pi^* \mathcal{V}_{j,loc}(X^*)\to \Conf^\sigma_{k,loc}(X^*).\\
	\end{spliteq*}
	This definition generalizes  the one given for $\mathcal{V}^\sigma_j(X)$, indeed in the $L_j^2$-case, the norm allows us to trivialize the bundle $\mathcal{O}(-1)$, hence  $\mathcal{V}^\sigma_j(X)$ is just the pullback via the blow-down map $\pi$. 
	The operator $\cF_\pertp^\sigma$ defined in \eqref{defFpertsigma} induces an a continuous section of $\cV_{j,loc}^\sigma$ invariant under the action of $\G_{k+1,loc}(X^*)$ that we will denoted again by
	$\cF_\pertp^\sigma$. We can define then $M(X^*,\sstruc_X)$ as the zero locus of $\cF_\pertp^\sigma$ quotiented out by $\G_{k,loc}(X^*)$.
	\end{recap}
	
	\newcommand{\criticalb}{{[\mathfrak{b}]}}
	\newcommand{\quotY}{\B^\sigma_{k}(Y,\sstruc_Y)}
	\newcommand{\quotW}{\B^\sigma_{k,loc}(X^*,\sstruc_X)}
	\newcommand{\limdx}{\underset{\to}{\lim}}
	\begin{recap}[$\tau$-model]
	In the case of a cylinder $ Z = I\times Y$ ($I\subset \R$ possibly unbounded), there is another model, called $\tau$-model  \cite[Section 13.1]{KM}. The $L^2_{k,loc}$-version  of it
		 is the set 
	\begin{equation}
		\Conf^\tau_{k,loc}(Z,\sstruc_Z) \subset \cA_{k,loc}(Z)\times L^2_{k,loc}(I)\times L^2_{k,loc}(Z,S^+)
		\end{equation}
		consisting of triples $(A,s, \varphi)$		 such that  $s(t)\geq 0 $ and $\norm{\varphi(t)}_{L^2(Y)} = 1$ for all $t \in I$. The \spinc structure over $Z$ is always the one induced by $Y$ and defined in \autoref{recap:SpincStructuresOnCylinders}.
	The gauge group	$\G_{k+1,loc}(Z)$ acts on $\Conf^\tau_{k,loc}(Z,\sstruc_Z)$ with quotient $\B^\tau_{k,loc}(Z,\sstruc_Z)$.
	Elements of $\Conf^\tau_{k,loc}(Z,\sstruc_Z)$ with  connection component in temporal gauge naturally represent a path of configurations  in  $\Conf^\sigma_k(Y,\sstruc)$.	
 	Similarly to the case of the blown-up model, there is a bundle $\cV_{k,loc}^\tau\to \Conf^\tau_{k,loc}(Z,\sstruc_Z)$ and 
 	a perturbed Seiberg-Witten map $\cF^\tau_{\pertp}$ which is a section of it.
 	Suppose that $I = \R_\geq $  and we are  given an element $[\gamma] \in  \B^\tau_{k,loc}(Z, \sstruc_Z)$ and a critical point $\criticalb \in \quotY$. Then we write that 
 	\begin{equation}
	 	\limdx [\gamma] = \criticalb
 	\end{equation} if $[ \gamma(t+\cdot, \cdot)]  $ tends in $\Lloc(Z)$ to the translation invariant solution defined by $\criticalb$ as $t\to +\infty$.
 	In an analogous way one can define ${\underset{\leftarrow}{\lim}}[\gamma] $, see  \cite[Section 13.1]{KM}.
 	\end{recap}

     \begin{recap}[Moduli spaces of monopoles asymptotic to a 3-monopole] Now we want to define the moduli space of solutions limiting to a critical point $\criticalb\in \quotY$. First of all, notice that there is a restriction map defined slicewise by renormalizing the spinor component
     \begin{equation}\label{__aux1}
     R_Z: \big	\{  [\gamma] \in \quotW \ \big |\ \cF^\sigma_\pertp (\gamma) = 0\big\} \to \B^\tau_{k,loc}(Z, \sstruc_Z),
     \end{equation}
     where $Z = [0,+\infty)\times Y$, and we are implicitly using the diffeomorphism $\partial X\simeq Y$. Notice that \eqref{__aux1} can be defined only on the solution space because we need
     to appeal to the unique continuation property for the solutions of the equations. 
	We define 
	\begin{equation}
		M(X^*, \sstruc_X, \criticalb) \subset \quotW
	\end{equation}
	as the set of $[\gamma]$ such that $\cF_\pert^\sigma(\gamma) = 0$ and the restriction $\limdx \ R_Z[\gamma] = \criticalb$.
	In \cite{KM}, the notation $	M_z(X^*, \sstruc_X, \criticalb) $ is used to denote the component of the moduli space 
	consisting of monopoles of homotopy class $z$  \cite[pg. 474]{KM} we will not need it because in our specific case there will be only one class $z$.
	\end{recap}
	
	\begin{recap}[Index function] 	Suppose now that our  4-manifold is a cobordism $W^4:Y_0\to Y_1$. We will assume to have fixed  \spinc structure $\sstruc_W$ over $W$  and that  $[\gota]$ and $ [\gotb]$ are non-degenerate critical points respectively over $Y_0$ and $Y_1$ relative
	to the \spinc structure obtained restricting $\sstruc_W$.
	
	To this data, Kronheimer and Mrowka associate an integer  
	$\gr([\gota],W^*,[\gotb]) $ \cite[pg. 475]{KM}, the index of a Fredholm operator.
	In general $\gr$ depends also on the choice of  a $W$-path $z$ from $[\gota]$ to $[\gotb]$,  \cite[Def. 23.3.2]{KM}  modulo homotopy fixing the endpoints, however in all our applications the restrictions of $\sstruc_W$ will be torsion and in this case $\gr$ becomes independent of $z$, thus we drop it from the notation.
	
 The function $\gr$ relates to the  formal dimension of the moduli spaces in the following way:
	\begin{equation}\label{eq:dimAndGr}
	 \dim  M([\gota], W^*, [\gotb]) = \gr ([\gota], W^*, [\gotb]) + \varepsilon,
	\end{equation}
where $\varepsilon$ is computed in the following way. Let $n_+$ be the number of \emph{outcoming} connected components of $\partial W$ with a boundary-\emph{unstable} critical point  associated  
 	 similarly let $n_-$ be the number of \emph{incoming} components of $\partial W$ with a boundary-\emph{stable} critical point associated.
		Now set $c := n_+ + n_- -1$, if $c>0$, we say that the moduli space is boundary-obstructed with corank $c$ and $\varepsilon := c$, otherwise $\varepsilon := 0$.

		The  function $\gr$ enjoys the following \emph{additivity property}:
		\begin{equation}
			\gr([A],W_1^*,[B]) + \gr([B], W^*_2, [C]) = \gr([A], (W_1\circ W_2)^*, [C]),
		\end{equation}
		which instead does not hold for $\dim$.

	\end{recap}
}	

\subsection{Formal dimension of moduli spaces asymptotic to a reducible critical point. } \label{subsec:FormalDimension}
\begin{plan}The aim of this subsection is to prove Proposition~\ref{prop:DimOfModuli} below which computes the formal dimension of  reducible moduli spaces $M^{\mathrm{red}}(N^*,[\gota_{k,j}])$ where $N$ is a protocork. 
The idea of the proof is to reduce to the case of the  classical (unperturbed) moduli space limiting to the torus of reducibles, which is a Morse-Bott singularity of $\cL$, and apply the Atiyah-Patodi-Singer (APS) index theorem. This is done in a several steps. We use the characterization given by Lemma~\ref{lemma:FromL2loctoL2} below to pass from $L^2_{k,loc}$-monopoles to monopoles in a weighted Sobolev space, this is
probably known to experts, but we decided to give an explicit proof in \autoref{app:ProofofLemmaL2locL2} for future reference because, to the authors' knowledge, it  is not present in the literature. 
Then we reduce to the $k=N^{\mathrm{red}}, j=0$ case, i.e.  $M^{\mathrm{red}}(N^*,[\gota_{N^{\mathrm{red}},0}])$,   for two reasons: $j=0$ allows us to pass from weighted Sobolev spaces to  $L^2$-spaces (Lemma~\ref{Lemma:OperatorQDimension}) and $k=N^{\mathrm{red}}$ allows us to kill the perturbative term (Lemma~\ref{Lemma:IndexQIsIndASDDirac}). 
At this point  the computation of the APS index is carried out with the help of the nice properties of $Y$ showed  in Subsection~\ref{subsec:CriticalPoints}.
\end{plan}

\paragraph{Setting.}
For the rest of this section $N$ will be   an oriented, compact 4-manifold diffeomorphic to a protocork $P_0(\Gamma)$ for some protocork plumbing graph $\Gamma$. The boundary of $N$ will be denoted by $Y$.
We suppose also that $N$ is endowed with a Riemannian metric so that on a collar neighbourhood of the boundary  is isometric to $(-C,0]\times Y$, for some $C>0$ where the Riemannian metric on $Y$ satisfies the hypothesis of Subsection~\ref{subsec:CriticalPoints}.
We denote by $\sstruc$ a \emph{trivial} \spinc structure over $Y$ and we consider a \emph{trivial} \spinc structure $\sstruc_N$ over $N$ extending $\sstruc$, such that on the collar $(-C,0]\times Y$ is of the form described in  \autoref{recap:SpincStructuresOnCylinders}.

Since these are the only \spinc structures we are going to deal with, we will omit them in our notation. 
The setting over $Y$  is that described in Subsection~\ref{subsec:CriticalPoints}, in particular the perturbation $\pert$ is defined in \eqref{defAdmissiblePerturbationqOverY} 
and in  $\B^\sigma_k(Y)$, we have irreducible critical points $[\gotb_i]$,  ${i=1, \dots N^{\mathrm{irr}}}$ and reducible critical points $[\gota_{k,i}]$,  ${k=1, \dots, N^{\mathrm{red}}},  i\in \Z$ where $i\geq 0$ ($i<0$) corresponds to boundary-stable (unstable) critical points. 

\newcommand{\gotc}{\mathfrak{c}}
	\paragraph{Premise on moduli spaces of reducible monopoles.}  For any critical point $[\mathfrak{c}]\in \B^\sigma_k(Y)$, we denote by $M^{\mathrm{red}}(N^*; [\mathfrak{c}])\subset M(N^*; [\mathfrak{c}])$ the moduli space of (perturbed) \emph{reducible} monopoles. 
	If $[\mathfrak{c}]$ is irreducible, this space is empty, if $[\mathfrak{c}]$ is reducible and boundary-unstable then  it is the whole of $M(N^*; [\mathfrak{c}])$ while if 
	$[\mathfrak{c}]$ is reducible and boundary-stable it consists of the boundary of $M(N^*; [\mathfrak{c}])$ (assuming regularity of the moduli spaces). 
 
	We remark that in general the formal dimension depends on the homotopy class of the monopoles  \cite[pg. 474]{KM}. However for any $[\mathfrak{c}]$ critical points, $\pi_0(\B^\sigma_k(N;[\mathfrak{c}])) =\{*\}$ consists of a point; indeed it is a principal homogeneous space over $H^1(Y; \Z)/ H^1(N;\Z)  $
		which is trivial by Proposition~\ref{propProtocork}. Consequently, we can omit the subscript $z\in \pi_0(\B^\sigma_k(N;[\mathfrak{c}]))$  in   $M_z(N^*; [\gotc]) $.

In the next lemma  we will denote by $L^2_{k,\delta}(N^*) $, $\delta \in \R$, the weighted Sobolev space defined using as weight function $t\mapsto e^{2\delta t \beta(t) }$, in particular, $L^2_k(N^*)\to L^2_{k,\delta}(N^*) = f\mapsto e^{-\delta t \beta(t)} f$ is an isometry. Here $\beta\in C^\infty(N^*, [0,1])$ is the  cut-off function defined  in \autoref{recap:4DPerturbations} extended to zero over the complement of the tube.

It is  convenient to define the following Hilbert manifolds, associated to a spinor $\psi_0 \in L^2_{k}(Y)$ and $\lambda_0 \in \R, \delta >0$, 
\begin{equation*}	
	X^\pm_{k,\lambda_0,\delta}(\psi_0) := \left\{ \Phi \in  L^2_{k}(N^*; S^\pm) \ | \ 
	\Phi_{|_{\R_+\times Y}} (t,y) = e^{-\lambda_0 t } \psi_0(y) + \varphi(t,y), \ \ \varphi \in L^2_{k, \lambda_0+\delta}(N^*; S^\pm) \right\}
\end{equation*}
endowed with the smooth structure defined by requiring that
\begin{equation}\label{def:diffeoL2kAndX}
	L^2_{k, \lambda_0+\delta}(N^*; S^\pm) \to X^\pm_{k,\lambda_0, \delta}(\psi_0) = \varphi \mapsto e^{-\lambda_0 (\cdot) \beta(\cdot)} \beta(\cdot)\psi_0 + \varphi
\end{equation}
is a diffeomorphism.
\begin{remark} \label{rem:ConjugationOfDiracOp} Let $A\in \cA_{k,loc}(N^*)$ be  translation invariant and equal to $B$ over the tube. Then the  Dirac operator defines a  smooth map 
  \begin{equation}
		D_A^+:X^+_{k,\lambda_0,\delta} (\psi_0)\to X^-_{k-1,\lambda_0,\delta} (-\lambda_0 \psi_0 + D_{B}\psi_0),
	\end{equation}	
	which is conjugated via the diffeomorphism \eqref{def:diffeoL2kAndX}, to 
		\begin{spliteq}
			D_A^+ + K: L^2_{k, \lambda_0+\delta}(N^*, S^+)\to L^2_{k-1, \lambda_0+\delta}(N^*, S^-)
		\end{spliteq}
		where $K:L^2_{k, \lambda_0+\delta}(N^*, S^+)\to L^2_{k-1, \lambda_0+\delta}(N^*, S^-)$ is a compact operator.
\end{remark}

\begin{lemma}[Characterization of reducible moduli spaces with cylindrical ends.]   \label{lemma:FromL2loctoL2}  Suppose that $[\gotc] = [(B_\gotc, 0,\psi_\gotc)]\in \B_k^\sigma(Y)$ is a non-degenerate, reducible critical point with associated eigenvalue $\lambda_\gotc$.  
Then there exists $\overline\delta>0$, such that for any $\delta \in (0,\overline \delta]$,  $M^{\mathrm{red}}(N^*; [\gotc])$ is naturally identified with
		the set of pairs $(A,\Phi) \in A_\gotc +L^2_{k,\delta}(N^*,i\Lambda^1N^*)\times X^+_{k,\lambda_\gotc, \delta}(\psi_\gotc)$ such that 
		\begin{equation}
		 \label{eq:monopoleCyl} \begin{matrix} F_{A^t}^+  + \hat{\mathfrak{p}}^0(A^t,0) &= 0\\
								  D^+_A \Phi  +\Differential_{(A,0)}\hat {\mathfrak{p}}^1(\Phi) &= 0 \end{matrix}\\
		\end{equation}
		quotiented out by the action of the gauge group
		\begin{equation}
					\G_{k+1,\delta}(N^*; 1) := \{u\in L^2_{k+1,loc}(N^*) \ \ | \ 1-u \in L^2_{k+1, \delta}(N^*), \ |u|=1\}		
		\end{equation}
 Here $A_\gotc$ denotes a connection that over the cylindrical end is translation invariant and equal to $B_\gotc$ and the Sobolev norm on $L_k^2(N^*, S^+)$ is computed using the connection $A_\gotc$.
		\end{lemma}	
				The proof of this Lemma is in \autoref{app:ProofofLemmaL2locL2}. 
				This characterization relies on the fact that $\pi_0(\B^\sigma_k(N;[\mathfrak{c}]))=\{*\}$, but can be easily modified to cover the
				general case by defining spaces $X^+_{z, k,\lambda_\gotc, \delta}(\psi_\gotc)$ dependent on the path $z$.

\newcommand{\ntop}{N^\mathrm{red}}
		\begin{proposition}\label{prop:DimOfModuli}
			 The formal dimension of $M(N^*; [\gota_{k,j}])$ is equal to 
					\begin{equation*}
						d(M(N^*; [\gota_{k,i}])) =\begin{cases}	  b_1(Y)  - \indf [\gota_{k}] -2i -1&\  \text{for $i\geq 0$}\\
																  b_1(Y) - \indf [\gota_{k}] -2i -2 & \ \text{for $i< 0$}\\
						\end{cases}
					\end{equation*}
		\end{proposition}
where $\indf [\gota_{k}] $ is defined in \eqref{defIndfMorse}.

		\begin{proof}[Proof of Proposition~\ref{prop:DimOfModuli}]
				Consider $[\gota_{\ntop}]\in \B_k(Y)$,  the only critical point  such that $\indf [\gota_{\ntop}] = b_1(Y)$  (see Subsection~\ref{subsec:CriticalPoints}).
				  We will show that $d(M(N^*; [\gota_{\ntop,0}]) = -1$. The general case will then  follow from the additivity  of the index:
						\begin{spliteq}
					d(M(N^*; [\gota_{k,j}])) &= d(M(N^*; [\gota_{\ntop,0}]) + \gr([\gota_{\ntop,0}], [\gota_{k,j}]) \\
													 	& = d(M(N^*; [\gota_{\ntop,0}]) +\bar \gr([\gota_{\ntop,0}], [\gota_{k,j}]) - \epsilon(j),\\
						\end{spliteq}					 
					 where $\epsilon = 1$ if  $j<0$ and $0$ otherwise and $\bar \gr$ has been defined in Proposition~\ref{prop:bargr}.
					Since reducible solutions constitute the boundary of the moduli space: $M^{\mathrm{red}}(N^*; [\gota_{\ntop,0}]) = \partial M(N^*; [\gota_{\ntop,0}])$, it is enough to show that the virtual dimension of $M^{\mathrm{red}}(N^*; [\gota_{\ntop,0}])$ is equal to $-2$.
							
						\newcommand{\SWmap}{\cF^\sigma}
						\newcommand{\pertp}{\mathfrak{p}}

			 Denote the components of  $\gota_{\ntop, 0}$ as  $\gota_{\ntop, 0} = (B_{\infty}, 0, \psi_{\infty})\in \Conf^\sigma_k(Y)$.
		 	Since $\gota_{\ntop, 0}$ is  non-degenerate we can invoke Lemma~\ref{lemma:FromL2loctoL2}, so that $M^{\mathrm{red}}(N^*; [\gota_{\ntop,0}])$ consists of 
		 	equivalence classes of  perturbed monopoles $(B_0 + a\otimes 1_{S^+}, \R_+ \Phi)$  where $B_0 $ is a smooth connection 
			equal to $B_{\infty}$   over the end and  $(a, \Phi) \in L^2_{k,\delta}(N^*; i\Lambda^1(N^*))\times X^+_{k,\lambda_0,\delta}(\psi_\infty)$. Here  $\lambda_0 >0$ is the eigenvalue associated to $\psi_\infty$ and $\delta>0$ is arbitrarily small. 
			Lemma~\ref{lemma:FromL2loctoL2}, will allow us to pass from $L^2_{k,loc}$ to $L^2_{k}$ sections, 	this is necessary step since we want  to invoke the  Atiyah-Patodi-Singer's index theorem \cite{APS1} later.

			\newcommand{\linearActadj}{\mathbf{d}^{\sigma,\dagger}}
			\begin{lemma}\label{Lemma:OperatorQDimension} The formal dimension of  $M^{\mathrm{red}}(N^*, [\gota_{\ntop,0}])$ is equal to $\ind_{L^2} Q-1$ where $Q$ is  the operator 
			\begin{spliteq}
				  Q &:   L^2_k(N^*; i\Lambda^1(N^*)\oplus S^+)\to L^2_{k-1}(N^*; i \Lambda^2_+\oplus  S^-\oplus i\R) \\
				 	& (a, \Phi) \mapsto ( d^+ a  + \Differential_\gamma \hat \pertp^0 (a), D^+_{B_0} \Phi + \Differential_{\gamma} \hat \pertp^1( \Phi), -d^*a),
			\end{spliteq}
			 and $\gamma=(B_0, 0)\in \Conf_{k,loc}(N^*)$.
			\end{lemma}
			\begin{proof} The formal dimension is given by the index of the Fredholm operator obtained by linearizing the equations \eqref{eq:monopoleCyl}  and adding a gauge fixing condition.

			Pick $ \Phi_0 \in X^+_{k,\lambda_0, \delta}(\psi_\infty) $,
			 then the linearization of
			\eqref{eq:monopoleCyl} at $(B_0, \Phi_0)$ is conjugated  via   \eqref{def:diffeoL2kAndX} to 
			\begin{spliteq}
				  \tilde Q :  \ \ & L^2_{k,\delta}(N^*; i\Lambda^1N^*)\oplus L^2_{k, \lambda_0+\delta}(N^*; S^+) 
				  \to L^2_{k-1,\delta}(N^*; i\Lambda^1N^*)\oplus L^2_{k-1, \lambda_0+\delta}(N^*; S^-) \\
				 		&  (a, \Phi) \mapsto (d^+ a  + \Differential_\gamma 	\hat\pertp^0 (a), D^+_{B_0} \Phi + \Differential_{\gamma}\hat \pertp^1(\Phi)) + K(a, \Phi)
			\end{spliteq}
			where  $K$ is  a compact operator (see also \autoref{rem:ConjugationOfDiracOp}).
			Choosing a different $(B_0,\Phi_0)$ will perturb $\tilde Q$ by a compact operator (hence irrelevant for the  computation of  the index)
		 due to compact embedding theorems for weighted Sobolev spaces.

			To take care of the action of the gauge group, 	we add a gauge-fixing condition, i.e. we consider the operator $\tilde Q \oplus \linearActadj_\gamma$ where 	 $\linearActadj_\gamma:L^2_{k,\delta}( i\Lambda^1(N^*)\oplus S^+)\to L^2_{k-1,\delta}(i\R) $ 
			is the operator defined at  \cite[pg. 143]{KM}; the $L^2$-kernel of $\linearActadj_\gamma$ is the tangent space to the Coulomb slice at $\gamma$. 
			$\tilde Q \oplus \linearActadj_\gamma$ is equal, modulo compact operators,  to the operator 
			\begin{spliteq*}\label{eq:QExpDecay}
				  Q &:  L^2_{k,\delta}(N^*; i\Lambda^1N^*)\oplus L^2_{k, \lambda_0+\delta}(N^*; S^+) \to L^2_{k-1,\delta}(N^*; i\Lambda^1N^*)\oplus L^2_{k-1, \lambda_0+\delta}(N^*; S^-) \oplus L^2_{k-1,\delta}(N^*; i\R) \\
				 	& (a, \Phi) \mapsto ( d^+ a  + \Differential_\gamma \pertp^0 (a), D^+_{B_0} \Phi + \Differential_{\gamma} \pertp^1( \Phi), -d^*a) .
			\end{spliteq*}
$Q$  is not quite the operator appearing in the thesis because of the \emph{weighted} Sobolev spaces.
			Since $\lambda_0$ is the first positive eigenvalue of the limiting perturbed Dirac operator and the latter has simple spectrum, by choosing $\delta>0$ small, 	 we have that $\ind_{L_{\lambda_0+\delta}^2} (D^+_{B_0} + \Differential_{\gamma} \pertp^1) = \ind_{L^2}(D^+_{B_0}+ \Differential_{\gamma} \pertp^1) -2$ (see also the proof of Lemma~\ref{lemma:SpinorSolutionIndex}).			
			Furthermore, $\ind_{L_{\delta}^2}  (d^+  + \Differential_\gamma \pertp^0, -d^*)= \ind_{L^2}  (d^+ 	  + \Differential_\gamma \pertp^0, -d^*) +1$ because the exponential decay cuts off an $\R$-summand in the cokernel due to constant functions. 
			In conclusion the difference between the above operator $Q$ and the operator $Q$ in the thesis is $-1$.
			\end{proof}
			
			In view of  Lemma~\ref{Lemma:OperatorQDimension},  in order to conclude the proof we will show that the $L^2$-index of $Q$ is equal to $-1$. 
			
			\begin{lemma} \label{Lemma:IndexQIsIndASDDirac}
				\begin{equation}
			 	\ind\ Q = \ind\  ASD   \oplus D^+_{B_0}
		 	\end{equation}
		 	where $ASD =  d^+ - d^*$ is the anti-self duality operator.
			\end{lemma}
			\begin{proof}
			 Recall that $D_{B_\infty}:L^2_k(Y;S)\to L^2_{k-1}(Y;S)$ is invertible by our choice of  metric and  we choose  our perturbation $\pert$ so small in the Banach space of perturbations 
			 that  $D_{B_\infty,\pert}$ remains injective, consequently, the spectral flow will be zero and the index of $D^+_{B_0} + \Differential_{\gamma} \pertp^1$
			 will be equal to the index of $D^+_{B_0}$ alone. 			
			 Therefore 
			 \begin{equation}
			 	\ind\ Q = \ind\  \big (ASD  + \Differential_\gamma \pertp^0  \big) \oplus D^+_{B_0}
		 	\end{equation} where $ASD =  d^+ - d^*$ is the anti-self duality operator. 
		 	
		 	The rest of the proof deals with $ASD  + \Differential_\gamma \pertp^0 $.
		 	\newcommand{\Hessf}{\Hess_{B_\infty}^f}
		 	Let  $\mathcal{H}^p$ denote the immaginary valued harmonic $p$-forms over $Y$ and let  $\Hessf: \cH^1\to \cH^1$  be the Hessian at zero of the Morse function 
		 	\begin{equation}
		 	 \alpha \in \cH^1\mapsto f (B_\infty + \alpha)\in \R.
		 	\end{equation}
		 	
			 The operator $\Differential_\gamma \pertp^0:L^2_k(N^*;i\Lambda^1)\to L^2_{k-1}(N^*; i\Lambda^2_+) $  is supported over the cylindrical end and acts on 1-forms as
			 \begin{equation}\label{eq:Diffpertp}
					 \Differential_\gamma \pertp^0 (a) (t) =dt\wedge \beta(t) (\Hessf)(\check a_{harm}(t)) + dt\wedge \beta_0(t)\Differential_{\check \gamma(t)}\pert_0^0(\check a(t))
			\end{equation}			  
			 (see the discussion of perturbations in \autoref{recap:4DPerturbations}). Since $\beta_0$ is compactly supported  and $\Differential_{\check \gamma(t)}\pert_0^0$ is compact ($\pert_0$ is a tame perturbation), $\Differential_{\gamma}(\beta_0\hat \pert_0^0)$ is a compact operator over the cylinder hence  is not relevant for index computions. We will thus assume from now on that $\pert_0^0= 0$.

			Equation \eqref{eq:Diffpertp} tells us that the operator $ ASD  + \Differential_\gamma \pertp^0$ is  an Atiyah-Patodi-Singer operator \cite{APS1, NicolaescuBook}, and after the usual identifications, the operator over the slice $\{t\}\times Y$
			takes the form (see for example \cite[pg. 312]{NicolaescuBook})
			\begin{equation}
				\partial_t + SIGN			+ \beta(t) \Hess f \circ \mathbb{P}_{\ker \Delta}
			\end{equation}
				where 
				\begin{equation}
				SIGN  =  \begin{bmatrix}
									& *d & -d\\
									& -d^* & 0 \\
								\end{bmatrix} : L^2_k(Y; \Lambda^1\oplus \Lambda^0)\to  L^2_{k-1}(Y; \Lambda^1\oplus \Lambda^0) 
				\end{equation} is the odd signature operator. 
			The gluing formulas for the index \cite{APS1} give: 
			\begin{spliteq}\label{eq:monopoleCyl4}
			\ind \ (ASD  + \Differential_\gamma \pertp^0  \big) = &
			 \ind \ (ASD) + \dim \ker (SIGN) \\
			 & + \ind \left(\partial_t + SIGN	+ \beta(t) \Hessf \circ \mathbb{P}_{\ker \Delta}\right)\\
			\end{spliteq}
			where the last term is the index over the infinite cylinder $\R\times Y$.
			
			 Clearly $\ker (SIGN)  \simeq H^0(Y)+H^1(Y)$, we claim that the last summand in \eqref{eq:monopoleCyl4} is equal to  $-1 - b_1(Y)$.
			This will imply that  $\ind \ (ASD  + \Differential_\gamma \pertp^0  \big) = \ind \ (ASD)$ and conclude the proof of the lemma.
			The operator  $SIGN + \beta(t)\Hessf \circ \mathbb{P}_{\ker \Delta}$ splits as 
			\begin{equation*} 
						\begin{bmatrix}
							& 0 & 0 & 0\\
							& 0 & \beta(t) \Hessf   & 0\\
							& 0& 0 & SIGN'\\
						\end{bmatrix}:\begin{matrix}
												&\mathcal{H}^0 \\	
												& \oplus \\
												&   \mathcal{H}^1\\
												& \oplus \\
												& (\ker SIGN)^\perp \\
\end{matrix}						   \to \begin{matrix}
												&\mathcal{H}^0 \\	
												& \oplus \\
												&   \mathcal{H}^1\\
												& \oplus \\
												& (\ker SIGN)^\perp, \\
\end{matrix}	
			\end{equation*}
			with $SIGN'$ being the restriction of $SIGN$.
			It follows that
			\begin{equation*}
				\ind \left(\partial_t + SIGN + \beta(t) \Hessf	 \circ \mathbb{P}_{\ker \Delta}\right) = \ind\  \left(\partial_t +\begin{bmatrix}
							& 0 & 0 \\
							& 0 & SIGN'\\
						\end{bmatrix}\right) + \ind \ \left(\partial_t + \beta(t)\Hessf\right).
			\end{equation*}
			where all the operators are over the infinite cylinder $\R\times Y$.
			 We recall the the Atiyah-Patodi-Singer index is equal to the dimension  the $L^2$-kernel minus the dimension of the $L^2_{ex}$ kernel of the adjoint, where $L^2_{ex}$ denotes the extended $L^2$ sections
			\cite[Corollary 3.14]{APS1}.
			  We can easily show by separation of variables that $\ker_{L^2} (\partial_t +\begin{bmatrix}
							& 0 & 0 \\
							& 0 & SIGN'\\
						\end{bmatrix})= \{0\}$  whilst $\ker_{L^2_{ex}} (\partial_t -\begin{bmatrix}
							& 0 & 0 \\
							& 0 & SIGN'\\
						\end{bmatrix})$ has dimension one, given by the constant solutions in $\mathcal{H}^0$.
			Similarly 			$\ker_{L^2} \left(\partial_t + \beta(t) \Hessf \right) = \{0\}$ because $\beta(t) = 0$ for $t<0$ forces a solution to be constant, hence it cannot be in $L^2$ unless it is trivial. 
			On the other hand an element in $\ker_{L^2_{ex}}\left(\partial_t - \beta(t) \Hessf \right)$ will be constant as $t\to -\infty$ and $O(e^{\lambda t})$ for $t\to +\infty$ where $\lambda $ is an eigenvalue of $\Hessf$ therefore the $L^2_{ex}$-kernel has dimension equal to the number of negative eigenvalues of the Hessian, which is $b_1(Y)$ by construction. 
			This concludes the proof of the claim and hence the proof of Lemma~\ref{Lemma:IndexQIsIndASDDirac}.
			\end{proof}			
						
				We have shown that $\ind \ Q = \ind (APS \oplus D^+_{B_0})$.
				The APS index of such operator is (see for example \cite[pg. 312]{NicolaescuBook}) equal to 
				\begin{equation*}
				 \frac {1} {4} (c_1^2({\sstruc}) - 2 \chi(N_0) - 3 \sigma(N_0)) - \frac{b_1(Y) +1}{2}    - \dim (\ker_{\C} D_{B_\infty})   - \eta (D_{B_\infty}) - \frac{1}{4} \eta (SIGN)
				\end{equation*}
				The $\eta$-invariant of the Dirac operator $D_{B_\infty}$ and that of the signature operator vanish because we have an orientation reversing isometry of $Y$ that fixes
				$B_\infty$ given by $\rho_B$ defined in Subsection~\ref{subsec:involutions}.
				In addition, $\ker_{\C} D_{B_\infty} =0 $ because we used a metric satisfying \autoref{lemma:DiracInjective} and   $\chi(N_0) = 1- b_1(Y)$, $c_1^2({\sstruc}) = \sigma(N_0)=0$.
				Hence $\ind \ Q  = -1$.
				This concludes the proof of Proposition~\ref{prop:DimOfModuli}.
		\end{proof}		
		
\subsection{Pointlike moduli spaces.}\label{subsec:PointlikeModuli} 
	\newcommand{\pertp}{{\mathfrak{p}}}
	  \begin{plan} We recall that we defined $ [\gota_{\ntop}]\in \torus $ to be a critical  point of $f_\torus$ of maximal index on the torus, i.e.   	 $\indf [\gota_{\ntop}] = \dim \torus = b_1(Y)$.  The aim of this subsection is to prove
	   Theorem~\ref{thm:ModuliIsPoint}.
	   
	   We know that the moduli space, if not-empty,  has dimension zero from Proposition~\ref{prop:DimOfModuli}. Lemma~\ref{lemma:RegularPerturbationInPertZero} 
	 will provide us with perturbations making the moduli spaces regular without perturbing the connection equation. 
	 Then we will solve the connection equation (Lemma~\ref{lemma:MorsePertCanBeOmmitted}, Lemma~\ref{lemma:uniqueSolutionForPert0Zero}) showing that it  has
	 a unique solution modulo gauge equivalence. In the second part of the proof we will deal with the spinor equation. The main difficulty here is to show that
	 the solution space is not empty. In fact by Lemma~\ref{lemma:SpinorSolutionIndex} it is easy to show that the equation will have a solution but this is not enough as we need solutions with the correct asymptotics.
	 Here the perturbation chosen in the previous step will play a crucial role, since will allow us to pass from a general asymptotically cylindrical operator to a cylindrical one.
	This will improve greatly our understanding of the behaviour of the solutions at infinity. 
	   	\end{plan}

	   \paragraph{} We recall that the subspace of perturbations $\PertZero$ has been defined in Definition~\ref{def:PertZero} and the structure of perturbations of 4D moduli spaces has been reviewed in \autoref{recap:4DPerturbations}.
	We restate Theorem~\ref{thm:ModuliIsPoint} from the introduction more precisely in view of the notation introduced previously.
\begingroup
\def\thetheorem{\ref{thm:ModuliIsPoint}}
\begin{theorem}
There is a subset of perturbations $S \subset \PertZero$, residual in $\PertZero$, such that for any $\pert_0\in S$, the perturbation $\pertp = \beta_0(t)\pert_0+\beta(t) \pert$ is such that  the moduli space $M(N^*; [\gota_{\ntop,-1}]) $ is  regular and consists of a single point, in particular is not empty.
				Moreover, if $\Phi: N\to P_0(\Gamma, \mathcal{R})$ is a diffeomorphism of $N$ with a realization of a protocork, then the same statement above holds for the moduli space over the reflection of the protocork
				$M(P_1(\Gamma,\mathcal{R})^*; [\gota_{\ntop,-1}]) $ where the identification $Y \simeq \partial P_1(\Gamma, \mathcal{R})$ is given by $\Phi|_Y$.

\end{theorem}
\addtocounter{theorem}{-1}
\endgroup

		\begin{proof} Let $(B_{\infty}, 0, \psi_{-1})\in \Conf^\sigma_k(Y)$ be a smooth representative of  $[\gota_{\ntop,-1}]$. 
		Since by Proposition~\ref{propProtocork} $H^1(N)\to H^1(Y)$ is an isomorphism we can also assume without loss of generality that exists a \emph{flat} connection $A_0 \in \cA(N^*)$ 
		that restricts over the end to a translation invariant connection equal to $B_\infty$ slicewise. $A_0$ is unique up to gauge-transformation. 
		Since $[\gota_{\ntop,-1}]$ is a boundary \emph{unstable} critical point	$M(N_0^*; [\gota_{\ntop,-1}]) $ consists \emph{entirely of reducible} monopoles. 
		We can therefore apply Lemma~\ref{lemma:FromL2loctoL2}, which tells us that the elements of the moduli space can be represented by pairs 
		$(A,\Phi)\in \left(A_0 + L^2_{k,\delta}(N^*; i\Lambda^1N^*)\right)\times X_{k, \lambda_{-1}, \delta}^+(\psi_{-1})$	for some $\delta>0$.
		Firstly we focus on the first equation of 	\eqref{eq:monopoleCyl}. 
		Making apparent the perturbative terms this equation is :		\begin{equation}\label{eq:monopoleCyl2_1}
				 (F_{A^t_0}^+  + d^+a) +  \beta_0 \ \hat\pert_0^0(A_0+a,0) + \beta \ \hat\pert^0(A_0+a, 0)  = 0
		\end{equation}		 
		where $a := A-A_0 \in L^2_{k,\delta}(N^*, i\Lambda^1N^*)$.
		
		Now we invoke Lemma~\ref{lemma:MorsePertCanBeOmmitted} below, that  tells us that  we can ignore the perturbation $\beta \pert^0$  so that the equation becomes
		\begin{spliteq}\label{eq:monopoleCyl2_22}
				&  (F_{A^t_0}^+  + d^+a) +  \beta_0 \ \hat \pert_0^0(A_0 +a) =  0.\\
		\end{spliteq}	
		 This is false in general and relies on our assumption that $ [\gota_{\ntop}]\in \torus $ is a critical  point of $f_\torus$ of maximal index.

		\begin{lemma}[The Morselike perturbation can be omitted]\label{lemma:MorsePertCanBeOmmitted} Regardless of $\pert_0$, the moduli space of solutions to \eqref{eq:monopoleCyl2_1} is equal to the moduli space of $L^2_{k,\delta}$-solutions to 
		\begin{spliteq}\label{eq:monopoleCyl2_2}
				&  (F_{A^t_0}^+  + d^+a) +  \beta_0 \ \hat \pert_0^0(A_0 +a) =  0\\
		\end{spliteq}	
		\end{lemma}		
		\begin{proof} \newcommand{\gradf}{V}
		Recall that  $\pert^0$ is the formal gradient of the Morselike perturbation $f: \Conf(Y)\to \R$.  Denote  by $\cH^1\subset L^2(Y;i\Lambda^1Y)$ the space of immaginary valued 1-forms.
		Define $\gradf: \cH^1\to  \cH^1$ to be the $L^2$-gradient of the function  
		\begin{equation}
			b\mapsto f(B_\infty + b,0):\cH^1\to \R.
		\end{equation}	
		Then since Morselike perturbations are pulled back from the torus of flat connections we have that 
		\begin{equation}
				\pert^0(B_\infty + b) = \gradf(b_{harm})\in \cH^1
		\end{equation}
		for any $b\in L^2(Y; i\Lambda^1Y)$, here $b_{harm}= \PP_{\ker \Delta}(b)$ is the $L^2$-projection to $\cH^1$.
	         Then over the end $Z := [-C/2, +\infty)\times Y$ with the usual identifications in place \cite[Section 4.3]{KM},  \eqref{eq:monopoleCyl2_1} takes the form 
	      \begin{equation}
	      \partial_t \check a(t)  = - *d (\check a(t)) + dc(t) - \beta(t)\gradf(\check a_{harm}(t)).
	      \end{equation}
	      where $a(t)  = \check{a}(t) + c(t) dt$, $\check a(t) \in L^2_{k,\delta}(Y; i \Lambda Y)$; here we exploited that  $A_0$ is flat and   $\beta_0([-C/2,+\infty)) = \{0\}$.
	      With respect to the  Hodge decomposition $L^2(Y; i\R)\simeq Im \ d \oplus Im \ d^*  \oplus \mathcal{H}^1$,   the equation splits as :
			 \begin{spliteq}\label{eq:monopoleConn}
			  \partial_t \check a_{harm}(t)  &=  - \beta(t)\gradf(\check a_{harm}(t))\\
			  \partial_t \check a^\perp(t)  &= - *d (\check a^\perp(t)) + dc(t)\\
			 \end{spliteq}
		where $a^\perp(t)$ is the  component of $\check a(t)$ in  $Im \ d \oplus Im \ d^*$.
		This implies that $\check a_{harm}(t)\equiv 0$ for any $t\geq -C/2$, because otherwise $t\mapsto [B_\infty+\check a_{harm}(t)]\in \torus$ would be a non trivial trajectory of $-\grad f_\torus$ converging to $[B_\infty]$ which cannot be because $[B_\infty]$ is a	 maximum of the Morse function $f_\torus$ by assumption.
		It follows that solutions to \eqref{eq:monopoleCyl2_1} satisfy \eqref{eq:monopoleCyl2_2}.
		Vice versa, the harmonic part $\check a_{harm}(t)$ of a solution to \eqref{eq:monopoleCyl2_2} must be constant since it satisfies \eqref{eq:monopoleConn} with $\beta(t) = 0$ for any $t$, however
		such a connection cannot be in $L^2_k$  unless $ \check a_{harm}(t) = 0$ for all $t\geq -C/2$.
		The proof that the Morselike perturbation can be omitted is concluded.
		\end{proof}
		 The next lemma, Lemma~\ref{lemma:uniqueSolutionForPert0Zero} shows that if $\pert^0_0 = 0$ then there is a unique solution to  \eqref{eq:monopoleCyl2_22}  modulo gauge, given by $a= 0$.	
		 This is enough, in fact Lemma~\ref{lemma:RegularPerturbationInPertZero} below ensures the existence of a residual set of perturbations  $\pert_0\in \PertZero$ (so that $\pert_0^0  =0$) 
		 that makes the moduli spaces regular. 
		\begin{lemma}[Unique solution for connection eq.]\label{lemma:uniqueSolutionForPert0Zero} If  $\pert_0^0 = 0$ there is a unique solution to \eqref{eq:monopoleCyl2_22}  modulo gauge.
		\end{lemma}
		\begin{proof}
		When $\pert_0^0 = 0$, then \eqref{eq:monopoleCyl2_2} is the ASD equation over $N^*$. 
		First of all we show that  $ A^t= A^t_0+a$ must be flat.  Indeed $A^t \in L^2_{k}$ limits exponentially to flat connections over the end
			therefore its curvature $F_{A^t} \in L^2_{k-1}$.
			Since $F_{A^t}$ is anti-self-dual and closed it  is also harmonic;
			the space of harmonic self dual forms has dimension $b^-(N) = 0$  \cite{APS1}, hence  $F_{A^t} = 0$ as we wanted. Since $H^1(N) \to H^1(Y)$ is an isomorphism,  we have a \emph{unique} gauge-equivalence class of flat connections extending $B_\infty$ (given by $A_0$).
		\end{proof}
		
		\begin{lemma}[Perturbations]\label{lemma:RegularPerturbationInPertZero} The set of perturbations $\pert_0\in \PertZero$ such that the moduli spaces 
			 $M^{\mathrm{red}}(N^*; [\gota_{\ntop,j}]) $ are  regular $\forall j\in \Z$  is residual in $\PertZero$.
		\end{lemma}
		Essentially this lemma tells us that in our specific case, in particular the moduli space in the blow-down is regular,  we can improve the transversality result  \cite[Section 24.4.7]{KM} from $\pert_0 \in \Pert$ to $\pert_0 \in \PertZero$. Notice that the lemma is only about \emph{reducible} monopoles. The proof will be a modification of \cite[Proposition 24.4.7]{KM}, notice however that our hypothesis are not strong enough to prove a generalization 
		of the density result \cite[Lemma 24.4.8]{KM}, we can though use some ideas used to prove \cite[Lemma 24.4.8]{KM} to show a partial result on the spinor component and then use our specific hypothesis to 
		deal with the connection component.
		\begin{proof}
		\newcommand{\paramModuli}{\mathfrak{M}^{red}(N)}
		\newcommand{\ppert}{{\mathfrak{q}_0}}
		\newcommand{\mapM}{\mathfrak{M}^\partial}
		\newcommand{\cc}{\mathfrak{c}}
		\renewcommand{\Q}{\mathfrak{Q}}
 
		\newcommand{\Tangent}{\cT^\sigma}
		\newcommand{\SWmap}{\cF}
		Let $\paramModuli\subset \partial\B_k^\sigma(N) \times \PertZero$ be the parametrized moduli space, i.e. the quotient of the zero locus of 
		$\mapM =  (\gamma, \ppert)\mapsto \SWmap_\ppert(\gamma)$ by the action of the gauge group  where with an abuse of language we wrote $\SWmap_\ppert$ to denote the Seiberg-Witten map
		with perturbation $\beta_0\pert_0 + \beta \pert$.
		Notice that this is a moduli space over the compact $N$ not $N^*$.
		Put $Z = [0,+\infty)\times Y$, and let $R_+: \paramModuli \to \partial \B_{k-\frac 1 2}^\sigma(Y) $  and $R_- : M^{red}(Z; [\gota_{\ntop, j}])\to \partial \B_{k-\frac 1 2}^\sigma(Y)$ denote the
		restriction to the boundary $\{0\}\times Y$.
		We have to show that the fiber product is regular, this is equivalent to showing that  for each element $(([\alpha],\ppert),[\beta])$ in the fiber product, the map 
		\begin{equation}\label{eq:fiberProdReg}
			\Differential_{([\alpha],\ppert)} R_+ - \Differential_{[\beta]} R_- : T_{([\alpha],\ppert)}\paramModuli \oplus T_{[\beta]} M^{red}(Z; [\gota_{\ntop, j}]) \to T_{[\mathfrak{c}]}   \partial \B_{k-\frac 1 2}^\sigma(Y)
		\end{equation}
		is surjective. 
		
		Note that we already know that this map is Fredholm by  \cite[Lemma 24.4.1 ]{KM}, moreover using the regularity result   \cite[Lemma 17.2.9]{KM} we reduce to showing
		surjectivity in the case when $k=1$. Notice also that we can assume that, possibly after  gauge transformation, $\alpha|_{\partial N} = \beta|_{\partial Z}$ so that they glue to a configuration over $N^*$, this follows from  \cite[Lemma 24.2.2]{KM}.
		
		Let $\Q = \Differential_{(\alpha,\ppert)}\mapM\oplus \mathbf{d}_\alpha^{\sigma,\dagger}: (\Tangent_{1,\alpha}(N))^\partial \oplus T_\ppert \PertZero\to \V_{0}^\partial $ be the linearization of the Seiberg-Witten map with an added gauge-fixing condition, and similarly let  $\Q^\tau = \Differential_\beta\mathfrak{F}_\pert ^\tau \oplus \mathbf{d}_\beta^{\tau,\dagger} : (\mathcal{T}^\tau_{1,\beta}(Z))^\partial \to \V_{0}^{\tau,\partial} $.
		The kernel of these maps is isomorphic to $T_{([\alpha],\ppert)}\paramModuli $ and $ T_{[\beta]} M^{red}(Z; [\gota_{\ntop, j}])$ respectively, moreover both are surjective operators because $N$ and $Z$ have non-empty boundary and their formal adjoint enjoys the unique continuation property   \cite[Proposition 14.1.5]{KM}.
		Under the usual identifications, we see that  the surjectivity of \eqref{eq:fiberProdReg} is equivalent to showing that 
		the range of 

		\begin{equation}
		(\Q, r_+)- (\Q^\tau, r_-) : \begin{matrix}
										L^2_1(N; iT^*N \oplus S^+)\\ \oplus \\ \PertZero \\  \oplus\\  L^2_1(Z; iT^*Z\oplus S) 
											\end{matrix}\longrightarrow
											\begin{matrix}
											L^2(N; iT^*Y\oplus S^+ \oplus i\R) \\ \oplus\\  L^2(Z; iT^*Y\oplus S\oplus i\R)\\ \oplus \\ L^2_{1/2}(Y; iT^*Y \oplus S \oplus i\R).
											\end{matrix}
		\end{equation}
		contains all elements of the form  
			\begin{equation}
					W = \underbrace{(\omega_N, \psi_N,0)}_{W_N}\oplus\underbrace{(\omega_Z,\psi_Z,0)}_{W_Z}\oplus \underbrace{(v,\psi,0)}_{w}.
			\end{equation}
		In the above definition $r_+, r_-$ are the differentials of the restriction map.

		Suppose by contradiction that such a   $W$ is $L^2$-orthogonal to the range of $(\Q, r_+)- (\Q^\tau, r_-)$.
		By considering variations supported away from the boundary that leave the perturbation component untouched it can be shown, as in the proof of \cite[ Lemma 24.4.8 ]{KM}, 
		that $W_N|_{\partial N} = - w = W_Z|_{\partial Z}$.
		 Since $\Q^*W_N = 0$, $(\Q^\tau)^* W_Z = 0$, we have that   $W_N, W_Z \in L^2_1$. Hence $W_N$ and $W_Z$ can be glued to obtain an $L^2_1$ configuration $W_N\#W_Q$ over $N^*$  because their boundary values coincide.

		By considering gauge-orbit directions, we see that  $w$ is orthogonal to the gauge orbits and, proceeding as in the proof of  \cite[Lemma 15.1.4]{KM}, we obtain that   $W_N|_{\partial N} = - w$  implies that  also $\check W_N(t)$ is  orthogonal 
		 to the orbit for each $t$ in the collar and  never vanishing unless $W = 0$.
		 
		We are now in the position to run the same argument  in \cite[Section 15.1.4]{KM} and showing  that, if $\psi_N\neq 0$ then we can  find an element  $\dot\ppert \in \PertZero$ (the dot here is just to remind us that is a variation)  such that 
		\begin{equation}\label{eq:PertSpinor}
			\langle \beta_0\dot\ppert(\alpha), \psi_N \rangle_{L^2(N)}> 0.
		\end{equation}
		Indeed, the cylinder function constructed at \cite[pg. 271]{KM} is constant on the reducibles hence its gradient is in $\PertZero$.
		This however contradicts the  assumption that $W$ is orthogonal to the range of $(\Q, r_+)$ as can be seen by using directions tangent to $\PertZero$. Therefore $\psi_N = 0$.
		And consequently, by the unique continuation property, $\psi_N\#\psi_Z = 0$ hence $\psi_Z = 0$.

		 Notice that up to this point we did not use our specific hypothesis. We will use it to show that $\omega_N\#\omega_Z = 0$.
		Consider   variations    $(\dot a_N, \dot a_Z)\in L^2_1(N; iT^*N)\oplus L^2_1(Z; iT^*Z)$ such that  $\dot a_N|_{\partial N } = \dot a_Z |_{\partial Z}$ (the dot here is just our notation to keep track of the  variation).  		In this case   we can glue the forms to get
		$\dot a_N	\# \dot  a_Z \in L^2_1(N^*; T^* N^*)$ and   $r_+\dot a_N = r_- \dot a_Z$.
		Now  the orthogonality relation  can be rewritten as
		\begin{equation}\label{eq:orthogonalityJoint}
			0 = \langle  \Q (\dot a_N), \omega_N  \rangle_{L^2(N)} + \langle  \Q^\tau (\dot a_Z), \omega_Z\rangle_{L^2(Z)} = \langle  (d^+  + \beta \Differential_{A_0} \hat\pert^0) (\dot a_N\# \dot a_Z), \omega_N\#\omega_Z  \rangle_{L^2(N^*)}, 
		\end{equation}
		in fact  $\pertp^0= \beta_0 \ppert^0 + \beta \pert^0 = \beta \pert^0 $ over the reducibles because $\ppert\in \PertZero	$.
		
		We claim that \eqref{eq:orthogonalityJoint} cannot hold for all $\dot a_N \#\dot a_Z$ unless $\omega_N\#\omega_Z$ is zero.
		First of all, notice that the \emph{blown-down} moduli space  of solutions to 
		\begin{equation}
			F_{A_0+a}^+ + \beta \hat\pert^0(a) = 0
		\end{equation} limiting to $[\gota_{\ntop}] = \pi [\gota_{\ntop, j}]$ modulo gauge equivalence 
		is regular. Indeed  the deformation complex 
		\begin{equation}
		L^2_{2,ex}(N^*, i\R)\overset{-d}{\to} L^2_1(N^*; i\Lambda^1) \overset{d^++ \beta \Differential_{A_0}\hat\pert^0}{\to } L^2(N^*; i\Lambda^+)
		\end{equation}
		has index equal to $-\ind(ASD\oplus \beta\Differential_{A_0}\pert^0) = -\ind (ASD) $ by the same argument used in the proof of Lemma~\ref{Lemma:IndexQIsIndASDDirac} 
		and the latter is equal to 
		\begin{equation}
			-(\underbrace {b^3(N_0)}_{= 0} + \underbrace{b^+(N_0)}_{=0 } + \dim (Im(\underbrace{H^2(N_0)}_{ = (0)}\to H^2(Y)) +1 = 1.
		\end{equation}		
		On the other hand, denoting by $H^0_{ex}, H^1, H^2$ the cohomology groups of the  deformation complex, we have that
		$H_{ex}^0 \simeq \R $ due to constant functions and  $H^1 \simeq (0)$ because by Lemma~\ref{lemma:uniqueSolutionForPert0Zero} the moduli space is a point when $\pert_0^0 = 0$.
		Consequently  $ H^2 = (0)$ as well.  
		This shows that   the  second map in the deformation complex is surjective,
		but this is precisely the map in the second term  in \eqref{eq:orthogonalityJoint}, thus  $\omega_N\#\omega_Z = 0$ and the 
		proof of  Lemma~\ref{lemma:RegularPerturbationInPertZero} is concluded.
		\end{proof}
		
		Thanks to Lemma~\ref{lemma:RegularPerturbationInPertZero}, we can assume for the rest of the proof that  we are working with a perturbation of the form $\pertp = \beta_0 \pert_0 + \beta \pert$ where $\pert_0^0 = 0$ such that the moduli spaces		 $M(N^*; [\gota_{\ntop,j}])$ are regular $\forall j\in \Z$. Let $A_0$ denote the unique  solution (modulo gauge equivalence) to \eqref{eq:monopoleCyl2_1}, 
		we can assume that $A_0 $ is constant over the end, and equal to $B_\infty$
		In order to conclude the proof of \autoref{thm:ModuliIsPoint}, we need to show that the 
			the second equation  of \eqref{eq:monopoleCyl}, i.e. the one involving the spinor,  has a unique  solution.
			This is not obvious mainly  because the moduli space can be both regular and empty.

			The second equation of \eqref{eq:monopoleCyl}  reads
			
			\begin{equation}\label{eq:monopoleSpinorCyl1}
			D^+_{A_0} \Phi  + \beta_0 \Differential_{(A_0 , 0)} \hat\pert_0^1(\Phi)	 +\beta \ \Differential_{(A_0, 0)} \hat\pert^1(\Phi) = 0\\
			\end{equation}
			where $\Phi \in X^+_{k,\lambda_{-1}, \delta}(\psi_{-1})$.
			Set $\mu  := \lambda_{-1}-\epsilon$, where $\epsilon>0$ is so small that $\lambda_{-1}-\epsilon>\lambda_{-2}$.
			The space of $L^2_{k,\mu}$-solutions  to \eqref{eq:monopoleSpinorCyl1}  contains the $X^+_{k,\lambda_{-1},\delta}(\psi_{-1})$-solutions
			because $\psi_{-1}$ is smooth, hence $X^+_{k,\lambda_{-1}, \delta}(\psi_{-1})\subset e^{-\lambda_{-1}(\cdot)\beta(\cdot)}(L^\infty_{k} + L^2_{k, \delta}) \subset L^2_{k,\mu}$.
			 As a first step, we will show that the space of $L^2_{k,\mu}$-solutions to \eqref{eq:monopoleSpinorCyl1}  is not empty. 
			This follows from the fact that the $L^2_{k,\mu}$-index of the operator in \eqref{eq:monopoleSpinorCyl1} is positive, hence the kernel must be nontrivial.	
			\begin{lemma}\label{lemma:SpinorSolutionIndex} The (real) index of the operator in \eqref{eq:monopoleSpinorCyl1}   as an operator $L^2_{k,\mu}(N^*, S^+) \to L^2_{k-1,\mu}(N^*, S^-)$
			is equal to $2$. 
			\end{lemma}
			\begin{proof} Let $W_\mu:N^*\to \R $  be  a smooth weight  function 
			equal to $1$ on $N$ and equal to $ e^{-2\mu(\cdot)}$ on the infinite cylinder.
			First of all, notice that modulo compact terms and continuous perturbations in the space of Fredholm operators $L^2_{k,\delta}\to L^2_{k-1,\delta}$, 
			the operator of \eqref{eq:monopoleSpinorCyl1}	becomes 
			$D^+_{A_0}   + \beta \ \Differential_{(B_\infty , 0)}\pert^1 $.
			Multiplication by $W_\mu^{-1}$ 	 gives an isometry  $L^2_k\to L^2_{k, \mu}$ under which the operator $D^+_{A_0}   + \beta \ \Differential_{(B_\infty , 0)} \pert^1:L^2_{k,\mu}\to L^2_{k-1,\mu} $ 
			is conjugated to $D^+_{A_0}   + \beta \ \Differential_{(B_\infty , 0)} \pert^1 + \beta(t-2) (-\lambda_{-1}+\epsilon): L^2_k\to L^2_{k-1}$ modulo a compact operator 
			due to the behaviour of $W_\mu^{-1}$ on the complement of the cylinder. 
			It follows that 
			\begin{spliteq*}
				\ind_{L^2_{k,\mu}}(D^+_{A_0}   + \beta \ \Differential_{(B_\infty , 0)}) &= \ind_{L^2}\left(D^+_{A_0}   + \beta \ \Differential_{(B_\infty , 0)} \pert^1+ \beta(t-2) (-\lambda_{-1}+\epsilon)\right)\\
								& = \ind_{L^2}(D^+_{A_0}) + sf\left \{D_{B_\infty}+ \beta(t) D_{(B_\infty,0)}\pert^1\right\}_t \\
								& \phantom{= }+ sf\left\{D_{B_\infty}+  D_{(B_\infty,0)}\pert^1 + \beta(t)(-\lambda_{-1}+\epsilon)\right\}_t\\
			\end{spliteq*}
			Where we have used the usual gluing formulas for the index \cite{APS1}, and $sf\{A_t\}_t$ denotes the  spectral flow of the family of operators $\{A_t\}_t$ (for a definition see \cite[pg. 244]{KM}  or \cite[Section 4.1.3]{NicolaescuBook}).
			Now, $\ind_{L^2}(D^+_{A_0}) = 0$ because $D_{B_\infty}$ is injective and $\eta(D_{B_\infty}) =0$ due to the orientation reversing isometry $\rho_B$ of $Y$ (see Subsection~\ref{subsec:involutions}).
			$sf\{D_{B_\infty}+ \beta(t) D_{(B_\infty,0)}\pert^1\}_t = 0$ because we can choose a priori a  perturbation so small that the operator is injective for any $t$, therefore we have no crossings.
			We claim that $sf\left\{D_{B_\infty}+  D_{(B_\infty,0)}\pert^1 + \beta(t)(-\lambda_{-1}+\epsilon)\right\}_t = 2$.
			Indeed we only have one crossing due to the flow of the eigenvalue $\lambda_{-1}$, which goes from negative ($\lambda_{-1}<0$) to positive ($\epsilon>0$), thus the crossing has positive sign and 
			the spectral flow equals the real dimension of the $\lambda_{-1}$-eigenspace which is $2$ since the spectrum is simple over $\C$.
			\end{proof}

			This shows that  the space of  $L^2_{k\,\mu}$-solutions to \eqref{prop:DimOfModuli} is not empty.
			The second step consists in showing that up to multiplication by an element in $\C^\times$, these solutions lie in $X_{k, \lambda_{-1}, \delta}^+(\psi_{-1})$.
			Observe that  over $[0,+\infty)\times Y$,  \eqref{eq:monopoleSpinorCyl1} takes the form
			\begin{equation}\label{eq:monOverEndPerturbed}
			\partial_t \Phi(t)  = - D_{B_\infty,\pert} \Phi(t) 
			\end{equation}
			 where the right hand side is defined by  \eqref{def:perturbedDirac}. 
			Consequently, by separation of variables, a general {$L^2_{k,loc}$-solution} of \eqref{eq:monOverEndPerturbed} takes the form  of a linear combination
			\begin{equation}\label{eq:expansion}
			\Phi(t) =   \sum_{i\in \Z} c_i e^{-\lambda_i t} \psi_{i}
			\end{equation}
			where $\psi_i$ is a $\lambda_i$-eigenvector of $D_{B_\infty,\pert} $, $c_i\in \C$. Notice that here we used that, thanks to our earlier discussion, the connection is translation invariant over the tube.
			Let $\Phi\neq 0$ be a solution given by step one above. Since $\Phi \in L^2_{k,\mu}$  
			only those $i$ with $\lambda_i>\mu$ can appear in the expansion \eqref{eq:expansion}, thus since   $\mu = \lambda_{-1} -\epsilon$ and $\mu>\lambda_{-2}$ we obtain that
			\begin{equation}
			\Phi(t) =   \sum_{i\geq -1} c_i e^{-\lambda_i t} \psi_{i}.
			\end{equation}
			Notice that if $c_{-1}\neq 0$ then $\frac 1{c_{-1}} \Phi \in X_{k, \lambda_{-1}, \delta}^+(\psi_{-1}) $.
			We will show that $c_{-1}= 0$ is impossible. Indeed, suppose by contradiction that $c_{-1}=0$,  then  $\Phi(t) \in L^2_k$, and thus by Lemma~\ref{lemma:FromL2loctoL2} we obtain an element of $M^{\mathrm{red}}(N^*, [\gota_{\ntop,j}]) $, $j\geq 0$.
			This would show that $M^{\mathrm{red}}(N^*;[\gota_{\ntop,j}]) \neq \emptyset$ for some $j\geq 0$, but this is not possible because $M^{\mathrm{red}}(N^*;[\gota_{\ntop,j}])$ is a regular moduli space, thanks to  our assumption on $\pertp$, with a negative formal dimension thanks to Proposition~\ref{prop:DimOfModuli}. 
			
			Consequently,  rescaling by  $\C^\times$ the vector space of of solutions, isomorphic to $\C$,  guaranteed by the index theorem, we obtain one solution in $X_{k, \lambda_{-1}, \delta}^+(\psi_{-1}) $.
			Notice that we cannot have more solutions because solutions to the spinor equation form a vector space, therefore they   would make the dimension of the moduli space higher that the formal dimension which is zero.
			This shows that the moduli space is not empty and is a point.
			The part of the thesis about the reflection  $P_1(\Gamma,\mathcal{R})$ follows by the same argument used for $N$ as all the key properties of $P_0(\Gamma,\cR)$ are satisfied also by its
			reflection.
			This concludes the proof of Theorem~\ref{thm:ModuliIsPoint}.
		\end{proof}


\appendix
\section{Proof of Proposition~\ref{prop:MorseLikePerturbationPerp}.}\label{appendix:ProofMorseLike}

For the sake of this appendix, $f$ will denote a Morselike perturbation.
\paragraph{Proof sketch.} The simpler part of the proof is to achieve non-degeneracy of the critical points. In fact the proof of the book without modification gives perturbations in $\Pert^\perp$ in our setting (see Lemma~\ref{app:Lemma1MorseLike} below). The more delicate part is the regularity of the moduli spaces. 
The perturbations used in the book to achieve such result are  perturbations vanishing in a fixed neighbourhood of the critical points. 
The natural thing to do would be to use the subset of such perturbations with connection-component vanishing on the reducible locus (i.e. in $\PertZero$). 
The problem is that this is not sufficient to prove regularity in the case of \emph{irreducible} trajectories between two \emph{reducible} critical points. 
In this case we have to allow our perturbations to be non-constant on the reducible locus. At the same time it is important to us to obtain in the end a perturbation in  $h+\PertZero$ where $h$ is Morselike. To this end we define a second  space of perturbations $\Pert_f$. This consists of perturbations which are pull back of functions on the torus of reducibles (not necessarily Morse)  vanishing in a neighbourhood of the critical points of $f_\torus$ and are made constant in a neighbourhood of the critical points of $\cL +f $ with the help of a bump function.
Using the space of perturbations $\Pert_f\times \PertZero_\cO$ we manage to achieve generic regularity (Lemma~\ref{app:Lemma2MorseLike}).

\paragraph{Non-degeneracy} From \cite[pg. 212]{KM} it follows that a \emph{reducible} critical point $[(B,0,\psi)]\in\B_k^\sigma(Y,\sstruc)$ in the blow-up is non-degenerate iff and only if 
the perturbed Dirac operator $D_{B,\pert}$ is invertible and has simple spectrum over $\C$ and the operator $*d+\Differential_{(B,0)}\pert $ is invertible.
In particular, for a Morselike perturbation $\pert = \grad f$ the second condition is automatically satisfied thanks to $f_\torus$ being a Morse function.
\begin{lemma} \label{app:Lemma1MorseLike}The set of perturbations $\pert' \in  \Pert^\perp$ such that $(\operatorname{grad}(\cL+f) + \pert')^\sigma$ has  non-degenerate critical points (in the blow-up) is residual in $\Pert^\perp$.
\end{lemma}
\begin{proof}
	Firsly we take care of the reducibles critical points. 
	Let $Op^{sa}$ be the set of self-adjoint Fredholm maps defined in  \cite[pg. 215]{KM}. 
	Let $\Sigma \subset Op^{sa}$ be the subset of operators with non-simple spectrum union the set of non-injective operators.  $\Sigma$, is stratified by manifolds and the map 
	\begin{equation*}
	\pert\in \Pert^\perp \mapsto D_B + \Differential_{(B,0)}\pert^1(0,\cdot)\in Op^{sa}
	\end{equation*}
	is transverse to the stratification  by \cite[Lemma 12.6.2]{KM}. 
	Since each stratum has positive codimension and  the set of reducibles critical points in the blow-down,  $\{[(B,0)]\}$, is 0-dimensional it follows that  for generic $\pert'\in \Pert^\perp$ the perturbed Dirac operator will have simple spectrum and be invertible.
	
	Secondly, the set of such $\pert'$ such that all the \emph{irreducible} critical points are non-degenerate is residual. The proof of this fact follows the same course of the  proof of   \cite[Lemma 12.5.2]{KM}.
	Indeed $\Differential_\alpha \operatorname{grad}(\cL +f)$ has the same key properties of $\Differential_\alpha \operatorname{grad}(\cL)$ and  we may assume that the perturbation in Corollary 11.2.2 of \cite{KM}
	vanishes on the reducibles (hence can be approximated with elements of $\Pert^\perp$) because the critical point we are considering is irreducible.
\end{proof}

Now we want to prove the generic regularity of the moduli spaces of trajectories. 
We will introduce the first class of perturbations that we will use. These are similar to those used in the book (i.e. vanishing in a neighbouhood of the critical points) with the additional constraint that 
their  connection component vanishes over the reducibles.
\paragraph{Perturbations  $\PertZero_{\cO}$.}Now let $\pert'$ be a perturbation given by Lemma~\ref{app:Lemma1MorseLike} and let  $f'$ a primitive for it. Set
 \begin{equation*}
	\cL' = \cL + f + f'.
\end{equation*}
For each critical point $[\alpha]$ of $\cL'$ in the blow-down, we may find a gauge-invariant  open neighbourhood $\cO_{[\alpha]}\subset \B^o(Y,\sstruc)$ as defined at page  265 of \cite{KM}.
Here $\B^o(Y,\sstruc)$ is the based quotient configuration space (see pg 173 of \cite{KM}).
We then set $\cO :=  \bigsqcup_{[\alpha]} \cO_{[\alpha]} \subset \B^o(Y,\sstruc) $ where $[\alpha]$ varies on the set of critical points of $\cL'$.
Define  $\PertZero_{\cO}$as the closed subspace of  perturbations $\pert \in \PertZero$ vanishing over $\cO$.

\paragraph{Perturbations $\Pert_f$.}
The perturbations $\PertZero_\cO$ are not enough to prove regularity because in the case of irreducible trajectories between reducible critical points, we may need to use cylinder functions that are not constant on the reducibles. To remedy this,  we introduce another space of perturbations which will ensure us to get, in the end, a perturbation  which is still  Morselike  over the reducibles.

Let $\cO_\torus \subset\torus$ be a  union of  disjoint contractible open neighbourhoods of the \emph{reducible} critical points of $f|_\torus$. 
Withouth loss of generality we may assume that $\cO$ and $\cO_\torus$ are chosen in such a way   that for each $[\alpha]$ reducible critical point of $\cL'$, $\cO_{[\alpha]}\subset p_\torus^{-1}(\cO_\torus)$ where $p_\torus:\B^o(Y,\sstruc)\to \torus$ is a retraction to the torus of flat connections analogous to  \eqref{def:RetractionToTorus}. Up to shrinking $\cO_{[\alpha]}$, we may also assume that for  any $[\alpha]$  \emph{irreducible}, $\cO_{[\alpha]}$ does not intersect the reducible locus. 

\newcommand{\myp}{\mathcal{p}}
\newcommand{\myb}{b}
Now we consider a function $\myp: \B^o(Y,\sstruc)\to \R^n\times \torus\times \C^m$ with  $n, m \geq 0$
 and a collection of functions $\myb_{[\alpha]}:\R^n\times \torus\times \C^m\to \R$, indexed by $[\alpha]$ varying among \emph{irreducible} critical points,   such that
\begin{enumerate}
	\item $\myp$  is of the kind of functions appearing in the first item of \cite[Definition 11.1.1]{KM},
	\item $\myb_{[\alpha]}:\R^n\times \torus\times \C^m\to \R$ is of the kind of functions appearing in the second  item of \cite[Definition 11.1.1]{KM},
	\item setting  $\beta_{[\alpha]} :=\myb_{[\alpha]} \circ\myp$, it holds that 
	 $\beta_{[\alpha]}\equiv 0$ on $\cO_{[\alpha]}$ and  $\beta_{[\alpha]}\equiv 1$  outside of
  $ \cO'_{[\alpha]}$, some  slighly larger neighbourhoods of $[\alpha]$ with the same properties.
\end{enumerate}
It follows  that $\beta_{[\alpha]} $ is $\G(Y)$-invariant  cylinder function \cite[Definition 11.1.1]{KM}.
The existence of such functions relies on the assumption that the sets  $\cO_{[\alpha]}$ do not intersect the reducible locus. 

Next we choose a countable family $\{\tilde h_n\}_{n\in \N}$ of smooth functions $\tilde h_n:\torus\to  \R$ such that 
\begin{enumerate}
\item $\tilde h_n$ is \emph{constant} over $\cO_\torus$, 
\item  $\{\tilde h_n\}_{n\in \N}$ is dense  in the set of smooth functions $\torus\to  \R$ constant over $\cO_\torus$ endowed with the $C^\infty$ topology.
\end{enumerate}

\paragraph{}Define $\Pert_f$ to be the Banach space of tame perturbations \cite[Theorem 11.6.1]{KM} generated by the countable family of cylinder functions 
\begin{equation}\label{eqdefFunhPertf}
	   h_n(x) := 	\tilde{h}(p_\torus([x])) \sum_{[\alpha]}\beta_{[\alpha]}(x) = \tilde h_n(pr_2(\myp([x])))\sum_{[\alpha]} \myb_{[\alpha]}(\myp([x]))
\end{equation}
$n\in \N$, where $pr_2 : \R^n\times \torus \times \C^m\to \torus$ is the projection onto the second factor.

Notice that  this implies that  if $\pert \in \Pert_f$, then $\pert|_\cO \equiv 0$.  Indeed this is ensured by the functions $\beta_{[\alpha]}$ for $[\alpha]$ irreducible  and by the
assumption $\cO_{[\alpha]}\subset p_\torus^{-1}(\cO_\torus)$ together with the first condition  on $\tilde h_n$ for $[\alpha]$ reducible.

\begin{remark} As defined, $\Pert_f$ is not a subspace of $\Pert$ in general, however, once  $f$ and $f'$  are given we can form the large Banach space of tame perturbations $\Pert$ using a countable family of cylinder functions that includes $\{h_n\}_{n\in \N}$, in this way we have an inclusion $\Pert_f\subset \Pert$ as a closed subspace. This is the possible enlarging of $\Pert$ to which the statement of Proposition~\ref{prop:MorseLikePerturbationPerp} refers to.
\end{remark}

\begin{lemma} \label{app:Lemma2_0MorseLike} There exists a neighbourhood of zero $\mathcal N\subset \Pert_f\times  \PertZero_\cO$ such that 
if $(\pert_f, \pert'') \in \mathcal N$ and we denote by $h$ and  $f''$ a primitive of $\pert_f$ and $\pert''$ respectively, then 
the critical points of $\Lpert := \cL'+ f''+h $ are the same as the critical points of $\cL'$,  are non-degenerate and $(f+h)$ is Morselike in a neighbouhood of the reducible locus.

\end{lemma}
\begin{proof}
Since $(\pert_f + \pert'')|_{\cO} = 0$ all the critical points of $\cL'$ will persist the perturbation, moreover they will still be non-degenerate since $\cO$ is open.
Finally, the properness property of tame perturbations  \cite[Proposition 11.6.4]{KM} ensures that if the perturbation $ \pert_f + \pert'' $ is small enough, it does not introduce any new critical point in the complement 
of $\cO$. 

It remains to show that for such small perturbations $(f+h)$ is Morselike in a neighbouhood of the reducible locus.
Notice that by construction (see \eqref{eqdefFunhPertf}), $f+h$ restricted to a neighbourhood of the reducible locus is the pullback of a function defined over the torus of reducibles, thus we only need to check that  $f+h$ is Morse-Smale on $\torus$.
 We already know that it is Morse, because $h$ does not introduce new  critical points and vanishes in a neighbourhood of the critical points of $f$. To conclude, we observe that the property regarding the intersection of stable and unstable manifolds of $f+h$ is preserved if the  perturbations is small enough.
 \end{proof}

\paragraph{}
In the following  we let $ \mathcal{N}\subset \Pert_f\times  \PertZero_\cO$ be given by Lemma~\ref{app:Lemma2_0MorseLike}. 
\begin{lemma} \label{app:Lemma2MorseLike}
The subset of $ \mathcal{N}$ consisting of  $(\pert_f, \pert'') $  such that
 all the moduli spaces of trajectories of $\left(\grad(\cL')  + \pert_f + \pert''\right)^\sigma$ between any two critical points are regular is residual in $  \mathcal N$.
\end{lemma}
\begin{proof} Let $[\gota], [\gotb]\in\B^\sigma_k(Y,\sstruc)$ be critical points of $(\grad(\cL'))^\sigma$. 
Our proof  is a modification of the proof of \cite[Theorem 15.1.1.]{KM} which follows the standard approach of showing that the parametrized moduli spaces $\mathfrak{M}_z([\gota], [\gotb])\subset \B^\tau_k([\gota], [\gotb])\times \Pert_\mathcal{O}$ are regular (see  pg. 267 of the book). 
There are three cases to distinguish depending on $[\gota], [\gotb]$:
\begin{enumerate}
\item If $[\gota],[\gotb]$  are both reducibles and project to the same reducible in the blow-down moduli space, then the original proof goes through as  the regularity of  $M([\gota], [\gotb])$ follows from the non-degeneracy of the critical points.

\item  If $[\gota],[\gotb]$ are reducibles but do not project to the same reducible downstairs and the trajectory is reducible, then  the proof of Theorem 15.1.1. in \cite{KM}  boils down to ensuring that an element $\check V = (\omega, \psi)\in L^2_0(\R\times Y; iT^*Y\oplus S)$, orthogonal to the range of the map defining the parametrized moduli space,  is  zero (see pg 270-271 in \cite{KM}).
The proof that $\psi = 0$ actually uses only perturbations   that are gradient of functions which are constant on the reducible locus, thus belong to $\PertZero$ so we are left to showing that $\omega = 0$. Here our proof diverges from the book. In our case, $\omega = 0$ is ensured by the fact that over the reducibles, 
\begin{equation}
	(\grad(\cL')  + \pert_f + \pert'')^0 = (\grad(\cL')  + \pert_f)^0   =  (\grad\cL + \grad f + \grad h)^0,
\end{equation}
and $\cL + f + h$  is a Morselike perturbation. Hence  the operator $d^+ + \Differential_{\gamma(t)}{({ \operatorname{grad} (f+h)}}^0)$ acting  on $Y\times \R$ is surjective, and this  is the same operator shown in pg. 272 of \cite{KM}.

\item Suppose now that the trajectory considered is irreducible. Then reasoning as in the proof of Theorem 15.1.1 in \cite{KM}, we arrive at a situation where we have a curve $\underline{\gamma}$ downstairs, 
image of the trajectory $\gamma$, limiting to the critical points. Moreover, we have a never-vanishing vector field $\overline{V}$ along $\underline\gamma$, orthogonal to the gauge orbit and we want to find
a perturbation  $\mathfrak{p} = \pert_f + \pert''$ such that 
$\int_{\R}\langle \mathfrak{ p}, \overline{V}(t) \rangle_{\mathcal{T}_{0,\underline \gamma(t)}} dt > 0$.
Here is where we need to use $\Pert_f$.
Indeed if $\overline{V}$ is not tangent to $\gamma$ for some $t_0$ and  $[\underline \gamma(t_0)]\not \in \cO$, then we can find a cylinder function  with gradient vanishing outside a neighbourhood of $\underline{\gamma}(t_0)$ consisting of irreducibles  that satisfies the inequality above. In particular, this perturbation will be in $\PertZero_\mathcal{O}$.
However if $\overline{V}(t)$ is tangent to $\underline{\gamma}$ for all $t$ such that $[\underline \gamma(t)]\not \in \cO$, and in addition $[\gota],[\gotb]$ are reducibles, then we cannot find such a function. Indeed such a function, say $g$, must be constant on the reducibles, but along $\underline\gamma$, $g$  has to  be increasing, forcing  thus the values of $g$ at the endpoints to  be different. This is solved using a cylinder function with gradient $\pert_f \in \Pert_f$  which can be made increasing along $\underline{\gamma}$ because it is allowed to  take  different values at the two limiting reducibles.
\end{enumerate}
\end{proof}

 \section{Proof of Lemma~\ref{lemma:FromL2loctoL2}.}\label{app:ProofofLemmaL2locL2}
 	
		 Call $S$ the set  of pairs   satisfying \eqref{eq:monopoleCyl}. We will construct a map $M^{\mathrm{red}}(N^*; [\gotc])\to S/\G_{k+1, \delta}(N^*; 1)$. 
		Suppose $[(A,0, \R_+\Phi)] \in M^{\mathrm{red}}(N^*; [\gotc])$, so in particular $(A,\Phi) \in L^2_{k,loc}(N^*)$.
		Denote by $Z = [0,+\infty)\times Y \subset N^*$ the tube and define $\psi \in L^2_{k,loc}(Z, S^+)$ by
			 $\psi (t, y) = \Phi(t,y)/\norm{\Phi(t,\cdot)}_{L^2(Y)}$.
			
			$(A,0,\psi) \in \Conf^\tau_{k,loc}(Z)$ solves the perturbed flow equations  in \cite[Eq. (6.11)]{KM}.
			Since $[\gotc]$ is non-degenerate,  \cite[Proposition 13.6.1]{KM} ensures the existence of  ${\delta}>0$ (depending on $\gotc$) and  a gauge transformation $u\in \G_{k+1,loc}(Z)$ such that 
			\begin{equation}\label{eq:ODE_of_s0}
			(u\cdot A, 0, u\psi)-\gamma_\gotc \in e^{-{\delta} t} L^2_{k, B_\gotc}(Z),
			\end{equation}
			where $\gamma_\gotc\in \Conf^\tau_k(Z)$ 	represents the translation invariant solution induced by  $\gotc$.
			Since $H^1(N)\to H^1(Y)$ is an isomorphism, $u$ extends to $N^*$, so we can suppose that  $u\in \G_{k+1,loc}(N^*)$.

			Now we will show that up to rescaling $\Phi$,  $(u\cdot A, u\Phi) \in S$. Equation \eqref{eq:monopoleCyl} is clearly satisfied, what is not 
			obvious is that up to rescaling $u\Phi \in X^+_{k,\lambda_\gotc, \delta} (\psi_\gotc)$.
			Consider the function   $s(t) =  \norm{u\Phi(t)}_{L^2(Y)}$. Since $(uA,0,u\psi)$ satisfies the flow equations we have that 
			\begin{equation}
			\frac{d}{dt} s(t) = - \Lambda(t) s(t),
			\end{equation}
			where $ \Lambda (t) = \Re\langle D^+_{uA}u \psi, \rho_Z(dt)^{-1}u\psi \rangle_{L^2(\{t\}\times Y)} $. By 
			Lemma~\ref{app:lemma:LambdasSobolev} proved below,  $\Lambda(\cdot) -\lambda_\gotc \in L^2_{k-1, \delta}(\R_+)$ and by Lemma~\ref{LemmaDecayOfFunS}, $s(t) = e^{-\lambda_\gotc t} c + g(t)$ for some $c\in \R^\times, g\in L^2_{k, \lambda_\gotc + \delta}(\R_+)$. 
			It follows that, over the end
			\begin{spliteq}
			u\Phi(t,y) = s(t)u\psi(t,y) &=  s(t)\psi_\gotc(y) + s(t)(u\psi(t,y)- \psi_\gotc(y))\\
												& = e^{-\lambda t }c\psi_\gotc(y) + g(t)\psi_\gotc(y)  +s(t)(u\psi(t,y)- \psi_\gotc(y)).\
			\end{spliteq}
			Notice that since  $(u\psi(t,\cdot)- \psi_\gotc)\in L^2_{k,\delta}$ and $ s\in L^2_{k,\lambda_\gotc + \delta}$ their product lies in $L^2_{k,\lambda_\gotc +\delta}$. Consequently,  we obtain that $(u\cdot A, u/c\Phi)  \in A_\gotc +L^2_{k,\delta}(N^*,i\Lambda^1N^*)\times X^+_{k,\lambda_\gotc, \delta}(\psi_\gotc)$. 
			
			It can be showed as in the proof of \cite[Lemma 13.3.1]{KM} that any other gauge transformation $\tilde u\in L^2_{k+1, loc}(N^*)$
			such that \eqref{eq:ODE_of_s0} holds  satisfies $\tilde u^{-1} u \in \G_{k+1,\delta}(N^*;1)$, therefore we established an injection $M^{\mathrm{red}}(N^*; [\gotc])\to S/\G_{k+1, \delta}(N^*; 1)$.

			Now suppose that $[(A,\Phi)]\in S/\G_{k+1, \delta}(N^*; 1)$;  we want to prove that $[A,0, \Phi] \in M^{\mathrm{red}}(N^*; [\gotc])$.
			Equation \eqref{eq:monopoleCyl} implies that $(A,\Phi)$ satisfies the Seiberg-Witten equations, and clearly $L^2_{k}\subset L^2_{k,loc}$, therefore the only thing to check is that $\Phi$ has the correct asymptotics. This follows from 
			special properties of $ X^+_{k,\lambda_\gotc, \delta}(\psi_\gotc)$ as can be seen by a direct computation.
 	
 	\begin{lemma}\label{app:lemma:LambdasSobolev}	Let $Z= \R_+\times Y$.
	Let $(A,0, \psi) \in \Conf_{k,loc}^\tau(Z)$ and define $c\in L^2_k(\R_+)$  by $A(t) = \check{A}(t) + c(t)dt\otimes 1_S$.
	Let $(A_\gotb, 0, \psi_\gotb)\in \Conf_{k}^\sigma(Y)$ and denote by the same symbol $(A_\gotb, 0, \psi_\gotb)\in \Conf^\tau_{k,loc}(Z)$ the translation invariant 
	configuration induced by it.
	Suppose that  $(A-A_\gotb, \psi- \psi_\gotb) \in L^2_{k}(Z, iT^*Z) \times L^2_{k}(Z,S^+)$. 
	Then the map defined by 
	\begin{equation}\label{def:Difference_LambdaFun_Eigenvalue}
		\Lambda(t) = \left \langle D_{\check A(t), \pert} \psi(t) + c(t)dt\cdot \psi(t) ,  \psi(t) \right \rangle_{L^2(Y)} - \left \langle D_{\check A_\gotb, \pert} \psi_\gotb , \psi_\gotb \right \rangle_{L^2(Y)}
	\end{equation}
	is in $L^2_{k-1}(\R_+)$. 
			If in addition the data decays exponentially, i.e. $(A-A_\gotb, \psi- \psi_\gotb)\in  L^2_{k,\delta}$ for some, then  also  \eqref{def:Difference_LambdaFun_Eigenvalue} is in $ L^2_{k-1,\delta}$.
	\end{lemma}
	
	\begin{proof}	
	The hypothesis $(A-A_\gotb, \psi- \psi_\gotb) \in L^2_{k}(Z, iT^*Z) \times L^2_{k}(Z,S^+)$ implies that 
	\begin{equation}
		t\mapsto \psi(t)- \psi_\gotb 
	\end{equation}
	belongs to $L^2_{k-1}(\R_+, L^2_1(Y, S))$. Indeed the Sobolev spaces for Hilbert valued functions are defined by requiring that the weak derivatives (defined by integrating against $C_0^\infty(\R,\R)$) lie in $L^2$ and
	this is clearly seen to be implied by the hypothesis using the Pettis integral. 
	Similarly, the map $t\mapsto A(t)-A_\gotb$ belongs to $L^2_{k-1}(\R_+, L^2_1(Y; iT^*Y\oplus \R))$.
	Now
	\begin{spliteq}
	\Lambda(t) &= \left \langle D_{\check A(t), \pert} \psi(t) + c(t)dt\cdot \psi(t) ,  \psi(t) \right \rangle_{L^2(Y)} -\left \langle D_{\check A_\gotb, \pert} \psi_\gotb , \psi_\gotb \right \rangle_{L^2(Y)}\\
					& = \left \langle D_{\check A(t), \pert} \psi(t) + c(t)dt\cdot \psi(t) ,  \psi(t) \right \rangle_{L^2(Y)} - \left \langle D_{\check A_\gotb, \pert} \psi_\gotb , \psi(t) \right \rangle_{L^2(Y)}\\
					& \phantom{aaaaaaaaaaaaaaaaa} +\left \langle D_{\check A_\gotb, \pert} \psi_\gotb , \psi(t) \right \rangle_{L^2(Y)}-\left \langle D_{\check A_\gotb, \pert} \psi_\gotb , \psi_\gotb \right \rangle_{L^2(Y)}\\
					&= \left \langle D_{\check A(t), \pert} \psi(t) + c(t)dt\cdot \psi(t)-D_{\check A_\gotb, \pert} \psi_\gotb  ,  \psi(t) \right \rangle_{L^2(Y)} 
					+\left \langle D_{\check A_\gotb, \pert} \psi_\gotb , \psi (t)- \psi_\gotb \right \rangle_{L^2(Y)}.\\
	\end{spliteq}
	We will use the following lemma: if $X$ is a Banach space, $g\in L^2_r(\R, X)$, $r>1$ and $F:X\to \R$ is a \emph{ smooth} function such that $F(0) = 0$ then  $F\circ g \in L^2_r(\R)$. The proof of this fact
	for vector valued functions is analogous to the case when $X = \R$.
	The map $L^2(Y,S)\to \R$ defined by $\varphi \mapsto \left \langle D_{\check A_\gotb, \pert} \psi_\gotb , \varphi\right \rangle_{L^2(Y)}$ clearly satisfies these hypotheses, therefore 
	$t\mapsto \left \langle D_{\check A_\gotb, \pert} \psi_\gotb , \psi (t)- \psi_\gotb \right \rangle_{L^2(Y)}$ belongs to $L^2_{k-1}(\R_+, \R)$. 
	Now define  
					\begin{equation}
						Q(t) := D_{\check A(t), \pert} \psi(t) + c(t)dt\cdot \psi(t)-D_{\check A_\gotb, \pert} \psi_\gotb.
					\end{equation}
	More precisely: 
	\begin{equation}\label{eq:ProofLambda(t)Sobolev1}
		Q(t) = D_{A_\gotb} (\psi(t) -\psi_\gotb) + \rho_Y (A(t)-A_\gotb) \left(\psi(t)\right) + \Differential_{(\check{A}(t),0)}\pert^1 (\psi(t)) - \Differential_{(A_\gotb,0)}\pert^1 (\psi_\gotb).
	\end{equation}
	We can write 
	\begin{equation}
		\langle Q(t), \psi(t)\rangle_{L^2(Y)} = \langle  G(\check A(t)-A_\gotb, c(t),  \psi(t)-\psi_\gotb), \psi(t)\rangle_{L^2(Y)}
	\end{equation}	
	 where
	   $G: L^2(Y,iT^*Y\oplus \R)\oplus L^2(Y,S)\to L^2(Y,S)$ is defined as
	\begin{equation}
			G(a,c, \phi) = D_{A_\gotb}\phi + \rho_Y(a+c)\cdot (\phi+\psi_\gotb) + \Differential_{(a+A_\gotb,0)}\pert^1 (\phi+\psi_b) - \Differential_{(A_\gotb,0)}\pert^1 (\psi_\gotb).
	\end{equation}	
	$G$ is smooth and $G(0) = 0$ as is the map $(a,c,\phi)\mapsto \langle G(a,c,\phi), \phi+\psi_\gotb\rangle_{L^2(Y)}$.
	 Consequently $t\mapsto \langle Q(t), \psi(t)\rangle_{L^2(Y)}$ is in  $L_{k-1}^2(\R_+)$.
	To prove the statement about exponential decay we apply the same argument using Sobolev spaces with weights.
	 
 	\end{proof}

	\begin{lemma}\label{LemmaDecayOfFunS} Let $\lambda \in \R, \delta >0$ and $\Lambda\in  \lambda + L^2_{k-1, \delta}(\R_+)$ with $k>2$. Set $s(t) = \exp(-\int_0^t \Lambda(\tau) d\tau): \R_+ \to \R$. 
	Then $s(t) = e^{-\lambda t} c + g(t)$ where $c = e^{-\int_0^{+\infty} (\Lambda(\tau) -\lambda) d\tau} $ and $g\in L^2_{k, \lambda + \delta}(\R_+)$. In particular  $s \in L^2_{k, \lambda}(\R_+)$.
	\end{lemma}
	\begin{proof}
		Since $L^2_{1,\delta}(\R_+)\subset L^1(\R_+)$,   the integral $\int_0^{+\infty} (\Lambda(\tau) - \lambda) d\tau$ is finite and    $c\in \R^\times$ is well defined. 
		Now we have that 
		\begin{spliteq}\label{eq:explambdast}
			e^{\lambda t} s(t) - c &= e^{-\int_0^t (\Lambda(\tau) - \lambda) d\tau} - e^{-\int_0^{+\infty} (\Lambda(\tau) - \lambda) d\tau} \\
											&= e^{-\int_0^t (\Lambda(\tau) - \lambda) d\tau}( 1  -  e^{-\int_t^{+\infty} (\Lambda(\tau) - \lambda) d\tau}) \\
		\end{spliteq}
		Using Sobolev multiplication theorems \cite[Section 13.2]{KM} and that 
		$|\int_t^{+\infty} (\Lambda(\tau) - \lambda) dt| \leq \text{const.} e^{-\delta t}$, we find that
		\begin{spliteq}
		 & ( 1  -  e^{-\int_t^{+\infty} (\Lambda(\tau) - \lambda) d\tau})\in L^2_{k, \delta}\\
		 & e^{-\int_0^t (\Lambda(\tau) - \lambda) d\tau}\in L^\infty_2(\R_+) \\
		 & \partial_t (e^{-\int_0^t (\Lambda(\tau) - \lambda) d\tau})\in L^2_{k-1, \delta}(\R_+).\\
		\end{spliteq}
		Consequently, $e^{\lambda t} s(t) - c \in L^2_{k, \delta}(\R_+)$ by applying  Sobolev multiplication theorems  to  \eqref{eq:explambdast}.
		Therefore there exists $\tilde g \in L^2_{k, \delta}(\R_+)$, such that 
		$e^{\lambda t} s(t)   =c+  \tilde g(t) $.
	\end{proof}

\bibliographystyle{alpha}
\bibliography{pcorkPaper}

\paragraph{}
\textsc{Department of Mathematics, Imperial College London}\newline
\phantom{aa} E-mail address:  \texttt{r.ladu19@imperial.ac.uk}

\end{document}